\documentclass[11pt]{amsart}
\usepackage{amssymb}
\usepackage{mathtools}
\usepackage{color}

\newtheorem{theo}{Theorem} 
\newtheorem{maincorol}[theo]{Corollary} 

\newtheorem{lemma}{Lemma}[section]
\newtheorem{prop}[lemma]{Proposition}
\newtheorem{corol}[lemma]{Corollary}
\newtheorem{theoint}[lemma]{Theorem}
\newtheorem{claim}[lemma]{Claim}

\theoremstyle{remark}
\newtheorem{remark}[lemma]{Remark}

\theoremstyle{definition}

\newcommand{\lin}{\textsc{l}}
\newcommand{\NL}{\textsc{nl}}

\newcommand{\dd}{\mathsf{d}}

\newcommand{\NN}{\mathbb{N}}
\newcommand{\RR}{\mathbb{R}}

\newcommand{\eps}{\varepsilon}

\newcommand{\BBB}{\mathcal{B}}

\newcommand{\III}{\mathcal{I}}
\newcommand{\tIII}{\widetilde{\mathcal{I}}}
\newcommand{\JJJ}{\mathcal{J}}

\newcommand{\NNN}{\mathcal{N}}

\newcommand{\PPP}{\mathcal{P}}
\newcommand{\SSS}{\mathcal{S}}

\newcommand{\lf}{\left}
\newcommand{\rg}{\right}

\newcommand{\tr}{\tilde{r}}

\newcommand{\tlambda}{\widetilde{\lambda}}

\newcommand{\tw}{\widetilde{w}}
\newcommand{\ttw}{\breve{w}}
\newcommand{\tteps}{\breve{\eps}}
\newcommand{\ttau}{\widetilde{\tau}}

\newcommand{\tU}{\widetilde{U}}

\newcommand{\tu}{\widetilde{u}}
\newcommand{\ttu}{\widetilde{\widetilde{u}}}
\newcommand{\ttv}{\widetilde{\widetilde{v}}}

\newcommand{\tphi}{\widetilde{\varphi}}

\newcommand{\teps}{\widetilde{\eps}}

\newcommand{\Sbf}{\mathbf{S}}

\newcommand{\hdot}{\dot{H}^1}

\newcommand{\EMPH}[1]{\medskip\noindent\textit{#1}.}

\DeclareMathOperator{\rad}{rad}

\DeclareMathOperator{\supp}{supp}

\newcommand{\ds}{\displaystyle}

\numberwithin{equation}{section} 

\title[Bounded solutions for critical wave]{Profiles of bounded radial solutions of the focusing, energy-critical wave equation}
\author[T.~Duyckaerts]{Thomas Duyckaerts$^1$}
\author[C.~Kenig]{Carlos Kenig$^2$}
\author[F.~Merle]{Frank Merle$^3$}
\thanks{$^1$LAGA, Universit\'e Paris 13 (UMR 7539). Partially supported by ERC Grant Dispeq}
\thanks{$^2$University of Chicago. Partially supported by NSF Grant DMS-0968472}
\thanks{$^3$Cergy-Pontoise (UMR 8088), IHES}

\date{\today}

\begin{document}
\begin{abstract}
In this paper we consider global and non-global bounded radial
solutions of the focusing energy-critical wave equation in space
dimension $3$. We show that any of these solutions decouples, along a sequence of times that goes to the maximal
time of existence, as a sum of modulated stationary solutions, a free
radiation term and a term going to $0$ in the energy space. In the case
where there is only one stationary profile, we show that this expansion
holds asymptotically without restriction to a subsequence.
\end{abstract}

\maketitle

\section{Introduction} 
Consider the energy critical focusing wave equation in space dimension $3$:
\begin{equation}
\label{CP}
\left\{
\begin{aligned}
 &\partial_t^2 u-\Delta u-u^5=0\\
&u_{\restriction t=0}=u_0\in \hdot(\RR^3),\quad
\partial_tu_{\restriction t=0}=u_1\in L^2(\RR^3).
\end{aligned}
\right.
\end{equation} 
Equation \eqref{CP} is locally well-posed in $\hdot\times L^2$: for any initial data $(u_0,u_1)\in \hdot\times L^2$, there exists a unique solution $u$ defined on a maximal interval of definition $(T_-(u),T_+(u))=I_{\max}(u)$ such that $(u,\partial_t u)\in C^0(I_{\max},\hdot\times L^2)$ and $u\in L^8(J\times \RR^3)$ for any compact interval $J\subset I_{\max}(u)$. The energy of the solution:
\begin{equation}
\label{def_energy}
 E[u]=E(u(t),\partial_tu(t))=\frac{1}{2}\int_{\RR^3} |\nabla u(t)|^2\,dx+\frac{1}{2}\int_{\RR^3} |\partial_t u(t)|^2\,dx-\frac 16\int_{\RR^3}|u(t)|^6\,dx
\end{equation} 
is conserved. Note that the $\hdot\times L^2$ norm and the energy are invariant under the scaling which preserves the equation.
In this work we study radial solutions of \eqref{CP} that are \emph{bounded} in the energy space for positive times, i.e. satisfy 
\begin{equation}
\label{bound_intro}
 \sup_{t\in [0,T_+(u))}\|\nabla u(t)\|^2_{L^2}+\|\partial_t u(t)\|^2_{L^2}<\infty.
\end{equation} 
In the case $T_+=T_+(u)<+\infty$, we call these solutions type II blow-up solutions.
Recall that \eqref{CP} has an explicit stationary solution:
\begin{equation}
\label{defW}
W:=\frac{1}{\left(1+\frac{|x|^2}{3}\right)^{\frac{1}{2}}}.
\end{equation}

Recall that in \cite{KeMe08}, \cite{DuKeMe10P}, it is shown that if $u$ is radial and
$$ \sup_{t\in [0,T_+)} \|\nabla u(t)\|^2_{L^2}<\|\nabla W\|^2_{L^2},$$
then $T_+=+\infty$ and the solution scatters forward in time. The threshold $\|\nabla W\|^2_{L^2}$ is sharp. Indeed from \cite{KrScTa09}, for all $\eta_0>0$, there exists a type II radial blow-up solution such that 
\begin{equation}
 \label{small_bound}
\sup_{t\in [0,T_+)} \|\nabla u(t)\|^2_{L^2}\leq \|\nabla W\|^2_{L^2}+\eta_0.
\end{equation}
Moreover the solution is of the form:
\begin{equation}
 \label{u_one_prof}
u(t,x)\approx \frac{1}{\lambda(t)^{1/2}}W\left(\frac{x}{\lambda(t)}\right)+\eps^*(x),
\end{equation} 
where $\eps^*\in \hdot$ and $\lambda(t)=(T_+-t)^{1+\nu}$, $\nu>1/2$. In \cite{DuKeMe11a}, the converse result was established, i.e all type II radial solutions satisfying \eqref{small_bound} have the behavior given in \eqref{u_one_prof} for some $\lambda(t)=o(T_+-t)$. This type of result was extended to the non-radial case in \cite{DuKeMe10P}. The introduction of \cite{DuKeMe11a} gives background and references on related results for other equations, such as GKdV \cite{MaMe01,MaMe03}, mass-critical NLS \cite{MeRa04,MeRa05b} and for wave maps in \cite{ChTZ93,ShTZ97} and \cite{Struwe02,Struwe03}. Note that these wave map results do not need the smallness assumption \eqref{small_bound}. However, in that problem, the density of the conserved energy is positive, which, combined with finite speed of propagation, gives decrease of the integral of the density of energy on sections of inverted cones. The existence of such a Lyapunov functional greatly simplifies the problem. In this paper we address for the first time a situation in which the smallness condition \eqref{small_bound} is not assumed and no Lyapunov functional as described above seems to be available.

In our results, we distinguish between two cases. Let us start with the case $T_+<\infty$. If $\{a_n\}$ and $\{b_n\}$ are two sequences of positive numbers, we write $a_n\ll b_n$ when $a_n/b_n$ goes to $0$ as $n$ goes to infinity.
\begin{theo}[Blow-up type II solutions]
\label{T:bup_sum} 
Let $u$ be a \emph{radial} solution of \eqref{CP} such that $T_+<\infty$ and \eqref{bound_intro} holds. Then there exist $(v_0,v_1)\in \hdot\times L^2$, a sequence $t_n\to T_+=T_+(u)$, an integer $J_0\geq 1$, $J_0$ sequences $\left\{\lambda_{j,n}\right\}_{n\in\NN}$, $j=1\ldots J_0$ of positive numbers, $J_0$ signs $\iota_j\in\{\pm 1\}$ such that
\begin{equation}
\label{bup_devt}
\lim_{n\to \infty}\left\|
\left(u(t_n)-v_0-\sum_{j=1}^{J_0} \frac{\iota_j}{\lambda_{j,n}^{1/2}}W\left(\frac{\cdot}{\lambda_{j,n}}\right),\partial_t u(t_n)-v_1\right)\right\|_{\hdot\times L^2}=0,
\end{equation} 
with
\begin{equation}
\label{bup_lambda_jn}
\lambda_{1,n}\ll \lambda_{2,n} \ll \ldots \ll \lambda_{J_0,n} \ll (T_+-t_n).
\end{equation} 
Furthermore
\begin{equation}
\label{quant_energy}
\lim_{t\to T_+} \int_{|x|\leq T_+-t}\frac 12|\nabla u(t,x)|^2+\frac 12\left(\partial_t u(t,x)\right)^2-\frac 16(u(t,x))^6\,dx=J_0 E(W,0). 
\end{equation} 
\end{theo}
Note that, at least for one sequence of times, this gives a complete description of type II radial blow-up solutions as a sum of rescaled, decoupled ground states $W$. We expect that such a decomposition still holds for all sequences of times, but at the moment this remains open. However, the quantization result \eqref{quant_energy} is established here for all times. Note that the analog of Theorem \ref{T:bup_sum} for co-rotational wave maps is unknown and does not follow from \cite{Struwe02}, \cite{Struwe03}. It is unclear at this time if one can actually contruct solutions with the description given in Theorem \ref{T:bup_sum} with $J_0\geq 2$.
\begin{theo}
 \label{T:bup_one}
Assume that $J_0=1$ in the preceding theorem. Then there exist $\lambda(t)$, defined for $t<T_+$ close to $T_+$, and a sign $\iota_1\in \{\pm 1\}$ such that
\begin{equation}
\label{bup_devt_1}
\lim_{t\to T_+}\left\|
\left(u(t)-v_0-\frac{\iota_1}{\lambda^{1/2}(t)}W\left(\frac{\cdot}{\lambda(t)}\right),\partial_t u(t)-v_1\right)\right\|_{\hdot\times L^2}=0,
\end{equation} 
and
\begin{equation}
\label{bup_lambda}
 \lim_{t\to T_+} \frac{\lambda(t)}{T_+-t}=0.
\end{equation} 
\end{theo}
\begin{maincorol}
\label{C:bup}
Let $u$ be a radial type II blow-up solution of \eqref{CP} such that 
\begin{equation}
\label{bup_bound}
 \limsup_{t\in[0,T_+)} \int_{|x|\leq T_+-t}|\nabla u(t,x)|^2\,dx<2\|\nabla W\|_{L^2}^2.
\end{equation} 
Then $u$ satisfies the conclusion of Theorem \ref{T:bup_one}.
\end{maincorol}
The proof of Theorem \ref{T:bup_one} uses the description of small type II blow-up solutions given in \cite{DuKeMe11a}, and is also heavily dependent on properties of the solution $W^+$ constructed in \cite{DuMe08}. In particular, we establish here a rigidity result for $W^+$ (Theorem \ref{T:rigidityW+}), as well as the fact that $W^+$ is a type I blow-up solution (Theorem \ref{T:blowupW+}), facts which are of independent interest and central to our proof of Theorem \ref{T:bup_one}. As a consequence of Theorem \ref{T:rigidityW+} and the main result of \cite{KrNaSc10P}, we also complete the classification of solutions of \eqref{CP} at the energy threshold $E[u]=E[W]$ in the radial case (see Corollary \ref{C:classification}). 

We will also prove the analoguous results in the globally defined case. In this case the solution is, at least for a time sequence, a sum of rescaled $W$ and of a finite energy solution to the linear wave equation:
\begin{equation}
\label{lin_wave_intro}
 \partial_t^2v_{\lin}-\Delta v_{\lin}=0,\quad (v_{\lin},\partial_t v_{\lin})_{\restriction t=0}\in \hdot\times L^2.
\end{equation} 
\begin{theo}[Globally defined solutions]
\label{T:global_sum} 
Let $u$ be a \emph{radial} solution of \eqref{CP} such that $T_+=+\infty$ and \eqref{bound_intro} holds. Then there exists a solution $v_{\lin}$ of \eqref{lin_wave_intro}, a sequence $s_n\to +\infty$, an integer $J_0\geq 0$, $J_0$ sequences $\left\{\lambda_{j,n}\right\}_{n\in\NN}$, $j=1\ldots J_0$ of positive numbers, $J_0$ signs $\iota_j\in\{\pm 1\}$ such that
\begin{equation}
\label{global_devt}
\lim_{n\to \infty}\left\|
\left(u(s_n)-v_{\lin}(s_n)-\sum_{j=1}^{J_0} \frac{\iota_j}{\lambda_{j,n}^{1/2}}W\left(\frac{\cdot}{\lambda_{j,n}}\right),\partial_t u(s_n)-\partial_t v_{\lin}(s_n)\right)\right\|_{\hdot\times L^2}=0,
\end{equation} 
with
\begin{equation}
\label{global_lambda_jn}
\lambda_{1,n}\ll \lambda_{2,n} \ll \ldots \ll \lambda_{J_0,n}\ll s_n.
\end{equation} 
Furthermore, for all $A>0$, 
\begin{equation}
\lim_{s\to +\infty} \int_{|x|\leq s-A} \int \frac{1}{2}|\nabla u(t,x)|^2+\frac 12 (\partial_tu(t,x))^2-\frac 16 (u(t,x))^6\,dx
=E_A 
\end{equation} 
exists and 
$$ \lim_{A\to +\infty} E_A=J_0 E(W,0).$$
\end{theo}
\begin{theo}
 \label{T:global_one}
Assume that $J_0=1$ in the preceding theorem. Then there exists $\lambda(t)$, defined for large $t>0$, and a sign $\iota_1\in \{\pm 1\}$ such that
\begin{equation}
\label{global_devt_1}
\lim_{t\to +\infty}\left\|
\left(u(t)-v_{\lin}(t)-\frac{\iota_1}{\lambda^{1/2}(t)}W\left(\frac{\cdot}{\lambda(t)}\right),\partial_t u(t)-\partial_tv_{\lin}(t)\right)\right\|_{\hdot\times L^2}=0,
\end{equation} 
and
\begin{equation}
\label{global_lambda}
 \lim_{t\to +\infty} \frac{\lambda(t)}{t}=0.
\end{equation} 
\end{theo}
\begin{maincorol}
\label{C:global}
Let $u$ be a bounded radial solution of \eqref{CP} with $T_+(u)=+\infty$. Assume that 
there exists $A\in \RR$ such that
\begin{equation}
\label{global_bound}
 \limsup_{t\to +\infty} \int_{|x|\leq t-A} |\nabla u(t,x)|^2\,dx<2\|\nabla W\|_{L^2}^2.
\end{equation} 
Then either $u$ scatters or satisfies the conclusion of Theorem \ref{T:global_one}.
\end{maincorol}
\begin{remark}
If $u$ is a radial solution satisfying \eqref{bound_intro}, such that $T_+(u)=+\infty$ and $E(u_0,u_1)<2E(W,0)$ then either $u$ scatters or satisfies the conclusion of Theorem \ref{T:global_one}.
\end{remark}
Theorem \ref{T:global_sum}, Theorem \ref{T:global_one} and Corollary \ref{C:global} are the analogs for bounded, globally defined solutions which do not scatter of the results for type II blow-up solutions. The proofs are essentially the same. Note that the existence of solutions as in Theorem \ref{T:global_one} follows from \cite{KrSc07} for $\lambda(t)=1$ and \cite{DoKr12P} for $\lambda(t)\approx t^{\alpha}$, $\alpha\in \RR$ close to $0$. 
\begin{remark}
 In theory, one might expect, in the case $J_0\geq 2$ that the conclusion of Theorems \ref{T:bup_sum} and \ref{T:global_sum} only holds for some sequence of times. However, we conjecture that it holds for all sequences converging to $T_+(u)$, i.e that there exist positive $\lambda_1(t),\ldots,\lambda_{J_0}(t)$ such that
\begin{equation}
\tag{\ref{bup_devt}'}
\label{bup_devt'}
\lim_{t\to T^+(u)}\left\|
\left(u(t)-v_0-\sum_{j=1}^{J_0} \frac{\iota_j}{\lambda_{j}^{1/2}(t)}W\left(\frac{\cdot}{\lambda_{j}(t)}\right),\partial_t u(t)-v_1\right)\right\|_{\hdot\times L^2}=0,
\end{equation} 
with $\lambda_1(t)\ll \ldots \ll\lambda_{J_0}(t)\ll T_+(u)-t$ as $t\to T_+(u)$
if $T_+(u)<\infty$ and 
\begin{equation}
\tag{\ref{global_devt}'}
\label{global_devt'}
\lim_{t\to \infty}\left\|
\left(u(t)-v_{\lin}(t)-\sum_{j=1}^{J_0} \frac{\iota_j}{\lambda_{j}^{1/2}(t)}W\left(\frac{\cdot}{\lambda_{j}(t)}\right),\partial_t u(t)-\partial_t v_{\lin}(t)\right)\right\|_{\hdot\times L^2}=0,
\end{equation} 
with $\lambda_1(t)\ll \ldots \ll\lambda_{J_0}(t)\ll t$ as $t\to +\infty$
if $T_+(u)=+\infty$. 

This is exactly the soliton resolution conjecture for bounded solutions of \eqref{CP} in the radial case. The only equation where this type of result is known so far is the KdV equation in the integrable case, see \cite{EcSc83}, \cite{Eckhaus86}.
\end{remark}
\begin{remark}
In the recent work of Krieger, Nakanishi and Schlag \cite{KrNaSc10P}, it is proved that there exists a small $\eta_0>0$ such that for any radial solution $u$ of \eqref{CP} with $E(u_0,u_1)<E(W,0)+\eta_0$, one of the following holds:
\begin{enumerate}
 \item \label{I:Bup} $T_+(u)<\infty$;
\item \label{I:scat} $u$ scatters forward in time to a linear solution;
\item \label{I:W} $T_+(u)=+\infty$ and $(u,\partial_tu)(t)$
is for large positive $t$ in a small $\hdot\times L^2$-neighborhood of the manifold 
$$\SSS=\left\{\left(\frac{\iota}{\lambda^{1/2}}W\left(\frac{\cdot}{\lambda}\right),0\right),\quad \lambda>0,\; \iota=\pm 1\right\}.$$
\end{enumerate}
According to Corollary \ref{C:global} the solutions in case \eqref{I:W} are exactly of the form given by the conclusion of Theorem \ref{T:global_one}.
\end{remark}
The key new ingredient for the proof of Theorem \ref{T:bup_sum} is a dispersion property for a compactly supported solution of \eqref{CP} (see \S \ref{SSS:Case1}), without any smallness assumption on the solution.
Using this property together with the profile decomposition of \cite{BaGe99}, we show in Section \ref{S:SSregion} that no energy concentrates in the self-similar region, i.e: $\int_{\eps_0(T_+-t)\leq |x|\leq T_+-t} |\nabla u(t)|^2+(\partial_tu(t))^2\,dx\to 0$ as $t\to T_+$. This is then combined in Section \ref{S:expansion} with 
arguments as in \cite{DuKeMe11a}, \cite{DuKeMe10P} to give Theorem \ref{T:bup_sum}.
Section \ref{S:one_profile} is dedicated to the proof of Theorems \ref{T:bup_one} and \ref{T:global_one}, and Section \ref{S:W+} to some properties of the solution $W^+$ constructed in \cite{DuMe08}. 

\subsection*{Notations}
We will denote by $o_n(1)$ a sequence real numbers that goes to $0$ as $n\to \infty$. If $\{\lambda_n\}_n$ and $\{\mu_n\}_n$ are two sequences of positive numbers, $\lambda_n\ll \mu_n$ means that $\lambda_n/\mu_n=o_n(1)$. We will denote
$$ B_{\rho}=\left\{y\in \RR^3,\text{ s.t. }|y|\leq \rho\right\},$$
and by $S_{\lin}(t)$ the linear wave propagator: 
$$v_{\lin}(t)=S_{\lin}(t)(v_0,v_1)=\cos(t\sqrt{-\Delta})v_0+\frac{\sin(t\sqrt{-\Delta})}{\sqrt{-\Delta}}v_1 $$
 is the solution of \eqref{lin_wave_intro} with initial data $(v_0,v_1)$ at $t=0$. We will also denote by $\Sbf_{\lin}(t)(v_0,v_1)$ the vector $(v_{\lin}(t),\partial_tv_{\lin}(t))$.

\section{Preliminaries}
\label{S:preliminaries}
\subsection{Profile decomposition}
We recall here the profile decomposition of H.~Bahouri and P.~G\'erard \cite{BaGe99}. Recall that such decompositions first appeared in the elliptic case in \cite{BrCo85} and \cite{Lions85b} and simultanuously to \cite{BaGe99} in \cite{MeVe98} for the Schr\"odinger case.

Consider a radial sequence $(u_{0,n},u_{1,n})_n$ which is bounded in $\hdot\times L^2$. 
Let $(U^j_{\lin})_{j\geq 1}$ be a sequence of radial solutions of the linear equation 
\begin{equation}
\label{lin_wave}
\partial_t^2v-\Delta v=0,\quad v_{\restriction t=0}=v_0,\; \partial_tv_{\restriction t=0}=v_1
\end{equation} 
with initial data $(U^j_0,U^j_1)\in \hdot\times L^2$, and $(\lambda_{j,n};t_{j,n})\in (0,+\infty)\times \RR$, $j,n\geq 1$, be a family of parameters satisfying the pseudo-orthogonality relation
\begin{equation}
\label{ortho_param}
j\neq k\Longrightarrow \lim_{n\to \infty} \frac{\lambda_{j,n}}{\lambda_{k,n}}+\frac{\lambda_{k,n}}{\lambda_{j,n}}+\frac{|t_{j,n}-t_{k,n}|}{\lambda_{j,n}}=+\infty.
\end{equation} 
We say that $(u_{0,n},u_{1,n})_n$ admits a profile decomposition  $\lf\{U_{\lin}^j\rg\}_j$, $\lf\{\lambda_{j,n},t_{j,n}\rg\}_{j,n}$ when $(w_{0,n}^J,w_{1,n}^J)$ defined by
\begin{equation}
\label{decompo_profil}
\left\{\begin{aligned}
w_{0,n}^J(x)&=u_{0,n}-\sum_{j=1}^J \frac{1}{\lambda_{j,n}^{\frac{1}{2}}}U_{\lin}^j\left(\frac{-t_{j,n}}{\lambda_{j,n}},\frac{x}{\lambda_{j,n}}\right),\\
w_{1,n}^J(x)&=u_{1,n}-\sum_{j=1}^J \frac{1}{\lambda_{j,n}^{\frac{3}{2}}}\partial_t U_{\lin}^j\left(\frac{-t_{j,n}}{\lambda_{j,n}},\frac{x}{\lambda_{j,n}}\right),
\end{aligned}\right.
\end{equation}
satisfies
\begin{equation}
\label{small_w}
\lim_{n\rightarrow+\infty}\limsup_{J\rightarrow+\infty} \left\|w_{n}^J\right\|_{L^8(\RR^4_{t,x})}=0,
\end{equation}
where $w_{n}^J$ is the solution of \eqref{lin_wave} with initial data $(w_{0,n}^J,w_{1,n}^J)$. 
By \cite{BaGe99}, if the sequence $(u_{0,n},u_{1,n})_n$ is bounded in the energy space $\hdot\times L^2$, there always exists a subsequence of $(u_{0,n},u_{1,n})_n$ which admits a profile decomposition. Furthermore, 
\begin{equation}
\label{weak_CV_wJ}
j\leq J\Longrightarrow 
\left(\lambda_{j,n}^{\frac{1}{2}} w_n^J\left(t_{j,n},\lambda_{j,n}y\right),\lambda_{j,n}^{\frac{3}{2}} \partial_tw_n^J\left(t_{j,n},\lambda_{j,n}y\right)\right)\xrightharpoonup[n\to \infty]{}0,
\end{equation}
weakly in $\hdot_y\times L^2_y$. In other words, the initial data $(U^j_0,U^j_1)$ of the profiles may be defined as weak limits, in $\hdot\times L^2$, of sequences $\left\{\lambda_n^{1/2}v_n(t_n,\lambda_n \cdot),\lambda_n^{3/2}\partial_tv_n(t_n,\lambda_n \cdot)\right\}$, where $v_n$ is the solution of \eqref{lin_wave} with initial data $(u_{0,n},u_{1,n})$, $\left\{\lambda_n\right\}_n$, $\left\{t_n\right\}_n$ are sequences in $(0,\infty)$ and $\RR$ respectively.

The following expansions hold for all $J\geq 1$:
\begin{gather}
\label{pythagore1a} 
\left\|u_{0,n}\rg\|_{\hdot}^2=\sum_{j=1}^J \left\|U^j_{\lin}\left(\frac{-t_{j,n}} {\lambda_{j,n}}\right)\rg\|_{\hdot}^2+\left\|w_{0,n}^J\rg\|_{\hdot}^2+o_n(1)\\
\label{pythagore1b} 
\lf\|u_{1,n}\right\|^2_{L^2}=\sum_{j=1}^J \lf\|\partial_t U^j_{\lin}\left(\frac{-t_{j,n}} {\lambda_{j,n}}\right)\right\|^2_{L^2}+\lf\|w_{1,n}^J\right\|^2_{L^2}+o_n(1)\\
\label{pythagore2}
E(v_{0,n},v_{1,n})=\sum_{j=1}^J E\left(U^j_{\lin}\left(-\frac{t_{j,n}}{\lambda_{j,n}}\right),\partial_t U^j_{\lin}\left(-\frac{t_{j,n}} {\lambda_{j,n}}\right)\right)+E\left(w_{0,n}^J, w_{1,n}^J\right)+o_n(1).
\end{gather}

Fixing $j$ and replacing $U^j_{\lin}(t,x)$ by  $V^j_{\lin}(t,x)=\frac{1}{\lambda_{j}^{\frac{1}{2}}}U_{\lin}^j\left(\frac{t-t_{j}}{\lambda_{j}},\frac{x}{\lambda_{j}}\right)$ for some good choice of the parameters $\lambda_j,\,t_j$, and extracting subsequences, we will always assume that one of the following two cases occurs
\begin{equation}
\label{choice_param}
\forall n,\;t_{j,n}=0\quad\text{or}\quad\lim_{n\to\infty} \frac{t_{j,n}}{\lambda_{j,n}}\in \{-\infty,+\infty\}. 
\end{equation} 
Furthermore, we will also assume that for all $j$, the sequences $\left\{t_{j,n}\right\}_n$ and $\left\{\lambda_{j,n}\right\}_n$
have limits in $\RR\cup\{\pm \infty\}$ and $[0,+\infty]$ respectively.

For any profile decomposition with profiles $\left\{U^j_{\lin}\right\}$ and parameters $\left\{\lambda_{j,n},t_{j,n}\right\}$, we will denote  by $\left\{U^j\right\}$ the non-linear profiles associated with $\left\{U^j_{\lin}\left(\frac{-t_{j,n}}{\lambda_{j,n}}\right),\partial_t U^j_{\lin}\left(\frac{-t_{j,n}}{\lambda_{j,n}}\right)\right\}$, which are the unique solutions of \eqref{CP} such that for all $n$, $\frac{-t_{j,n}}{\lambda_{j,n}}\in I_{\max}\left(U^j\right)$ and
$$ \lim_{n\rightarrow +\infty} \left\|U^j\left(\frac{-t_{j,n}}{\lambda_{j,n}}\right)-U^j_{\lin}\left(\frac{-t_{j,n}}{\lambda_{j,n}}\right)\right\|_{\hdot}
+\left\|\partial_t U^j\left(\frac{-t_{j,n}}{\lambda_{j,n}}\right)-\partial_t U^j_{\lin}\left(\frac{-t_{j,n}}{\lambda_{j,n}}\right)\right\|_{L^2}=0.$$

By the Strichartz inequalities on the linear problem and the small data Cauchy theory, if 
$\lim_{n\rightarrow +\infty} \frac{-t_{j,n}}{\lambda_{j,n}}=+\infty$, then 
$T_+\left(U^j\right)=+\infty$ and
\begin{equation}
\label{finite_norm2}
 s_0>T_-\left(U^j\right) \Longrightarrow \|U^j\|_{L^8(s_0,+\infty)}<\infty.
\end{equation} 
An analoguous statement holds in the case $\lim_{n\rightarrow +\infty} \frac{t_{j,n}}{\lambda_{j,n}}=+\infty$.
We will often write, for the sake of simplicity
\begin{equation}
\label{simplicity} 
U^j_n(t,x)=\frac{1}{\lambda_{j,n}^{1/2}}U^j\left(\frac{t-t_{j,n}}{\lambda_{j,n}},\frac{x}{\lambda_{j,n}}\right),\quad
U^j_{\lin,n}(t,x)=\frac{1}{\lambda_{j,n}^{1/2}}U^j_{\lin}\left(\frac{t-t_{j,n}}{\lambda_{j,n}},\frac{x}{\lambda_{j,n}}\right),
\end{equation}
and
\begin{equation}
 \label{simplicity_bis}
\left(U^{j}_{0,n},U^j_{1,n}\right)=\left(U^j_n,\partial_tU^j_n\right)_{\restriction t=0},\quad \left(U^{j}_{0,\lin,n},U^j_{1,\lin,n}\right)=\left(U^j_{\lin,n},\partial_tU^j_{\lin,n}\right)_{\restriction t=0}
\end{equation}
We will need the following approximation result, which follows from a long time perturbation argument. See the Main Theorem p. 135 in \cite{BaGe99}, and a sketch of proof right after Proposition 2.8 in \cite{DuKeMe11a}.
\begin{prop}
\label{P:lin_NL}
 Let $\{(u_{0,n},u_{1,n})\}_n$ be a bounded sequence in $\hdot\times L^2$ admitting a profile decomposition  with profiles $\{U^j_{\lin}\}$ and parameters $\{t_{j,n},\lambda_{j,n}\}$. Let $\theta_n\in (0,+\infty)$. Assume 
\begin{equation}
\label{bounded_strichartz}
\forall j\geq 1, \quad\forall n,\;\frac{\theta_n-t_{j,n}}{\lambda_{j,n}}<T_+(U^j)\text{ and } \limsup_{n\rightarrow +\infty} \left\|U^j\right\|_{L^8\big(\big(-\frac{t_{j,n}}{\lambda_{j,n}},\frac{\theta_n-t_{j,n}}{\lambda_{j,n}}\big)\times\RR^3\big)}<\infty.
\end{equation}
Let $u_n$ be the solution of \eqref{CP} with initial data $(u_{0,n},u_{1,n})$.
Then for large $n$, $u_n$ is defined on $[0,\theta_n)$,
\begin{equation}
\label{NL_bound}
\limsup_{n\rightarrow +\infty}\|u_n\|_{L^8\big((0,\theta_n)\times \RR^3\big)}<\infty,
\end{equation} 
and
\begin{equation}
\label{NL_profile} 
\forall t\in [0,\theta_n),\quad
u_n(t,x)=\sum_{j=1}^J U^j_n\left(t,x\right)+w^J_{n}(t,x)+r^J_n(t,x),
 \end{equation}
where 
\begin{equation}
\label{cond_rJn}
\lim_{J\rightarrow +\infty}\left[ \limsup_{n\rightarrow +\infty} \|r^J_n\|_{L^8\big((0,\theta_n)\times\RR^3\big)}+\sup_{t\in (0,\theta_n)} \left(\|\nabla r^J_n(t)\|_{L^2}+\|\partial_t r^J_n(t)\|_{L^2}\right)\right]=0.
\end{equation}
An analoguous statement holds if $\theta_n<0$.
\end{prop}
\begin{remark}
\label{R:lin_NL}
Assume that for all $j$, at least one of the following occurs:
\begin{enumerate}
 \item \label{small_profile}
$\ds \|U^j_0\|_{\hdot}^2+\|U^j_1\|_{L^2}^2\leq \frac{2}{3}\|W\|^2_{\hdot}$,
\item \label{dispersive_profile}$\ds \lim_{n\rightarrow+\infty} \frac{-t_{j,n}}{\lambda_{j,n}}=+\infty$,
\item \label{finite_limit_profile}
$\ds \limsup_{n\rightarrow +\infty} \frac{\theta_n-t_{j,n}}{\lambda_{j,n}}<T_+\left(U^j\right).$
\end{enumerate}
Then \eqref{bounded_strichartz} holds. 
\end{remark}

\subsection{Energy trapping}
The following Claim is shown in \cite{DuKeMe11a} by variational arguments:
\begin{claim}
\label{C:variational}
Let $f\in \hdot(\RR^3)$. Then
\begin{equation}
 \label{variational}
\|\nabla f\|_{L^2}^2\leq \|\nabla W\|_{L^2}^2\text{ and }E(f,0)\leq E(W,0)\Longrightarrow \|\nabla f\|_{L^2}^2\leq \frac{\|\nabla W\|_{L^2}^2}{E(W,0)}E(f,0)=3 E(f,0).
\end{equation}
Furthermore, if $\|\nabla f\|_{L^2}^2\leq \sqrt{3}\|\nabla W\|_{L^2}^2$, then $E(f,0)\geq 0$.
\end{claim}
\subsection{Linear behavior}
In this subsection we give two localization properties for solutions to the linear equation \eqref{lin_wave}. Lemma \ref{L:lin_odd} can be deduced from the strong Huygens principle, Proposition \ref{P:linear_behavior} can be shown using explicit formulas for the solutions of \eqref{lin_wave}. We refer to \cite{DuKeMe11a}, Lemma 4.1 and 4.2, and \cite{DuKeMe10P}, Proposition 2.8 for the proofs.
\begin{lemma}
\label{L:lin_odd}
Let $v$ be a solution of the linear wave equation \eqref{lin_wave}, with initial data $(v_0,v_1)$, $\lf\{\lambda_n\rg\}_n$, $\lf\{t_n\rg\}_n$ be two real sequences, with $\lambda_n$ positive. Let
$$ v_n(t,x)=\frac{1}{\lambda_n^{1/2}}v\lf(\frac{t}{\lambda_n},\frac{x}{\lambda_n}\rg)$$
and assume $\lim_{n\to \infty} \frac{t_n}{\lambda_n}=\ell\in [-\infty,+\infty]$.
Then, if $\ell=\pm \infty$,
$$  \lim_{R\to \infty} \limsup_{n\to \infty} \int_{\big||x|-|t_n|\big|\geq R\lambda_n}  |\nabla v_n(t_n)|^2+\frac{1}{|x|^2}|v_n(t_n)|^2+\lf(\partial_t v_n(t_n)\rg)^2dx=0$$
and if $\ell\in \RR$,
$$  \lim_{R\to \infty} \limsup_{n\to \infty} \int_{\substack{\{|x|\geq R \lambda_n\}\\ \cup\{|x|\leq \frac{1}{R}\lambda_n\}}} |\nabla v_n(t_n)|^2+\frac{1}{|x|^2}|v_n(t_n)|^2+\lf(\partial_t v_n(t_n)\rg)^2dx=0.$$
\end{lemma}
\begin{prop}
\label{P:linear_behavior}
Let $v$ be a radial solution of \eqref{lin_wave}, $t_0\in \RR$, $0<r_0<r_1$.
Then the following property holds for all $t\geq t_0$ or for all $t\leq t_0$
\begin{multline}
\label{pro_linear}
\int_{r_0+|t-t_0|<r<r_1+|t-t_0|} \left(\partial_r \big(r v(t,x)\big)\right)^2 +r^2(\partial_t v(t,x))^2dr\\
\geq \frac 12\int_{r_0<r<r_1} \left(\partial_r \big(r v(t_0,x)\big)\right)^2+r^2(\partial_t v(t_0,x))^2dr.
\end{multline} 
Furthermore, the following property holds for all $t\geq t_0$ or for all $t\leq t_0$:
\begin{equation}
\label{out_positive}
\int_{|x|\geq |t-t_0|} \lf(\lf|\nabla v(t,x)\rg|^2+ \lf(\partial_t v(t,x)\rg)^2\rg)dx\geq \frac 12 \int \lf(\lf|\nabla v(t_0,x)\rg|^2+ \lf(\partial_t v(t_0,x)\rg)^2\right)dx
\end{equation}
\end{prop}

\section{Self similar region}
\label{S:SSregion}
In this section we show that no energy of the singular part of the solution can concentrate in the self-similar region ($|x|\approx T_+-t$ in the finite blow-up case and $|x|\approx t$ in the global case). We treat the finite time blow-up case in Subsection \ref{SS:self_simF} and the global case in Subsection \ref{SS:self_simG}.
\subsection{Finite time blow-up case}
\label{SS:self_simF}
Let $u$ be a \emph{radial} solution of \eqref{CP} that satisfies \eqref{bound_intro}. Assume to fix ideas that $T_+(u)=1$. By \cite[Section 3]{DuKeMe11a}, there exists $(v_0,v_1)\in \hdot\times L^2$ such that
\begin{equation*}
 (u(t),\partial_tu(t)) \xrightharpoonup[t\to 1]{}(v_0,v_1)
\end{equation*} 
weakly in $\hdot\times L^2$. Furthermore, if $v$ is the solution of \eqref{CP} with initial data $(v_0,v_1)$ at time $t=1$,  then $a=u-v$ is well-defined on $[t_-,1)$ (where $t_->\max(T_-(u),T_-(v))$) and supported in the cone
$$ \left\{(t,x)\in \RR\times \RR^3\text{ s.t. } t_-\leq t \leq 1\text{ and }|x|\leq 1-t\right\}.$$
We call $v$ the \emph{regular} part of $u$ at the blow-up point and $a$ its \emph{singular} part. In this section, we show the following:
\begin{prop}
\label{P:no_ss}
Let $u$ be as above. Then 
\begin{equation*}
\forall c_0\in(0,1),\quad \lim_{t\to 1} \int_{c_0(1-t)\leq |x|\leq 1-t} \Big(|\nabla a(t,x)|^2+(\partial_t a(t,x))^2\Big)\,dx=0.
\end{equation*}
\end{prop}
In the remainder of this Subsection we prove Proposition \ref{P:no_ss}. We start by renormalizing the solution (\S \ref{SSS:renormalization}), then introduce a profile decomposition (\S \ref{SSS:profile}). The core of the proof is divided into three cases (\S \ref{SSS:Case1}, \S \ref{SSS:Case2}, \S \ref{SSS:Case3}).

\subsubsection{Renormalization}
\label{SSS:renormalization}
Assume without loss of generality that $t_-=0$.

We will get a contradiction assuming that there exists
a sequence $\{t_n\}_n$ in $(0,1)$ and $c_0>0$ such that
\begin{gather}
\label{limtn}
\forall n,\; t_n<1\text{ and } \lim_{n\to\infty} t_n=1\\
\label{ss_energy}
\int_{c_0(1-t_n)\leq |x|\leq 1-t_n}\Big[|\nabla a(t_n,x)|^2+(\partial_t a(t_n,x))^2\Big]\,dx\geq c_0.
\end{gather}
Let $T\in (0,1)$. Let 
\begin{equation}
 \label{def_In}
I_n=\left[\frac{T-t_n}{1-t_n},1\right),
\end{equation} 
and note that for large $n$, $[0,1)\subset I_n$.

Define, for $\tau\in I_n$,  
\begin{align*}
u_n(\tau,y)&=(1-t_n)^{1/2}u\left((1-t_n)\tau+t_n,(1-t_n)y\right)\\
v_n(\tau,y)&=(1-t_n)^{1/2}v\left((1-t_n)\tau+t_n,(1-t_n)y\right).
\end{align*}
Note that
\begin{equation}
\label{localization}
\supp (u_n(\tau)-v_n(\tau),\partial_tu_n(\tau)-\partial_t v_n(\tau))\subset B_{1-\tau}.
\end{equation} 

\subsubsection{Profile decomposition, properties of the profiles}
\label{SSS:profile}
Consider, after extraction of a subsequence, a profile decomposition $\left\{U^{j}_{\lin};\tau_{j,n};\lambda_{j,n}\right\}$ for the sequence $\Big\{\big(u_n(0)-v_n(0),\partial_{\tau}u_n(0)-\partial_{\tau}v_n(0)\big)\Big\}_n$:
\begin{align}
\label{expansion0}
 u_n(0,y)&=v_n(0,y)+\sum_{j=1}^J U^j_{0,\lin,n}(y)+w_{0,n}^J(y)\\
\label{expansion1} 
\partial_{\tau}u_n(0,y)&=\partial_{\tau}v_n(0,y)+\sum_{j=1}^J U^j_{1,\lin,n}(y)+w_{1,n}^J(y)
\end{align}
where 
$$ U^j_{\lin,n}(\tau,y)=\frac{1}{\lambda_{j,n}^{1/2}}U^{j}_{\lin}\left(\frac{\tau-\tau_{j,n}}{\lambda_{j,n}},\frac{y}{\lambda_{j,n}}\right),$$
and $(U^j_{0,\lin,n},U^{j}_{1,\lin,n})$ is the initial data for $U^j_{\lin,n}$.

As usual, we will assume that for all $j$,
$$(\forall n,\; \tau_{j,n}=0)\text{ or } \lim_{n\to \infty}\frac{\tau_{j,n}}{\lambda_{j,n}}=\pm\infty.$$

Let us show:
\begin{claim}
\label{C:profiles}
 Fix $j\geq 1$. Then (after extraction of subsequences in $n$),
\begin{equation}
\label{lim_tau_jn}
 \lim_{n\to \infty}\tau_{j,n}\in[-1,1],
\end{equation} 
and the sequence $\{\lambda_{j,n}\}_n$ is bounded.
\end{claim}
\begin{proof}
Claim \ref{C:profiles} follows from
\begin{equation}
\label{support_1}
 \supp\Big((u_n-v_n)(0),\partial_{\tau}(u_n-v_n)(0)\Big)\subset B_1,
\end{equation} 
and arguments as in \cite{BaGe99} (see \cite[Lemma 2.5]{DuKeMe11a}). 

Note that by \cite[Lemma 2.5]{DuKeMe11a}, the sequences $\big\{|\tau_{j,n}|\big\}_n$ and $\big\{\lambda_{j,n}\big\}$ are bounded for all $n$. It remains to show \eqref{lim_tau_jn}.

Fix $j\geq 1$. If $\tau_{j,n}=0$ for all $n$, then \eqref{lim_tau_jn} holds. Assume
$$ \lim_{n\to\infty}\frac{-\tau_{j,n}}{\lambda_{j,n}}=\pm \infty.$$
By Lemma \ref{L:lin_odd}, 
$$ \lim_{R\to +\infty}\left[\limsup_{n\to\infty} \int_{\big||y|-|\tau_{j,n}|\big|\geq R\lambda_{j,n}}\left|\nabla U_{0,\lin,n}^j(y)\right|^2+\left|U^j_{1,\lin,n}(y)\right|^2\,dy\right]=0.$$
Let $\{\mu_{j,n}\}_n$ be a sequence of positive numbers such that
\begin{equation}
 \label{mujn}
\lim_{n\to\infty}\frac{-\tau_{j,n}}{\mu_{j,n}}=\pm\infty,\quad \lim_{n\to\infty}\frac{\mu_{j,n}}{\lambda_{j,n}}=+\infty,
\end{equation} 
so that
\begin{equation}
\label{out_mujn}
 \lim_{n\to\infty}\int_{\big||y|-|\tau_{j,n}|\big|\geq \mu_{j,n}}\left|\nabla U_{0,\lin,n}^j(y)\right|^2+\left(U_{1,\lin,n}^j(y)\right)^2\,dy=0.
\end{equation} 
Then, in view of \eqref{expansion0}, \eqref{expansion1},
\begin{multline}
\label{expansion_un-vn}
 \int_{\big||y|-|\tau_{j,n}|\big|\leq \mu_{j,n}}\left|\nabla_{\tau,y}(u_n-v_n)(0,y)\right|^2\,dy\\
\geq \int_{\big||y|-|\tau_{j,n}|\big|\leq \mu_{j,n}}\left|\nabla_{\tau,y}U_{\lin,n}^j(0,y)\right|^2\,dy\\
+2\int_{\big||y|-|\tau_{j,n}|\big|\leq \mu_{j,n}}\nabla_{\tau,y}U_{\lin,n}^j(0,y)\cdot \left(\sum_{k=0}^{j-1}\nabla_{\tau,y}U_{\lin,n}^k(0,y)+\nabla_{\tau,y}w_n^{j}(0,y)\right).
\end{multline}
By \eqref{out_mujn} and the pseudo-orthogonality of the parameters, if $k\neq j$,
\begin{multline*} 
\int_{\big||y|-|\tau_{j,n}|\big|\leq\mu_{j,n}}\nabla_{\tau,y}U_{\lin,n}^j(0,y)\cdot \nabla_{\tau,y}U_{\lin,n}^k(0,y)\,dy
\\=\int \nabla_{\tau,y}U_{\lin,n}^j(0,y)\cdot \nabla_{\tau,y}U_{\lin,n}^k(0,y)\,dy+o_n(1)\underset{n\to\infty}{\longrightarrow}0.
\end{multline*}
Similarly, by \eqref{weak_CV_wJ},
$$ \lim_{n\to\infty} \int_{\big||y|-|\tau_{j,n}|\big|\leq\mu_{j,n}}\nabla_{\tau,y}U_{\lin,n}^j(0,y)\cdot \nabla_{\tau,y}w_n^j(0,y)\,dy=0,$$
by \eqref{out_mujn} and the definition of $w_n^j$. Hence, in view of \eqref{out_mujn} and \eqref{expansion_un-vn},
\begin{equation}
\label{liminf_loc}
\liminf_{n\to +\infty}\int_{\big||y|-|\tau_{j,n}|\big|\leq\mu_{j,n}}\left|\nabla_{\tau,y}(u_n-v_n)(0)\right|^2\,dy\geq \int\left|\nabla_{\tau,y}U^j(0,y)\right|^2\,dy.
\end{equation} 
By \eqref{support_1}, we get $\limsup_n|\tau_{j,n}|\leq 1$ which shows that we can always assume \eqref{lim_tau_jn}.

\end{proof}

In the sequel, we will use that in the expansion \eqref{expansion0}, \eqref{expansion1}, $(v_n(0),\partial_{\tau}v_n(0))$ might be seen as a profile 
$$\left(\frac{1}{\lambda_{0,n}^{1/2}}U_{\lin}^0\left(\frac{-\tau_{0,n}}{\lambda_{0,n}},\frac{y}{\lambda_{0,n}}\right),\frac{1}{\lambda_{0,n}^{3/2}}\partial_{\tau}U_{\lin}^0\left(\frac{-\tau_{0,n}}{\lambda_{0,n}},\frac{y}{\lambda_{0,n}}\right)\right)$$
up to a term going to $0$ in the energy space. Namely:
$$ \left(U_{0,\lin}^0,U_{1,\lin}^0\right)=(v_{0},v_1),\quad \lambda_{0,n}=\frac{1}{1-t_n}, \quad \tau_{0,n}=0.$$
It follows from Claim \ref{C:profiles} that the sequence $\big\{(\tau_{0,n},\lambda_{0,n})\big\}_n$ is pseudo-orthogonal to the other sequences $\big\{(\tau_{j,n},\lambda_{j,n})\big\}_n$.

In view of Claim \ref{C:profiles} we can always assume (after extraction) that for all $j$
\begin{equation}
\label{def_ttau}
\lim_{n\to \infty}-\tau_{j,n}=\ttau_j\in[-1,+1], \quad \lim_{n\to\infty} \lambda_{j,n}=\tlambda_j\in [0,\infty).
\end{equation} 
By the pseudo-orthogonality of the parameters and \eqref{support_1} (see again \cite[Lemma 2.5]{DuKeMe11a}), there is at most one index $j\geq 1$ such that $\tlambda_j>0$. Furthermore, for this profile, we can assume $\tau_{j,n}=0$ for all $n$. Reordering the profiles, we will always assume that this index is $1$, setting $U^1=0$ if there is no index $j$ with the preceding property. Rescaling (and extracting again), we can also assume:
$$ \forall n,\quad \lambda_{1,n}=1.$$
\begin{claim}
 \label{C:compact_supp}
The profile $(U^1(0),\partial_{\tau}U^1(0))$ is supported in $B_1$. Moreover, for all $J\geq 1$,
$$
\lim_{n\to \infty} \int_{|y|\geq 1} \left|\nabla w_{0,n}^J(y\right)|^2+\left(w_{1,n}^J(y)\right)^2dy=0.
$$
\end{claim}
\begin{proof}
This follows from \eqref{support_1} and simple orthogonality arguments.
\end{proof}

The main novelty of the proof with respect to the ``small'' type II blow-up case treated in \cite{DuKeMe11a}, \cite{DuKeMe10P} is that instead of using the smallness of $U^1$, we extract from $U^1$ away from the origin a very small piece which has non-negligible energy, inspired by the following Lemma: (in this Lemma, $Du\subset \RR_t\times \RR^3_x$ stands for ``domain of influence of $u$'', see for example \cite{Alinhac95Bo})

\begin{lemma}
\label{L:channel_for_compact}
 Let $(U_0,U_1)\in \hdot\times L^2$ be radial such that $(U_0,U_1)\neq (0,0)$ and $\supp(U_0,U_1)\subset B_1$. Then there exists $0<r_0<1$, $\eta_0>0$, $\eps_0>0$ such that for all $t\in \RR$:
$$\{t\}\times \big[|t|-\eps_0+r_0,|t|+r_0\big]\subset Du$$ and for either all $t>0$ or all $t<0$, we have
$$ \int_{|t|-\eps_0+r_0<r<|t|+r_0}\left((\partial_r U)^2+(\partial_tU)^2\right)r^2dr\geq \eta_0.$$
\end{lemma}
With this Lemma and the support property of $U^1$, one can reach a contradiction using ``channels of energy'' as in \cite{DuKeMe11a}. We have not carried out the proof this way because of technical complications arising from the introduction of $Du$, but the spirit of the proof below is the same. The Lemma can be proved using the arguments of Step 1 of the proof of \eqref{to_show} below.

Let us mention that Lemma \ref{L:channel_for_compact} gives an elementary proof of the fact that there is no compact self-similar blow-up for a radial solution of \eqref{CP} in dimension $3$. See \cite[Section 6]{KeMe08} for the proof in dimensions $3,4,5$ without the radiality assumption.

Let
\begin{equation}
\label{def_r0}
r_0=\inf\Big\{\rho\in [0,1],\quad \supp\left(U^1_0,U^1_1\right)\subset B_{\rho}\Big\}.
\end{equation}

We distinguish three cases:
\begin{itemize}
 \item Case $1$: $r_0>0$ and for all $j\geq 2$, $|\ttau_j|\leq r_0$.
\item Case $2$: there exists $j\geq 2$ such that $|\ttau_j|>r_0$.
\item Case $3$: $r_0=0$ and $\ttau_j=0$ for all $j\geq 2$.
\end{itemize}
\subsubsection{Proof of Case 1}
\label{SSS:Case1}
Denote by $r=|y|$ the radial variable. 
By definition of $r_0$, $\supp (U^1_0,U^1_1)\subset B_{r_0}$ and 
$$\forall \eps>0,\quad \int_{r_0-\eps<|y|<r_0}|\nabla U^1_0|^2+(U^1_1)^2\,dy>0.$$

Consider a small $\eps_0$ such that $\eps_0<\frac{r_0}{100}$ and 
\begin{equation}
 \label{def_eta0}
\eta_0=\int_{r_0-\eps_0<|y|<r_0} |\nabla U^1_0|^2+(U_1^1)^2\,dy.
\end{equation} 
is small. 
Let $\tU_0$, $\tU_1$ be the radial functions on $\RR^3$ defined by
\begin{equation}
\label{def_tU}
 \left\{\begin{aligned}
        \tU_0(r)&=U^1_0(r)\text{ if }r_0-\eps_0<r,&\quad \tU_0(r)&=U^1_0(r_0-\eps_0)&\text{if }0< r\leq r_0-\eps_0\\
        \tU_1(r)&=U^1_1(r)\text{ if }r_0-\eps_0<r,&\quad \tU_1(r)&=0&\text{if }0< r\leq r_0-\eps_0.
       \end{aligned}\right.
\end{equation}
Let $\tU$ be the solution of \eqref{CP} with initial data $(\tU_0,\tU_1)$. 

Note that
\begin{equation}
\label{eta0_bis}
\eta_0= \int_{r_0-\eps_0<|y|<r_0}\big|\nabla \tU_0\big|^2+\big(\tU_1\big)^2\,dy=\int \big|\nabla \tU_0\big|^2+\big(\tU_1\big)^2\,dy.
\end{equation} 
As $\eta_0$ is small, $\tU$ is globally defined. 

Let $\varphi\in C^{\infty}(\RR^3)$, radial, be such that $\varphi(y)=1$ if $|y|\geq r_0-\eps_0$ and $\varphi(y)=0$ is $|y|\leq r_0-2\eps_0$. 
Fix $J>1$ and define (see \eqref{simplicity}, \eqref{simplicity_bis} for the notations $U_{0,\lin,n}^j$ and $U_{1,\lin,n}^j$):
\begin{equation}
\label{deftu0}
\left\{ 
\begin{aligned}
  \tu_{0,n}&=v_n(0)+\tU_0+\varphi\sum_{j=2}^JU_{0,\lin,n}^j+\varphi\,w_{0,n}^J\\
  \tu_{1,n}&=\partial_{\tau}v_n(0)+\tU_1+\varphi\sum_{j=2}^JU_{1,\lin,n}^j+\varphi\, w_{1,n}^J.
 \end{aligned}
\right.
\end{equation}
Note that this definition is independent of $J$.

Let $\tu_n$ be the solution of \eqref{CP} with data $(\tu_{0,n},\tu_{1,n})$ at $\tau=0$. The data $(u_{0,n}(y),u_{1,n}(y))$ and $(\tu_{0,n}(y),\tu_{1,n}(y))$ are equal if $|y|\geq r_0-\eps_0$. By finite speed of propagation, if $\tau$ is in the domain of definition of $u_n$ and $\tu_n$,
$$ |y|\geq r_0-\eps_0+|\tau|\Longrightarrow (u_n(\tau,y),\partial_{\tau}u_n(\tau,y))=(\tu_n(\tau,y),\partial_{\tau}\tu_n(\tau,y)).$$

Let $I_n^+=I_n\cap [1/2-r_0/100,+\infty)$, $I_n^-=I_n\cap (-\infty,-1/2+r_0/100]$. We will show that there exists $\eta_1>0$ such that the following holds for large $n$, for all sequences $\{\sigma_n\}_n$ with $\sigma_n\in I_n^+$, or for all sequences $\{\sigma_n\}_n$ with $\sigma_n\in I_n^-$:
\begin{equation}
\label{to_show}
\int_{r_0-\eps_0+|\sigma_n|\leq |y|\leq r_0+|\sigma_n|} |\nabla_{\tau,y}(u_n-v_n)(\sigma_n,y)|^2\,dy\geq \eta_1.
\end{equation} 
This will contradict the following lemma:
\begin{lemma}
\label{L:no_channel}
Let $u$ satisfy the assumptions of Proposition \ref{P:no_ss}. Let $\sigma_n^+=\frac{1}{2}$ and $\sigma_n^-=\frac{T-t_n}{1-t_n}$. Then for large $n$, $\sigma_n^{\pm}\in I_n^{\pm}$ and
\begin{equation}
\label{no_channel_finite}
\lim_{n\to\infty} \int_{|y|\geq |\sigma_n^{\pm}|}\Big[ \big|\nabla(u_n-v_n)(\sigma_n^{\pm})\big|^2+\big(\partial_{\tau}(u_n-v_n)(\sigma_n^{\pm})\big)^2\Big]\,dy=0.
\end{equation}
\end{lemma}
\begin{proof}
Let $\sigma_n^+=\frac 12$. If \eqref{no_channel_finite} does not hold for $\sigma_n^+$ then (after extractions of subsequences) there exists $\eta_2>0$ such that for large $n$:
\begin{equation}
\label{channel_bup}
 \int_{|y|\geq |\sigma_n^+|}|\nabla(u_n-v_n)(\sigma_n^+)|^2+|\partial_t(u_n-v_n)(\sigma_n^+)|^2\,dy\geq \eta_2.
\end{equation} 
Going back to the original variables $(t,x)$, we rewrite \eqref{channel_bup} as
\begin{equation*}
 \int_{|x|\geq \frac{1-t_n}{2}} \left|\nabla a\left(\frac 12+\frac{t_n}{2},x\right)\right|^2+\left|\partial_{t} a\left(\frac 12+\frac{t_n}{2},x\right)\right|^2\,dx\geq \eta_2,
\end{equation*} 
which contradicts the fact that $a(1/2+t_n/2)$ is supported in the set $\big\{|x|\leq (1-t_n)/2\big\}$. 

Let $\sigma_n^-=\frac{T-t_n}{1-t_n}\in I_n^-$. If \eqref{no_channel_finite} does not hold for $\sigma_n^-$, we obtain a $\eta_2>0$ such that (at least for some subsequences in $n$),
\begin{equation*}
 \int_{|x|\geq t_n-T} \left|\nabla a\left(T,x\right)\right|^2+\left|\partial_{\tau} a\left(T,x\right)\right|^2\,dx\geq \eta_2.
\end{equation*} 
Recalling that $a(T,\cdot)$ is supported in $|x|\leq 1-T$ and that $t_n$ tends to $1$, this is also a contradiction. 
\end{proof}
We next prove that \eqref{to_show} holds for large $n$, for all sequences $\{\sigma_n\}_n$ with $\sigma_n\in I_n^+$ or for all sequences $\{\sigma_n\}_n$ with $\sigma_n\in I_n^-$. We divide the proof into a few steps.

\EMPH{Step 1}
Note that by chosing $\eps_0>0$ small, we can always insure that
\begin{equation}
 \label{big_dr}
\int_{r_0-\eps_0}^{r_0}\big(\partial_r(r\tU_0)\big)^2+\big(r\tU_1\big)^2\,dr\geq \frac{\eta_0}{2}.
\end{equation} 
Indeed by integration by parts,
$$ \int_{r_0-\eps_0}^{r_0} \big(\partial_r(r \tU_0)\big)^2\,dr=-(r_0-\eps_0)\left(\tU_0(r_0-\eps_0)\right)^2+\int_{r_0-\eps_0}^{r_0} \big(\partial_r \tU_0\big)^2r^2\,dr.$$
By Cauchy-Schwarz inequality,
$$\tU_0(r_0-\eps_0)^2=\left(\int_{r_0-\eps_0}^{r_0} \partial_r \tU_0\,dr\right)^2\leq \frac{\eps_0}{(r_0-\eps_0)^2} \int_{r_0-\eps_0}^{r_0} \big(\partial_r \tU_0\big)^2\,r^2dr,$$
and \eqref{big_dr} follows from the definition of $\eta_0$, if $\frac{\eps_0}{r_0-\eps_0}\leq \frac 12$.

We denote by $\tU_{\lin}$ the solution of the linear equation \eqref{lin_wave} with data $(\tU_0,\tU_1)$.
By Proposition \ref{P:linear_behavior}, the following holds for all $\tau>0$ or for all $\tau<0$:
$$\int_{r_0-\eps_0+|\tau|}^{r_0+|\tau|} \big(\partial_r (r\tU_{\lin}(\tau,r))\big)^2+\big(\partial_{\tau} (r\tU_{\lin}(\tau,r))\big)^2\,dr\geq \frac{\eta_0}{4}.$$
By integration by parts (using that $\tU_{\lin}$ is supported in $\{r\leq r_0+|\tau|\}$), for $\eps_0$ small,
$$\int_{r_0-\eps_0+|\tau|}^{r_0+|\tau|} \big(\partial_r (r\tU_{\lin}(\tau,r))\big)^2\,dr=-(r_0-\eps_0+|\tau|)\tU_{\lin}^2(\tau,r_0-\eps_0+|\tau|)+\int_{r_0-\eps_0+|\tau|}^{r_0+|\tau|} \big(\partial_r \tU_{\lin}(\tau,r)\big)^2r^2\,dr,$$
so that 
\begin{equation*}
\int_{r_0-\eps_0+|\tau|}^{r_0+|\tau|} \Big[\big(\partial_r \tU_{\lin}(\tau,r)\big)^2
+\big(\partial_{\tau} \tU_{\lin}(\tau,r)\big)^2\Big]\,r^2dr\geq \frac{\eta_0}{4}
\end{equation*} 
for all $\tau<0$ or for all $\tau>0$. As $\eta_0=\left\|(\tU_0,\tU_1)\right\|^2_{\hdot\times L^2}$ is small, the local well-posedness theory implies that for all $\tau$,
\begin{equation}
\label{tU_quasi_lin}
\left\|\left(\tU(\tau)-\tU_{\lin}(\tau),\partial_{\tau}\tU(\tau)-\partial_{\tau}\tU_{\lin}(\tau)\right)\right\|_{\hdot\times L^2}^2\leq \frac{\eta_0}{100}.
\end{equation} 
We deduce that for all $\tau>0$ or for all $\tau<0$,
\begin{equation}
\label{channel2}
\int_{r_0-\eps_0+|\tau|}^{r_0+|\tau|} \Big[\big(\partial_r \tU(\tau,r)\big)^2
+\big(\partial_{\tau} \tU(\tau,r)\big)^2\Big]r^2\,dr\geq \frac{\eta_0}{5}.
\end{equation} 
Let us assume that \eqref{channel2} holds for all $\tau>0$, the case $\tau<0$ being similar. 

\EMPH{Step 2}
By Lemma \ref{L:lin_odd}, for each fixed $J\geq 2$, 
$$\varphi(y)\sum_{j=2}^J \left(U_{0,\lin,n}^j(y),U_{1,\lin,n}^j(y)\right)=\sum_{j=2}^J \varphi(|\ttau_j|)\left(U_{0,\lin,n}^j(y),U_{1,\lin,n}^j(y)\right)+\eps_n^J(y),$$
where
$$ \lim_{n\to\infty}\big\|\eps_n^J\big\|_{\hdot\times L^2}=0.$$

\EMPH{Step 3}
In this step we show that 
for each fixed $j\geq 2$, we have:
\begin{equation}
\label{disperse1}
\lim_{n\to\infty}
 \left\|S_{\lin}(t)\left( \varphi(|\ttau_j|)U_{0,\lin,n}^j, \varphi(|\ttau_j|)U_{1,\lin,n}^j\right)\right\|_{L^8\left((0,r_0/4)\times \RR^3\right)}=0,
\end{equation} 
and that
\begin{equation}
 \label{disperse2}
\lim_{J\to\infty} \limsup_{n\to+\infty} \Big\|S_{\lin}(t)\left(\varphi w_{0,n}^J,\varphi w_{1,n}^J\right)\Big\|_{L^8(\RR^4)}=0.
\end{equation} 
Indeed, \eqref{disperse2} follows from \cite[Claim 2.11]{DuKeMe11a}. To show \eqref{disperse1}, we distinguish $3$ cases.
\begin{itemize}
 \item 
If $|\ttau_j|\leq r_0/4$, then $\varphi(|\ttau_j|)=0$ and \eqref{disperse1} is obvious.
In the other cases, we will use the formula
\begin{equation}
\label{change_of_var}
\left\|S_{\lin}(t)\left( U_{0,\lin,n}^j, U_{1,\lin,n}^j\right)\right\|_{L^8\left((0,r_0/4)\times \RR^3\right)}^8=
\int_{-\tau_{j,n}/\lambda_{j,n}}^{\frac{r_0/4-\tau_{j,n}}{\lambda_{j,n}}} \int_{\RR^3}\left|U_{\lin}^j(t,x)\right|^8\,dx\,dt.
\end{equation} 
\item If $\ttau_j>r_0/4$, then
$$ \lim_{n\to \infty} \frac{-\tau_{j,n}}{\lambda_{j,n}}=+\infty,$$
and both ends of the time integral in \eqref{change_of_var} go to $+\infty$, which implies \eqref{disperse1}.
\item If $\ttau_j<-r_0/4$, 
then 
$$ \lim_{n\to \infty} \frac{r_0/4-\tau_{j,n}}{\lambda_{j,n}}=-\infty,$$
and \eqref{disperse1} follows again from \eqref{change_of_var}.
\end{itemize}
\EMPH{Step 4}
By \eqref{disperse1} and \eqref{disperse2}, we get that  
$$ \lim_{J\to \infty} \limsup_{n\to\infty} \left\|S_{\lin}(\tau)\left(\sum_{j=2}^J \varphi(|\ttau_j|) U_{0,\lin,n}^j+\varphi w_{0,n}^J,\sum_{j=2}^J \varphi(|\ttau_j|) U_{1,\lin,n}^j+\varphi w_{1,n}^J\right)\right\|_{L^{8}((0,r_0/4)\times \RR^3)}=0. $$
Hence by \eqref{deftu0} and long time perturbation theory (see e.g. \cite[Theorem 2.20]{KeMe08}), 
\begin{align*}
 \tu_n\left(\frac{r_0}{4},y\right)-v_n\left(\frac{r_0}{4},y\right)&=\tU\left(\frac{r_0}{4},y\right)+\sum_{j=2}^J \varphi(|\ttau_j|) U_{\lin,n}^j\left(\frac{r_0}{4},y\right)+\tw_{0,n}^J(y)\\
 \partial_{\tau}\tu_n\left(\frac{r_0}{4},y\right)-\partial_{\tau}v_n\left(\frac{r_0}{4},y\right)&=\partial_{\tau}\tU\left(\frac{r_0}{4},y\right)+\sum_{j=2}^J \varphi(|\ttau_j|) \partial_{\tau}U_{\lin,n}^j\left(\frac{r_0}{4},y\right)+\tw_{1,n}^J(y),
\end{align*}
where the sequences $\big\{(\tw_{0,n}^J,\tw_{1,n}^J)\big\}_n$ (uniformly bounded in $\hdot\times L^2$) satisfy
\begin{equation}
\label{disp_tw}
\lim_{J\to \infty} \limsup_{n\to\infty} \left\|S_{\lin}(t)(\tw_{0,n}^J,\tw_{1,n}^J)\right\|_{L^8(\RR\times \RR^3)}=0.
\end{equation} 
Let $\tphi\in C^{\infty}(\RR^3)$, radial, such that $\tphi(y)=1$ if $|y|\geq r_0$, $\tphi(y)=0$ if $|y|\leq 9r_0/10$. Let 
$\ttu_n$ be the solution of \eqref{CP} such that:
\begin{align*}
 \ttu_n\left(\frac{r_0}{4},y\right)-v_n\left(\frac{r_0}{4},y\right)&=\tphi(y)\left[\tu_n\left(\frac{r_0}{4},y\right)-v_n \left( \frac{r_0}{4},y \right)\right]\\
 \partial_{\tau}\ttu_n\left(\frac{r_0}{4},y\right)-\partial_{\tau}v_n\left(\frac{r_0}{4},y\right)&=\tphi(y)\left[\partial_{\tau}\tu_n\left(\frac{r_0}{4},y\right)-\partial_{\tau}v_n \left( \frac{r_0}{4},y \right)\right].
\end{align*}
Note that
$$ |y|\geq \frac{5r_0}{4}-\eps_0\Longrightarrow \left(\ttu_n\left( \frac{r_0}{4},y\right),\partial_{\tau}\ttu_n\left( \frac{r_0}{4},y \right)\right)
=\left(u_n\left( \frac{r_0}{4},y\right),\partial_{\tau}u_n\left( \frac{r_0}{4},y \right)\right).$$
Furthermore, by Lemma \ref{L:lin_odd}, using that $|\ttau_j|\leq r_0$ for all $j$ (see the proof of \ref{lim_tau_jn}), 
\begin{multline*}
 \left(\ttu_n\left(\frac{r_0}{4}\right),\partial_{\tau}
 \ttu_n\left(\frac{r_0}{4}\right)\right)-
 \left(v_n\left(\frac{r_0}{4}\right),\partial_{\tau}
 v_n\left(\frac{r_0}{4}\right)\right)\\
=\left(\tphi\,\tU\left( \frac{r_0}{4}\right),\tphi\,\partial_{\tau}\tU\left( \frac{r_0}{4}\right)\right)+
\sum_{\substack{j=2,\ldots, J\\ \ttau_j>0}}\varphi(|\ttau_j|)\left(U_{\lin,n}^j\left( \frac{r_0}{4}\right),\partial_{\tau}U_{\lin,n}^j\left( \frac{r_0}{4}\right)\right)+\left(\tphi \tw_{0,n}^J,\tphi\tw_{1,n}^J  \right)+\teps_n^J,
\end{multline*}
where for all $J\geq 2$,
$$\lim_{n\to\infty} \left\|\teps_n^J\right\|_{\hdot\times L^2}=0.$$
Indeed, if $\ttau_j<0$, $U_{\lin,n}\left( \frac{r_0}{4}\right)$ concentrates (by Lemma \ref{L:lin_odd}) on $\{|y|=\left|\frac{r_0}{4}-|\ttau_j|\right|\}$, which is outside the support of $\tphi$. If $\ttau_j>0$, if concentrates on $\{|y|=\frac{r_0}{4}+\ttau_j\}$, and the statement follows from the fact that $\varphi(\ttau_j)\tphi\left(\frac{r_0}{4}+\ttau_j\right)=\varphi(\ttau_j)$. 

We will again use long time perturbation arguments, for $\tau>\frac{r_0}{4}$. By the same proof as above,
$$ \limsup_{J\to\infty}\left\|S_{\lin}(t)\left( \sum_{\substack{j=2,\ldots J\\ \ttau_j>0}} \varphi(\ttau_j)\left(U_{\lin,n}^j\left(\frac{r_0}{4}\right),\partial_{\tau}U_{\lin,n}^j\left(\frac{r_0}{4}\right)\right)+(\tphi\tw_{0,n}^J,\tphi\tw_{1,n}^J)\right) \right\|_{L^8\left([r_0/4,+\infty)\times \RR^3\right)}=0.$$
Note also that by the smallness of $\eta_0$, the solution of \eqref{CP} with initial data $\left( \tphi\, \tU\left( \frac{r_0}{4} \right),\tphi\, \partial_{\tau}\tU\left( \frac{r_0}{4} \right) \right)$ at $\tau=r_0/4$ is globally defined and scatters, and coincides with $\tU$ for any $\tau,y$ such that $\tau\geq r_0/4$, $|y|\geq \tau+r_0-\eps_0$.

Hence, for $\tau\geq r_0/4$ and $|y|\geq \tau+r_0-\eps_0$,
\begin{multline}
\label{last_expansion}
 \left(u_n,\partial_{\tau}u_n\right)(\tau,y)-\left( v_n,\partial_{\tau}v_n\right)(\tau,y)
=\left( \tU,\partial_{\tau}\tU\right)(\tau,y)\\
+\sum_{\substack{j=2,\ldots,J\\ \ttau_j>0}} \varphi(\ttau_j)\left(U^j_{\lin,n},\partial_{\tau}U^j_{\lin,n}\right)(\tau,y)+\left(\ttw_{\lin,n}^J,\partial_{\tau}\ttw_{\lin,n}^J\right)(\tau,y)+\left(\tteps_{n}^J,\partial_{\tau}\tteps_{n}^J\right)(\tau,y),
\end{multline}
where $\ttw_{\lin,n}^J$ is a linear solution that satisfies
\begin{equation}
\label{ttwnJ}
\lim_{J\to\infty}\limsup_{n\to\infty}\left\|\ttw_n^J\right\|_{L^8([r_0/4,+\infty)\times \RR^3)}=0,   
\end{equation} 
and $\tteps_n^J$ satisfies:
\begin{equation}
\label{ttepsnJ}
\lim_{J\to\infty}\limsup_{n\to \infty}\left[ \left\|\tteps_n^J\right\|_{L^8([r_0/4,+\infty)\times \RR^3) } +\sup_{\tau\geq r_0/4}\left\|\left(\tteps_{n}^J(\tau),\partial_{\tau}\tteps_{n}^J(\tau)\right)\right\|_{\hdot\times L^2}\right]=0.
\end{equation} 
Note also that if $\{\sigma_n\}_n$ is a sequence such that $\sigma_n\in I_n^+$ for all $n$, the following pseudo-orthogonality property holds (see Appendix \ref{A:quasi_ortho}):
\begin{multline}
 \label{quasi_ortho}
\lim_{n\to+\infty} \left|\int_{S_n} \nabla_{\tau,y}\tU_{\lin}(\sigma_n,y)\cdot \nabla_{\tau,y} U^j_{\lin,n}(\sigma_n,y)\,dy\right|+ 
\left|\int_{S_n} \nabla_{\tau,y}\tU_{\lin}(\sigma_n,y)\cdot \nabla_{\tau,y} w_n^J(\sigma_n,y)\,dy\right|\\
+\left|\int_{S_n} \nabla_{\tau,y}U_{\lin,n}^j(\sigma_n,y)\cdot \nabla_{\tau,y} U^k_{\lin,n}(\sigma_n,y)\,dy\right|+\left|\int_{S_n} \nabla_{\tau,y}U_{\lin,n}^j(\sigma_n,y)\cdot \nabla w_{n}^J(\sigma_n,y)\,dy\right|=0,
\end{multline}
where $S_n=\{y\in \RR^3,\; \sigma_n+r_0-\eps_0\leq |y|\leq \sigma_n+r_0\}$ and $2\leq k<j\leq J$.

In view of \eqref{tU_quasi_lin}, \eqref{channel2} and \eqref{last_expansion}, we get that \eqref{to_show} holds for all sequences $\{\sigma_n\}_{n}$ with $\sigma_n\in I_n^+$, concluding the proof of this case.

\subsubsection{Sketch of proof for Case 2}
\label{SSS:Case2}
Assume that there exists $j\geq 2$ such that $|\ttau_{j}|>r_0$. Let $\tr_0=\sup_{j\geq 2}|\ttau_j|>r_0$. Choose  $\eps_0<\frac{\tr_0-r_0}{100}$. Note that there exists $j_0\geq 2$ such that $\tr_0-\eps_0<|\ttau_{j_0}|\leq \tr_0$. Assume (for example) that $\ttau_{j_0}>0$. Let $\varphi\in C^{\infty}(\RR^3)$, radial, such that $\varphi(y)=1$ if $|y|\geq \tr_0-\eps_0$ and $\varphi(y)=0$ if $|y|\leq \tr_0-2\eps_0$. Define:
\begin{equation}
\label{deftun}
\left\{ 
\begin{aligned}
  \tu_{0,n}&=v_n(0)+\varphi\sum_{j=2}^JU_{0,\lin,n}^j+\varphi\,w_{0,n}^J\\
  \tu_{1,n}&=\partial_{\tau}v_n(0)+\varphi\sum_{j=2}^JU_{1,\lin,n}^j+\varphi\, w_{1,n}^J.
 \end{aligned}
\right.
\end{equation}
The proof now proceeds as before up to \eqref{last_expansion} (but without $(\tU,\partial_t \tU)$), replacing $r_0$ by $\tr_0$. To conclude the proof, use energy conservation for the linear solution $\varphi(|\tau_{j_0}|)U_{\lin,n}^{j_0}$ together with Lemma \ref{L:lin_odd} and the pseudo-orthogonality of the profiles \eqref{quasi_ortho}.
\subsubsection{Sketch of proof for Case 3}
\label{SSS:Case3}
We now assume that $r_0=0$ (i.e. $U^1=0$) and that $\ttau_j=0$ for all $j$. Note that by assumption \eqref{ss_energy},
\begin{equation}
\label{wnJ_exterior}
\exists \eps_0>0 \text{ s.t. } \forall J\geq 2,\quad \limsup_{n\to\infty} \int_{|y|\geq \eps_0}\left|\nabla w_{0,n}^J\right|^2+\left(w_{1,n}^J\right)^2\geq \eta_0.
\end{equation}
Let $\varphi\in C^{\infty}$ radial, such that $\varphi(y)=1$ if $|y|\geq \eps_0$, $\varphi(y)=0$ if $|y|\leq \eps_0/2$. Define $(\tu_{0,n},\tu_{1,n})$ by \eqref{deftun}, and $\tu_n$ as the solution of \eqref{CP} with initial data $(\tu_{0,n},\tu_{1,n})$. Then
\begin{equation*}
 \tu_{0,n}=v_{0,n}+\tw_{0,n}\qquad
\tu_{1,n}=v_{1,n}+\tw_{1,n},
\end{equation*}
where 
\begin{gather}
\label{wn_disperse}
\lim_{n\to\infty}\left\|S_{\lin}(\tau)(\tw_{0,n},\tw_{1,n})\right\|_{L^8(\RR^4)}=0\\
\label{wn_exterior}
 \limsup_{n\to+\infty} \int_{|y|\geq \eps_0}|\nabla \tw_{0,n}|^2+(\partial_{\tau}\tw_{1,n})^2\,dy\geq \eta_0.
\end{gather}
Indeed, fixing a large $J$ we see that 
$$ \tw_{0,n}=\varphi\sum_{j=1}^J U_{0,\lin,n}^j+ \varphi w_{0,n}^J,\quad \tw_{1,n}=\varphi\sum_{j=1}^J U_{1,\lin,n}^j +\varphi w_{1,n}^J,$$
where by Lemma \ref{L:lin_odd}, for all $j$,
$$ \lim_{n\to \infty} \left\|\left(\varphi U_{0,\lin,n}^j,\varphi U_{1,\lin,n}^j\right)\right\|_{\hdot\times L^2}=0$$
and by \cite[Claim 2.5]{DuKeMe11a},
$$ \lim_{J\to +\infty} \limsup_{n\to \infty}\left\|S_{\lin}(\tau)(\varphi w_{0,n}^J,\varphi w_{1,n}^J)\right\|_{L^8(\RR^4)}=0,$$
hence \eqref{wn_disperse} and, in view of \eqref{wnJ_exterior}, \eqref{wn_exterior}.

Note that \eqref{wn_disperse} implies
\begin{equation}
\label{pointwise}
\lim_{n\to\infty}\|\tw_{0,n}\|_{L^{\infty}(\eps_0/2,1)}=0.
\end{equation} 
Indeed, if \eqref{pointwise} does not hold, there exists (after extraction of subsequences) a sequence $r_n\in(\eps_0/2,1)$ and $\eps_1>0$ such that
$$ \forall n,\quad |\tw_{0,n}(r_n)|\geq \eps_1.$$
By \eqref{wn_disperse} and the paragraph following \eqref{weak_CV_wJ}, we know that 
$$ r_n^{1/2}w_{0,n}\left(r_n \cdot\right)\xrightharpoonup[n\to \infty]{} 0 \text{ weakly in }\hdot.$$
In particular, using that $f\mapsto f(1)$ is a continuous linear form on $\hdot_{\rad}$, we get that $\tw_{0,n}(r_n)\to 0$ as $n\to\infty$, a contradiction.

By integration by parts and \eqref{pointwise}, we get that for large $n$,
$$ \int_{\eps_0}^1 (\partial_r(r \tw_n))^2+(\partial_{\tau}(r \tw_n))^2\,dr\geq \frac{\eta_0}{2}.$$
We now argue as before, using Proposition \ref{P:linear_behavior}, directly for all $\tau>0$ or for all $\tau<0$, globally in time. This concludes the proof of Proposition \ref{P:no_ss}.
\subsection{Global solutions}
\label{SS:self_simG}
In this subsection we show the following analogue of Proposition \ref{P:no_ss} for globally defined solutions:
\begin{prop}
 \label{P:global_no_ss}
Let $u$ be a radial solution of \eqref{CP}  satisfying \eqref{bound_intro} and such that $T_+(u)=+\infty$. Let $s_n\to +\infty$. Assume that $\Sbf_{\lin}(-s_n)(u(s_n),\partial_t u(s_n))$ has a weak limit $(v_0,v_1)$ in $\hdot\times L^2$ as $n\to\infty$. Let $(v_{\lin}(t),\partial_tv_{\lin}(t))=\Sbf_{\lin}(t)(v_0,v_1)$. Then
\begin{equation*} 
 \forall c_0>0, \lim_{n\to\infty}\int_{|x|\geq c_0s_n}\left(|\nabla (u-v_{\lin})(s_n,x)|^2+(\partial_t (u-v_{\lin})(s_n,x))^2\right)\,dx=0
\end{equation*} 
\end{prop}
The proof is close to the one of Proposition \ref{P:no_ss}, we will only highlight the aspects that are specific to the globally defined case. We argue by contradiction: if the conclusion of Proposition \ref{P:global_no_ss} does not hold, there exists a subsequence of $\{s_n\}_n$ (still denoted by $\{s_n\}_n$) and a $c_0>0$ such that
\begin{equation} 
\label{absurd_global}
 \forall n,\quad \int_{|x|\geq c_0s_n}\left(|\nabla (u-v_{\lin})(s_n,x)|^2+|\partial_t (u-v_{\lin})(s_n,x)|^2\right)\,dx\geq c_0.
\end{equation} 
\subsubsection{Preliminaries}
Let $v_{\NL}$ be the unique solution of \eqref{CP} such that $T_+(v_{\NL})=+\infty$ and
$$ \lim_{t\to+\infty} \|(v_{\lin}(t),\partial_{t} v_{\lin}(t))-(v_{\NL}(t),\partial_{t} v_{\NL}(t))\|_{\hdot\times L^2}=0.$$
Without loss of generality, we can assume that $v_{\NL}(t)$ is defined on $[0,+\infty)$. Let, for $\tau\in [-1,+\infty)$, $y\in \RR^3$,
\begin{equation*}
 u_n(\tau,y)=s_n^{1/2}u(s_n(1+\tau),s_ny),\quad
v_n(\tau,y)=s_n^{1/2}v_{\NL}(s_n(1+\tau),s_ny).
\end{equation*} 
Consider (after extraction of a subsequence in $n$) a profile decomposition of the form \eqref{expansion0}, \eqref{expansion1} for the sequence $(u_{n}(0),\partial_{\tau} u_{n}(0))-(v_{n}(0),\partial_{\tau} v_{n}(0))$.
\begin{claim}
\label{C:profiles_gl}
 Fix $j\geq 1$. Then (after extraction of subsequences in $n$),
\begin{equation}
\label{lim_tau_jn_gl}
 \lim_{n\to \infty}-\tau_{j,n}=\ttau_j \in[-1,1],
\end{equation} 
and 
$$\lim_{n\to\infty} \lambda_{j,n}=\tlambda_j\in [0,+\infty).$$
Furthermore,
\begin{equation}
 \label{lim_sn}
\lim_{n\to\infty} s_n\lambda_{j,n}+\frac{1}{s_n\lambda_{j,n}}+ |\tau_{j,n}+1|s_n=+\infty.
\end{equation} 
\end{claim}
\begin{proof}
 Note that by finite speed of propagation
$$\lim_{R\to +\infty} \left[\limsup_{n\to\infty}\int_{|x|\geq |s_n|+R}\left|\nabla_{t,x}(u-v_{\lin})(s_n,x)\right|^2\,dx\right]=0,$$
and thus
\begin{equation}
\label{out_global}
 \lim_{R\to+\infty} \left[\limsup_{n\to \infty} \int_{|y|\geq 1+\frac{R}{s_n}} \left|\nabla_{\tau,y}(u_n-v_n)(0,y)\right|^2\,dy\right]=0.
\end{equation} 
The proof of \eqref{lim_tau_jn_gl} and of the fact that the sequence $\{\lambda_{j,n}\}_n$ is bounded is then the same as the proof of Claim \ref{C:profiles}, using \eqref{out_global} instead of \eqref{support_1}.

It remains to show \eqref{lim_sn}. This follows from the fact that:
$$ w-\lim_{n\to \infty} \Sbf_{\lin}(-s_n)\Big(u(s_n)-v_{\NL}(s_n),\partial_{t}u(s_n)-\partial_t v_{\NL}(s_n)\Big)=0$$
in $\hdot\times L^2$. Indeed, by the profile decomposition \eqref{expansion0}, \eqref{expansion1} and the definitions of $u_n$ and $v_n$, we get
\begin{multline*}
S_{\lin}(-s_n)\big(u(s_n)-v_{\NL}(s_n),\partial_{t}u(s_n)-\partial_t v_{\NL}(s_n)\big)\\
=\sum_{j=1}^{J}\frac{1}{s_n^{1/2}\lambda_{j,n}^{1/2}} U^j_{\lin}\left(\frac{-s_n-s_n\tau_{j,n}}{s_n\lambda_{j,n}},\frac{\cdot}{\lambda_{j,n}}\right)+\frac{1}{s_n^{1/2}}w_n^J\left(1,\frac{\cdot}{s_n}\right),
\end{multline*}
and the analoguous expansion for the time derivatives, which yields \eqref{lim_sn}.
\end{proof}
Note also that \eqref{out_global} implies:
\begin{equation}
\label{support_wnJ}
\lim_{n\to \infty} \int_{|y|\geq 1} |\nabla w_{0,n}^J(y)|^2+|w_{1,n}^J(y)|^2dy=0.
\end{equation} 

As in the proof of Proposition \ref{P:no_ss}, by the pseudo-orthogonality of the parameters and \eqref{out_global} there is at most one index $j\geq 1$ such that $\tlambda_j>0$. Reordering the profiles, we will always assume that this index is $1$, setting $U^1=0$ if there is no index $j$ with the preceding property. We will also assume that for all $n$, $\tau_{1,n}=0$ and $\lambda_{1,n}=1$. It follows from \eqref{out_global} that $(U^1(0),\partial_{\tau}U^1(0))$ is supported in $B_1$. 

If $U^1\neq 0$, define $r_0$ by \eqref{def_r0}. Let $\eps_0$ be such that $\eta_0$ (defined by \eqref{def_eta0}) is small, and let $(\tU_0,\tU_1)$ be as in \eqref{def_tU}.

We next prove:
\begin{lemma}
\label{L:I(A)}
For all $A\in \RR$, the following limit exists:
$$ \lim_{t\to+\infty} \int_{|x|\geq t+A} |\nabla (u-v_{\NL})(t,x)|^2+(\partial_t(u-v_{\NL})(t,x))^2\,dx=\III(A).$$
Furthermore, there is an $\III_{\infty}\geq 0$ such that
\begin{equation}
\label{lim_I}
\lim_{A\to -\infty} \III(A)=\III_{\infty},\quad \lim_{A\to +\infty} \III(A)=0,
\end{equation}
 and the function $\III$ is nonincreasing and uniformly continuous on $\RR$. 
\end{lemma}
\begin{proof}
We first argue as in the beginning of \S \ref{SSS:Case1}. Set $r_0=1/2$, $\eps_0=1/100$ if $U^1=0$. 
 Let $\varphi\in C^{\infty}(\RR^3)$ be a radial function such that $\varphi(x)=1$ if $|x|\geq r_0-\eps_0$, $\varphi(x)=0$ if $|x|\leq r_0-2\eps_0$. Define $(\tu_{0,n},\tu_{1,n})$ by \eqref{deftu0}. Let $\tu_{n}$ as the solution of \eqref{CP} with initial data $(\tu_{0,n},\tu_{1,n})$. Arguing as in \S \ref{SSS:Case1}, we get that 
\begin{equation}
\label{FSOP_gl}
|y|\geq r_0-\eps_0+|\tau| \Longrightarrow (u_n(\tau,y),\partial_{\tau}u_n(\tau,y))=(\tu_n(\tau,y),\partial_{\tau}\tu_n(\tau,y))
 \end{equation} 
and
\begin{multline*}
\left(\tu_n\left( \frac{r_0}{4},y\right),\partial_{\tau}\tu_n\left( \frac{r_0}{4},y\right)\right) -
\left(v_n\left( \frac{r_0}{4},y\right),\partial_{\tau}v_n\left( \frac{r_0}{4},y\right)\right) \\
=\left( \tU\left( \frac{r_0}{4},y \right),\partial_{\tau} \tU\left( \frac{r_0}{4},y\right) \right)+\sum_{j=2}^J \varphi(|\ttau_j|) \left( U_{\lin,n}^j\left( \frac{r_0}{4},y\right) \right)+\left(\tw_{0,n}^J,\tw_{1,n}^J\right),
\end{multline*}
where $\left(\tw_{0,n}^J,\tw_{1,n}^J\right)$ satisfies \eqref{disp_tw}. Let $\tphi\in C^{\infty}(\RR^3)$, radial such that $\tphi(y)=1$ if $|y|\geq r_0$ and $\tphi(y)=0$ if $|y|\leq \frac{9}{10}r_0$. Let $\ttu_n$ be the solution of \eqref{CP} such that
\begin{align*}
 \ttu_n\left(\frac{r_0}{4},y\right)-v_n\left(\frac{r_0}{4},y\right)&=\tphi(y)\left[\tu_n\left(\frac{r_0}{4},y\right)-v_n \left( \frac{r_0}{4},y \right)\right]\\
 \partial_{\tau}\ttu_n\left(\frac{r_0}{4},y\right)-\partial_{\tau}v_n\left(\frac{r_0}{4},y\right)&=\tphi(y)\left[\partial_{\tau}\tu_n\left(\frac{r_0}{4},y\right)-\partial_{\tau}v_n \left( \frac{r_0}{4},y \right)\right].
\end{align*}
We have 
\begin{equation}
\label{un_tun}
|y|\geq \frac{5}{4}r_0-\eps_0\Longrightarrow \left( \ttu_n\left( \frac{r_0}{4},y\right), \partial_{\tau}\ttu_n\left( \frac{r_0}{4},y\right)\right)=\left(u_n\left( \frac{r_0}{4},y\right), \partial_{\tau}u_n\left( \frac{r_0}{4},y\right)\right). 
\end{equation} 
By the arguments of Steps 3 and 4 of \S \ref{SSS:Case1}, one shows that for large $n_0$, $\ttu_{n_0}$ scatters forward in time. Fixing such a large $n_0$, using \eqref{un_tun} and finite speed of propagation, we deduce that $u_n(\tau,y)$ coincides with a scattering solution for $|y|\geq \tau+r_0-\eps_0$, $\tau>r_0$. Going back to the original variables $(t,x)$ with $\tau=(1+s_n)t$, $y=s_nx$, we get that there exists a solution 
$u_{\lin,n_0}$ of the linear wave equation \eqref{lin_wave} such that
\begin{equation}
\label{linear_limit}
\lim_{t\to\infty} \int_{|x|\geq t+s_{n_0}(r_0-\eps_0-1)} |\nabla (u-u_{\lin,n_0})(t,x)|^2+|\partial_t (u-u_{\lin,n_0})(t,x)|^2\,dx=0.
\end{equation} 
Let $A<B$ and chose $n_0$ large so that $s_{n_0}(r_0-\eps_0-1)<A$. It is classical that the limit 
$$
\lim_{t\to+\infty} \int_{|x|\geq t+A} |\nabla (u_{\lin,n_0}-v_{\lin})(t,x)|^2+|\partial_t (u_{\lin,n_0}-v_{\lin})(t,x)|^2\,dx
$$
exists. This shows by \eqref{linear_limit} and the definition of $v_{\NL}$ that 
\begin{multline}
\label{expr_I}
\III(A)=
\lim_{t\to+\infty} \int_{|x|\geq t+A} |\nabla (u-v_{\NL})(t,x)|^2+|\partial_t (u-v_{\NL})(t,x)|^2\,dx\\=\lim_{t\to+\infty}
\int_{|x|\geq t+A} |\nabla (u_{\lin,n_0}-v_{\lin}(t,x)|^2+(\partial_t (u_{\lin,n_0}-v_{\lin})(t,x))^2\,dx 
\end{multline} 
exists. The fact that $\III(A)$ is nonincreasing is obvious. By the assumption \eqref{bound_intro}, $\III$ is bounded, which shows that $\III(A)$ has a limit as $A\to -\infty$.  By finite speed of propagation $\III(A)\to 0$ as $A\to +\infty$. This shows \eqref{lim_I}. Finally, in view of \eqref{lim_I}, to show that $\III$ is uniformly continuous, it is sufficient to check that it is continuous. This last fact follows from the formula:
\begin{equation*}
\III(B)-\III(A)=\lim_{t\to+\infty} 
\int_{A\leq |x|\leq B} |\nabla (u_{\lin,n_0}-v_{\lin})(t,x)|^2+(\partial_t (u_{\lin,n_0}-v_{\lin})(t,x))^2\,dx  
\end{equation*}
and classical properties of the linear wave equation.
\end{proof}
As in the proof of Proposition \ref{P:no_ss} we distinguish three cases: $r_0>0$ and for all $j\geq 2$, $|\ttau_j|\leq r_0$ (Case $1$); there exists $j\geq 2$ such that $|\ttau_j|>r_0$ (Case $2$); $r_0=0$ and $\ttau_j=0$ for all $j\geq 2$ (Case $3$).
\subsubsection{Sketch of proof for Case 1}
In this case, we define $\tu_n$ and $\ttu_n$ as in \S \ref{SSS:Case1} and in the proof of Lemma \ref{L:I(A)}. Arguing exactly as in \S \ref{SSS:Case1}, one can show that there exists $\eta_1>0$ such that the following holds for all sequences $\{\sigma_n\}_n$ in $\left[-1,-\frac{1}{2}\right]$ or for all sequences $\{\sigma_n\}_n$ in $\left[\frac{1}{2},+\infty\right)$.
\begin{equation}
\label{channel_gl}
\limsup_{n\to\infty}
\int_{r_0-\eps_0+|\sigma_n|\leq |y|\leq r_0+|\sigma_n|} |\nabla_{\tau,y}(u_n-v_n)(\sigma_n,y)|^2\,dy\geq \eta_1.
\end{equation} 
If \eqref{channel_gl} holds for all sequences in $\left[-1,-\frac{1}{2}\right]$, one easily reaches a contradiction using \eqref{channel_gl} with $\sigma_n=-1$ for all $n$.

Assume that \eqref{channel_gl} holds for all sequences $\{\sigma_n\}_n$ in $\left[\frac{1}{2},+\infty\right)$. Going back to the original variables $(t,x)$, we get that the following holds for all sequences $\{\theta_n\}_n$ such that $\theta_n\geq \frac 32 s_n$:
\begin{equation}
\label{channel_bis} 
\limsup_{n\to\infty} \int_{\theta_n+s_n(r_0-\eps_0-1) \leq |x|\leq \theta_n+s_n(r_0-1)} |\nabla_{t,x}(u-v_{\NL})(\theta_n,x)|^2\geq \eta_1.
\end{equation} 

Noting that $s_n(r_0-1)\to -\infty$, we get by \eqref{lim_I} that for $n$ large enough
$$ \lim_{t\to\infty} \int_{t+s_n(r_0-\eps_0-1) \leq |x|\leq t+s_n(r_0-1)} |\nabla_{t,x}(u-v_{\NL})(t,x)|^2\,dx=\III\big(s_n(r_0-\eps_0-1)\big)-\III\big(s_n(r_0-1)\big)\leq \frac{\eta_1}{4}.$$
Thus if $\theta_n\geq \frac{3}{2}s_n$ is large enough,
$$ \int_{\theta_n+s_n(r_0-\eps_0-1) \leq |x|\leq \theta_n+s_n(r_0-1)} |\nabla_{t,x}(u-v_{\NL})(\theta_n,x)|^2\,dx\leq \frac{\eta_1}{2}.$$
Finally we have obtained that for all large $n$, there exists $\theta_n\geq \frac{3}{2}s_n$ such that
$$ \limsup_{n\to\infty} \int_{\theta_n+s_n(r_0-\eps_0-1) \leq |x|\leq \theta_n+s_n(r_0-1)} |\nabla_{t,x}(u-v_{\NL})(\theta_n,x)|^2\leq \frac{\eta_1}{2},$$
contradicting \eqref{channel_bis}. This concludes the proof of Case $1$.

\subsubsection{Sketch of proof for Case 2}
Assume that there exists $j\geq 2$ such that $|\ttau_j|>r_0$. As in \ref{SSS:Case2}, we let $\tr_0=\sup_{j\geq 2} |\ttau_j|>r_0$, and chose $\eps_0<\frac{\tr_0-r_0}{100}$. Thus there exists $j_0\geq 2$ such that $\tr_0-\eps_0<|\ttau_{j_0}|\leq \tr_0$. Assume that $\ttau_{j_0}>0$: the proof in the case $\ttau_{j_0}<0$ is very close to the proof of Case 2 in the finite time blow-up setting and we omit it. 

As in \S \ref{SSS:Case2} above, we let $\varphi\in C^{\infty}(\RR^3)$ be a radial function such that $\varphi(y)=1$ if $|y|\geq \tr_0-\eps_0$ and $\varphi(y)=0$ if $|y|\leq \tr_0-2\eps_0$. By the same argument as in \S \ref{SSS:Case1} and \S \ref{SSS:Case2}, we construct a solution $\ttu_n$ of \eqref{CP} such that
\begin{equation}
\label{FSP_global2}
\Big(\tau\geq \frac{\tr_0}{4}\text{ and } |y|\geq \tau+\tr_0-\eps_0\Big)\Longrightarrow \ttu_n(\tau,y)=u_n(\tau,y).
\end{equation}
and, for $\tau\geq \frac{\tr_0}{4}$, 
\begin{multline}
\label{expansion_global2}
 \left(\ttu_n(\tau,y),\partial_{\tau}\ttu_n(\tau,y)\right)-\left( v_n(\tau,y),\partial_{\tau}v_n(\tau,y) \right)\\
=\sum_{\substack{j=2,\ldots,J\\ \ttau_j>0}} \varphi(\ttau_j)\left(U^j_{\lin,n}(\tau,y),\partial_{\tau}U^j_{\lin,n}(\tau,y)\right)+\ttw_{\lin,n}^J(\tau,y)+\tteps_{n}^J(\tau,y),
\end{multline}
where $\ttw_n^J$ and $\tteps_n^J$ satisfies respectively \eqref{ttwnJ} and \eqref{ttepsnJ}. The rest of the proof is quite different from the type II blow-up setting and we will give more details. We will show:
\begin{multline}
\label{absurd_to_show}
 \lim_{J\to+\infty}\limsup_{n\to+\infty} \int_{\RR^3} \left(\sum_{\substack{j=2\ldots J\\ \ttau_j>0}}\int \varphi(\ttau_j)\nabla U_{\lin,n}^j(-1,y)+\nabla w_n^J(-1,y)\right)\cdot \nabla U^{j_0}_{\lin,n}(-1,y)\,dy\\
+\int_{\RR^3} \left(\sum_{\substack{j=2\ldots J\\ \ttau_j>0}}\int \varphi(\ttau_j)\partial_t U_{\lin,n}^j(-1,y)+\partial_t w_n^J(-1,y)\right) \partial_t U^{j_0}_{\lin,n}(-1,y)\,dy=0,
\end{multline}
an obvious contradiction (the limit is also equal to $\int_{\RR^3} |\nabla U^{j_0}_{\lin}(0)|^2+(\partial_tU^{j_0}_{\lin}(0))^2$).
Let $\delta>0$. Fix a large $J$, then a large $n_0$. Let
$$ \ttu(t,x)=\frac{1}{s_{n_0}^{1/2}}\ttu_{n_0}\left(\frac{t}{s_{n_0}}-1,\frac{x}{s_{n_0}}\right).$$
Note that $\ttu$ depends on $n_0$. We can rewrite \eqref{expansion_global2} (with $n=n_0$) as 
\begin{multline}
 \label{expansion_global2b}
\left(\ttu,\partial_{t}\ttu\right)(t,x)-\left( v_{\NL},\partial_{t}v_{\NL}\right)(t,x)\\
=\sum_{\substack{j=2,\ldots,J\\ \ttau_j>0}} \varphi(\ttau_j)\left(\frac{1}{s_{n_0}^{1/2}}U^j_{\lin,n_0},\frac{1}{s_{n_0}^{3/2}}\partial_{\tau}U^j_{\lin,n_0}\right)\left(\frac{t}{s_{n_0}}-1,\frac{x}{s_{n_0}}\right)\\
+\left(\ttw_{\lin,n_0}^J,\partial_{\tau}\ttw_{\lin,n_0}^J\right)\left(\frac{t}{s_{n_0}}-1,\frac{x}{s_{n_0}}\right)+\left(\tteps_{n_0}^J,\partial_{\tau}\tteps_{n_0}^J\right)\left(\frac{t}{s_{n_0}}-1,\frac{x}{s_{n_0}}\right).
\end{multline}
By \eqref{FSP_global2},
\begin{equation}
\label{FSP_global2b}
 \left(t\geq s_{n_0}\left(\frac{\tr_0}{4}+1\right)\text{ and } |x|\geq t+(\tr_0-\eps_0-1)s_{n_0}\right)\Longrightarrow\ttu(t,x)=u(t,x)
\end{equation} 
After extraction of a subsequence in $p$, we can assume that $\Sbf_{\lin}(-s_p)\left(\ttu(s_p),\partial_{t}\ttu(s_p)\right)$ has a weak limit $(\ttv_0,\ttv_1)$ in $\hdot\times L^2$ as $p\to \infty$. Thus the following limit exists:
\begin{align*}
 \ell(n_0)&=\lim_{p\to\infty} \Bigg\langle\Sbf_{\lin}(-s_p)\left((\ttu-u)(s_p),(\partial_t\ttu-\partial_tu)(s_p)\right),\\
&\qquad \qquad \qquad \left(\frac{1}{s_{n_0}^{1/2}} U_{\lin,n_0}^{j_0},
\frac{1}{s_{n_0}^{3/2}} \partial_{\tau}U_{\lin,n_0}^{j_0}\right)\left(-1,\frac{\cdot}{s_{n_0}}\right)\Bigg\rangle_{\hdot\times L^2}\\
&=\Bigg\langle (\ttv_0,\ttv_1)-(v_0,v_1),
\left(\frac{1}{s_{n_0}^{1/2}} U_{\lin,n_0}^{j_0},
\frac{1}{s_{n_0}^{3/2}} \partial_{\tau}U_{\lin,n_0}^{j_0}\right)\left(-1,\frac{\cdot}{s_{n_0}}\right)\Bigg\rangle_{\hdot\times L^2}
\end{align*}
In view of \eqref{expansion_global2b}, if $J$ and $n_0$ are large,
\begin{multline*}
\Bigg\|\left( \ttv_0,\ttv_1 \right)-\left( v_0,v_1 \right)-\sum_{\substack{j=1\ldots J\\ \ttau_j>0}}\varphi(\ttau_j)\left( \frac{1}{s_{n_0}^{1/2}}U^{j}_{\lin,n_0}, \frac{1}{s_{n_0}^{3/2}}\partial_t U^{j}_{\lin,n_0}\right)\left( -1,\frac{\cdot}{s_{n_0}}\right)\\
 -\left( \frac{1}{s_{n_0}^{1/2}}\ttw_{n_0}^{J},\frac{1}{s_{n_0}^{3/2}}\partial_t \ttw_{n_0}^{J}\right)\left( -1,\frac{\cdot}{s_{n_0}} \right) \Bigg\|_{\hdot\times L^2}\leq \eps.
\end{multline*}
Hence, if $J$ and $n_0$ are large,
\begin{multline}
\label{est1_ell}
\Bigg|\Bigg\langle  
\sum_{\substack{j=1\ldots J\\ \ttau_j>0}}\varphi(\ttau_j)\left(U^{j}_{\lin,n_0}, \partial_t U^{j}_{\lin,n_0}\right)\left( -1,\cdot\right)
 -\left(\ttw_{n_0}^{J},\partial_t \ttw_{n_0}^{J}\right)\left( -1,\cdot\right),\\ \left(U_{\lin,n_0}^{j_0},
\partial_{\tau}U_{\lin,n_0}^{j_0}\right)\left(-1,\cdot\right)\Bigg\rangle_{\hdot\times L^2}
-\ell(n_0)\Bigg|\leq C\eps 
\end{multline}
On the other hand,
$$ \ell(n_0)=\lim_{p\to\infty} \Bigg\langle \left((\ttu-u)(s_p),(\partial_t\ttu-\partial_tu)(s_p)\right),
\left(\frac{1}{s_{n_0}^{1/2}} U_{\lin,n_0}^{j_0},
\frac{1}{s_{n_0}^{3/2}} \partial_{\tau}U_{\lin,n_0}^{j_0}\right)\left(\frac{s_{p}}{s_{n_0}}-1,\frac{\cdot}{s_{n_0}}\right)\Bigg\rangle_{\hdot\times L^2}.$$
Using \eqref{FSP_global2b}, we can write this last scalar product 
as
\begin{equation}
\label{loc_integral}
\int_{|x|\leq s_p+(\tr_0-\eps_0-1)s_{n_0}} \nabla_{t,x}(\ttu-u)(s_p,x)\cdot \frac{1}{(\lambda_{j_0,n_0}s_{n_0})^{3/2}}\nabla_{\tau,y}U_{\lin}^{j_0}\left(\frac{s_{p}-s_{n_0}(1+\tau_{j_0,n_0})}{\lambda_{j_0,n_0}s_{n_0}},\frac{x}{\lambda_{j_0,n_0}s_{n_0}}\right)\,dx.
\end{equation} 
If $n_0$ is large, then by the definition of $\ttau_{j_0}$ and the fact that $\lim_{n\to\infty}-\tau_{j_0,n}=\ttau_{j_0}>\tr_0-\eps_0$, we get that 
$$-s_{n_0}(1+\tau_{j_0,n_0})> (\tr_0-\eps_0-1)s_{n_0}.$$
Furthermore, we can chose $\lambda_{j_0,n_0}$ arbitrary small (taking a larger $n_0$ if necessary).
By Lemma \ref{L:lin_odd} (and using that $(\ttu(t),\partial_t\ttu(t))$ can be bounded independently of $n_0$ and $t\geq 0$ in $\hdot\times L^2$), we conclude from \eqref{loc_integral} that for large $n_0$, 
$$ |\ell(n_0)|\leq \eps.$$
Combining with \eqref{est1_ell} we get the desired estimate \eqref{absurd_to_show}. The proof of case 2 is complete.
\subsubsection{Sketch of proof for Case 3}
It remains to treat the case where $r_0=0$ (i.e. $U^1=0$) and $\ttau_j=0$ for all $j\geq 2$. By \eqref{absurd_global}, there exists $c_0>0$ such that for all $J\geq 2$,
\begin{equation}
\label{absurd_case3}
 \limsup_{n\to \infty}\int_{c_0\leq |y|\leq 1} |\nabla w_{0,n}^J|^2+(w_{1,n}^J)^2\,dy\geq c_0.
\end{equation} 
\EMPH{Step 1} We fix a large $J\geq 2$ and show:
\begin{equation}
\label{energy_in}
 \exists \eps_0>0 ,\; \forall B>0,\quad \limsup_{n\to\infty} \int_{c_0\leq |y|\leq 1-\frac{B}{s_n}} \left|\nabla w_{0,n}^J\right|^2+(w_{1,n}^J)^2\geq \eps_0.
\end{equation}
If not,
\begin{equation}
\label{energy_in_absurd}
 \forall \eps>0 ,\; \exists B(\eps)>0,\quad \limsup_{n\to\infty} \int_{c_0\leq |y|\leq 1-\frac{B(\eps)}{s_n}} \left|\nabla w_{0,n}^J\right|^2+(w_{1,n}^J)^2\leq \eps.
\end{equation}
We first prove that there exists a constant $C_1>0$, independent of $\eps$, such that:
\begin{equation}
\label{small_scalar_v}
\limsup_{n\to+\infty} \left|\int_{|y|\geq 1-\frac{B(\eps)}{s_n}} \left(\nabla v_{n}(0)\cdot w_{0,n}^J+\partial_{\tau}v_{n}(0)w_{1,n}^J\right)\,dy\right|\leq C_1\sqrt{\eps}
\end{equation} 
and
\begin{equation}
\label{small_scalar_u}
\limsup_{n\to+\infty} \left|\int_{|y|\geq 1-\frac{B(\eps)}{s_n}} \left(\nabla u_{n}(0)\cdot w_{0,n}^J+\partial_{\tau}u_{n}(0)w_{1,n}^J\right)\,dy\right|\leq C_1\sqrt{\eps}.
\end{equation} 
Indeed, 
\begin{equation}
\label{decompo_int}
\int_{|y|\geq 1-\frac{B(\eps)}{s_n}} \left(\nabla v_{n}(0)\cdot w_{0,n}^J+\partial_{\tau}v_{n}(0)w_{1,n}^J\right)\,dy=\int_{\RR^3}\ldots-\int_{|y|\leq c_0}\ldots-\int_{c_0\leq |y|\leq 1-B(\eps)/s_n}\ldots
\end{equation} 
The first term in \eqref{decompo_int} goes to $0$ by the property \eqref{weak_CV_wJ} of $w_n^J$. The second term goes to zero using that $v_n(0,y)=s_n^{1/2} v_{\lin}(s_n,s_ny)+o_n(1)$ and Lemma \ref{L:lin_odd}. Finally, by \eqref{energy_in_absurd} and Cauchy-Schwarz inequality, the third term is bounded as $n\to\infty$ by
$$ \sqrt{\eps}\limsup_{n\to\infty} \sqrt{\int |\nabla v_{0,n}|^2+(v_{1,n})^2}=\sqrt{\eps} \sqrt{\int |\nabla v_{\lin}(0)|^2+(\partial_t v_{\lin}(0))^2},$$
which shows \eqref{small_scalar_v}.

To show \eqref{small_scalar_u} , note that for $n_0$ large, the solution $u_{\lin,n_0}$ to the linear wave equation introduced in the proof of Lemma \ref{L:I(A)} satisfies
$$\lim_{n\to \infty} \int_{|x|\geq s_n-B(\eps)} |\nabla_{t,x} u(s_n,x)-\nabla_{t,x} u_{\lin,n_0}(s_n,x)|^2\,dx=0.$$
The proof of \eqref{small_scalar_u} works the same as the proof of \eqref{small_scalar_v}, provided  the conserved linear energy of $u_{\lin,n_0}$ satisfies
$$ \frac 12 \int |\nabla u_{\lin,n_0}(t,x)|^2+\frac 12\int |\partial_t u_{\lin,n_0}(t,x)|^2\leq C_2,$$
where the constant $C_2>0$ is independent of $n_0$. This follows easily from the definition of $u_{\lin,n_0}$ in the proof of Lemma \ref{L:I(A)}, which concludes the proof of \eqref{small_scalar_u}.

We next notice that the assumption that $\ttau_j=0$ for all $j\geq 2$ implies that for all $J\geq 2$,
\begin{multline*}
\int_{|y|\geq 1-B(\eps)/s_n} \nabla u_{n}(0)\cdot \nabla w_{0,n}^J+\partial_{\tau} u_{n}(0)\, w_{1,n}^J\\=
\int_{|y|\geq 1-B(\eps)/s_n} \nabla v_{n}(0)\cdot \nabla w_{0,n}^J+\partial_{\tau} v_{n}(0)\,w_{1,n}^J+\int_{|y|\geq 1-B(\eps)/s_n} |\nabla w_{0,n}^J|^2+(w_{1,n}^J)^2+o_n(1). 
\end{multline*}
Combining with \eqref{small_scalar_v} and \eqref{small_scalar_u} we see that there exists a constant $C>0$ such that for all $\eps>0$,
$$\limsup_{n\to \infty}\int_{|y|\geq 1-B(\eps)/s_n} |\nabla w_{0,n}^J|^2+(w_{1,n}^J)^2\leq C\sqrt{\eps}.$$
Together with \eqref{energy_in_absurd}, we get that 
$$\lim_{n\to\infty}\int_{c_0\leq |y|} |\nabla w_{0,n}^J|^2+(w_{1,n}^J)^2=0,$$
 a contradiction with \eqref{absurd_case3}.

\EMPH{Step 2}
By Step 1, there exist (after extraction) $\eps_0>0$ and a sequence $B_n\to+\infty$ such that $B_n/s_n\to 0$ and
$$ \liminf_{n\to \infty} \int_{c_0\leq |y|\leq 1-B_n/s_n}|\nabla w_{0,n}^J|^2+(w_{1,n}^J)^2\,dy\geq \eps_0.$$
Chosing $J$ large and arguing as in \S \ref{SSS:Case3}, we deduce that the following inequality holds for large $n$ and for all $\tau>0$ or for all $\tau<0$:
\begin{equation}
 \label{ultimate_channel}
\int_{c_0+|\tau|}^{1-B_n/s_n+|\tau|}\Big[(\partial_r(r u_n-rv_n))^2(\tau,r)+(\partial_{\tau}(r u_n-rv_n))^2(\tau,r)\Big]\,dr\geq \eps_0/8.
\end{equation} 
If it holds for all $\tau<0$, we get an easy contradiction taking $\tau=-1$ and letting $n$ goes to infinity.

Assume that \eqref{ultimate_channel}  holds for all $\tau>0$. Following the proof of Lemma \ref{L:I(A)}, we see (after extraction of subsequences) that there exists a continuous non-increasing positive function $\tIII(A)$ such that
$$ \lim_{t\to +\infty}\int_{t+A}^{+\infty}(\partial_r( ru-rv_{\lin}))^2+(\partial_t( ru-rv_{\lin}))^2\,dr=\tIII(A),$$
and
$$\lim_{A\to +\infty} \tIII(A)=0,\quad \lim_{A\to -\infty} \tIII(A)\geq 0.$$
By \eqref{ultimate_channel}, for large $n$,
$$ \tIII\left((c_0-1)s_n\right)-\tIII(-B_n)\geq \eps_0/8,$$
which yields a contradiction letting $n\to+\infty$.
\subsubsection{Existence of a linear component}
We conclude this section by giving a corollary to Proposition \ref{P:global_no_ss}:
\begin{corol}
\label{C:linear_component}
 Let $u$ be a radial solution of \eqref{CP} defined on $[0,+\infty)$ such that
$$ \sup_{t\geq 0} \|\nabla u(t)\|_{L^2}+\|\partial_tu(t)\|_{L^2}<+\infty.$$
Then there exists $(v_0,v_1)\in \hdot\times L^2$ such that:
$$ S_{\lin}(-t)(u(t),\partial_t u(t))\xrightharpoonup[t\to\infty]{}(v_0,v_1)$$
weakly in $\hdot\times L^2$. Furthermore, for all $c_0>0$
$$\lim_{t\to\infty} \int_{|x|\geq c_0t} |\nabla u(t)-\nabla v_{\lin}(t)|^2+(\partial_tu(t)-\partial_tv(t))^2\,dx=0,$$
where $v_{\lin}$ is the solution of the linear equation \eqref{lin_wave} with initial data $(v_0,v_1)$.
\end{corol}
\begin{proof}
 Let $t_n,t_n'$ be two sequences of time such that $t_n,t_n'\to +\infty$. Assume that there exists $(v_0,v_1)\in \hdot \times L^2$, $(v_0',v_1')\in \hdot\times L^2$ such that
$$ S_{\lin}(-t_n)(u(t_n),\partial_tu(t_n))\xrightharpoonup[n\to\infty]{}(v_0,v_1),\text{ and }S_{\lin}(-t_n)(u(t_n'),\partial_tu(t_n'))\xrightharpoonup[n\to\infty]{}(v_0',v_1').$$
In view of Proposition \ref{P:global_no_ss}, it is sufficient to check that $(v_0,v_1)=(v_0',v_1')$. 

We denote by $v_{\lin}(t)$ (respectively $v_{\lin}'(t)$) the solution of the linear equation \eqref{lin_wave} with initial data $(v_0,v_1)$ (respectively $(v_0',v_1')$). Fix $\eps>0$. By Lemma \ref{L:lin_odd}, there exists $R>0$ such that
\begin{equation}
 \label{loc_linear}
\limsup_{t\to+\infty} \int_{|x|\leq t-R} \left[|\nabla v_{\lin}(t,x)|^2+(\partial_tv_{\lin}(t,x))^2+|\nabla v_{\lin}'(t,x)|^2+(\partial_tv_{\lin}'(t,x))^2\right]\,dx\leq \eps.
\end{equation} 
Furthermore, by Proposition \ref{P:global_no_ss}, 
$$ \lim_{n\to\infty} 
\int_{|x|\geq t_n-R} |\nabla v_{\lin}(t_n)-\nabla u(t_n)|^2+(\partial_t v_{\lin}(t_n)-\partial_t u(t_n))^2\,dx=0.
$$
Thus by finite speed of propagation
\begin{equation}
 \label{vL_u1} 
\lim_{t\to+\infty} \int_{|x|\geq t-R} |\nabla v_{\lin}(t)-\nabla u(t)|^2+(\partial_t v_{\lin}(t)-\partial_t u(t))^2\,dx=0.
\end{equation} 
The same argument gives the analog of \eqref{vL_u1} with $v_{\lin}$ replaced by $v_{\lin}'$. Hence:
$$\lim_{t\to+\infty} \int_{|x|\geq t-R} |\nabla v_{\lin}(t)-\nabla v_{\lin}'(t)|^2+(\partial_t v_{\lin}(t)-\partial_t v_{\lin}'(t))^2\,dx=0.
$$
Combining with \eqref{loc_linear}, we get:
\begin{equation*}
 \limsup_{t\to +\infty} \int_{\RR^3} |\nabla v_{\lin}(t)-\nabla v_{\lin}'(t)|^2+(\partial_t v_{\lin}(t)-\partial_t v_{\lin}'(t))^2\,dx\leq \eps.
\end{equation*}
As $\eps>0$ is arbitrary small, we get:
\begin{equation*}
 \lim_{t\to +\infty} \int_{\RR^3} |\nabla v_{\lin}(t)-\nabla v_{\lin}'(t)|^2+(\partial_t v_{\lin}(t)-\partial_t v_{\lin}'(t))^2\,dx=0.
\end{equation*}
By conservation of the linear energy, $(v_0,v_1)=(v_0',v_1')$, concluding the proof.
\end{proof}

\section{Expansion as a sum of rescaled stationary solution}
\label{S:expansion}
In this section we prove Theorems \ref{T:bup_sum} and \ref{T:global_sum}. 
\subsection{Expansion up to a dispersive term}
We first show that Proposition \ref{P:no_ss} (respectively Proposition \ref{P:global_no_ss}) implies:
\begin{corol}
\label{C:devt}
Let $u$ be a radial solution of \eqref{CP} such that $T_+(u)=1$. Assume that
$$ \sup_{0\leq t<1}\|(u(t),\partial_tu(t))\|_{\hdot\times L^2}<\infty.$$
 Then there exist $J_0\geq 1$, $\iota_1,\ldots,\iota_{J_0}\in\{\pm 1\}$, sequences $t_n\to 1$,  $\lambda_{1,n},\ldots,\lambda_{J_0,n}>0$ with
\begin{equation}
\label{ordre_lambda}
 \lambda_{1,n}\ll \lambda_{2,n}\ll\ldots\ll \lambda_{J_0,n}\ll 1-t_n,
\end{equation} 
and a sequence $\{w_{0,n}\}_n$, $w_{0,n}\in \hdot$ such that
\begin{equation}
\label{devt_corol}
\lim_{n\to \infty}\left\|
\left(u(t_n)-v_0-\sum_{j=1}^{J_0} \frac{\iota_j}{\lambda_{j,n}^{1/2}}W\left(\frac{\cdot}{\lambda_{j,n}}\right)-w_{0,n},\partial_t u(t_n)-v_1\right)\right\|_{\hdot\times L^2}=0,
\end{equation} 
where $(v_0,v_1)$ is defined in the beginning of Section \ref{S:SSregion} and, denoting by $w_n$ the solution of the linear equation \eqref{lin_wave} with initial data $\left(w_{0,n},0\right)$
\begin{equation}
 \label{dispersive_wn}
\lim_{n\to \infty}\|w_n\|_{L^8(\RR\times \RR^3)}=0.
\end{equation} 
\end{corol}
\begin{corol}
\label{C:global_devt}
 Let $u$ be a radial solution of \eqref{CP} defined on $[0,+\infty)$ such that
$$ \sup_{t\geq 0} \|(u(t),\partial_tu(t))\|_{\hdot\times L^2}<\infty.$$
Then there exist $J_0\geq 0$, $\iota_1,\ldots,\iota_j\in \{\pm 1\}$, sequences $s_n\to +\infty$ and $\lambda_{1,n},\ldots,\lambda_{J_0,n}>0$ with
\begin{equation}
\label{ordre_lambda_global}
 \lambda_{1,n}\ll \lambda_{2,n}\ll\ldots\ll \lambda_{J_0,n}\ll s_n,
\end{equation} 
and a sequence $\{w_{0,n}\}_n$, $w_{0,n}\in  \hdot$ such that
\begin{equation}
\label{global_devt_corol}
\lim_{n\to \infty}\left\|
\left(u(s_n)-v_{\lin}(s_n)-\sum_{j=1}^{J_0} \frac{\iota_j}{\lambda_{j,n}^{1/2}}W\left(\frac{\cdot}{\lambda_{j,n}}\right)-w_{0,n},\partial_t u(s_n)-\partial_t v_{\lin}(s_n)\right)\right\|_{\hdot\times L^2}=0,
\end{equation} 
where $v_{\lin}$ is defined in Corollary \ref{C:linear_component} and the solution $w_n$ of \eqref{lin_wave} with initial data $\left(w_{0,n},0\right)$ satisfies \eqref{dispersive_wn}.
\end{corol}
Corollary \ref{C:devt} follows from Proposition 5.1 in \cite{DuKeMe11a} and  Proposition \ref{P:no_ss}, and we omit the proof. The proof of Corollary \ref{C:global_devt} is very similar. We sketch it for completeness. We divide the proof into a few lemmas.
\begin{lemma}
 \label{L:average}
Let $u$ be as in Corollary \ref{C:global_devt}. Then
$$ \lim_{T\to+\infty} \frac{1}{T}\int_0^T \int_{\RR^3} (\partial_t u-\partial_t v_{\lin})^2(t,x)\,dx\,dt=0.$$
\end{lemma}
\begin{proof}
Let $\varphi\in C_0^{\infty}(\RR^3)$, radial such that $\varphi(x)=1$ if $|x|\leq \frac{3}{2}$ and $\varphi(x)=0$ if $|x|\geq 2$. Let
\begin{equation}
\label{def_Z}
 Z(t)=\frac{1}{2}\int_{\RR^3}\varphi\left(\frac{x}{t}\right)\left(u\partial_t u-v_{\lin}\partial_tv_{\lin}\right)\,dx+
\int_{\RR^3}\varphi\left(\frac{x}{t}\right)\left(x\cdot \nabla u\partial_t u-x\cdot\nabla v_{\lin}\partial_tv_{\lin}\right)\,dx.
\end{equation}
Using that $u$ satisfies \eqref{CP} and $v_{\lin}$ satisfies \eqref{lin_wave}, we get (cf e.g. \cite[Claim 2.11]{DuKeMe10P})
\begin{equation}
 \label{Z'}
Z'(t)=-\int \left((\partial_tu)^2-(\partial_t v)^2\right)\,dx +R(t),
\end{equation} 
where
$$ |R(t)|\leq C\int_{|x|\geq \frac{5}{4}|t|} \left(|\nabla u|^2+(\partial_tu)^2+u^6+|\nabla v_{\lin}|^2+(\partial_t v_{\lin})^2\right)\,dx$$
and thus by finite speed of propagation:
\begin{equation}
\label{R(t)0}
 \lim_{t\to +\infty} R(t)=0.
\end{equation}
Furthermore,
$$ \int \left((\partial_tu)^2-(\partial_t v_{\lin})^2\right)\,dx =\int \left(\partial_t u-\partial_t v_{\lin}\right)^2\,dx+2\int\left(\partial_t u-\partial_t v_{\lin}\right)\partial_t v_{\lin}\,dx.$$
By Lemma \ref{L:lin_odd},
$$ \lim_{t\to+\infty} \int_{|x|\leq \frac{1}{2}|t|} (\partial_t v_{\lin}(t))^2=0.$$
By Corollary \ref{C:linear_component},
$$ \lim_{t\to +\infty} \int_{|x|\geq \frac{1}{2}|t|}(\partial_t u-\partial_t v_{\lin})^2\,dx=0.$$
Combining, we get
$$ \int \big((\partial_tu)^2-(\partial_t v)^2\big)\,dx -\int (\partial_tu-\partial_tv_{\lin}))^2\,dx\underset{t\to +\infty}{\longrightarrow}0,$$
and hence
\begin{equation}
\label{asymptoticZ'}
\lim_{t\to+\infty}
\left|Z'(t) +\int (\partial_t u-\partial_tv_{\lin})^2\,dx\right|=0.
\end{equation} 
We next show
\begin{equation}
 \label{Ztt0}
\lim_{t\to+\infty}\frac{Z(t)}{t}=0.
\end{equation} 
Assuming \eqref{Ztt0} and integrating \eqref{asymptoticZ'} between $0$ and $T$, we get the conclusion of the Lemma.

We have:
\begin{equation}
\label{(i)and(ii)}
\int_{\RR^3}\varphi\left(\frac{x}{t}\right)(u\partial_t u-v_{\lin}\partial_tv_{\lin})=\underbrace{\int_{\RR^3} \varphi\left(\frac{x}{t}\right)u(\partial_tu-\partial_t v_{\lin})}_{(i)}+\underbrace{\int_{\RR^3} \varphi\left(\frac{x}{t}\right) (u-v_{\lin})\partial_tv_{\lin}}_{(ii)}.
\end{equation}
Let $\eps_0>0$ and write: 
\begin{equation}
\label{(i)}
 (i)=\int_{|x|\leq \eps_0t} \varphi\left(\frac{x}{t}\right)u(\partial_tu-\partial_tv_{\lin})+\int_{|x|\geq \eps_0t} \varphi\left(\frac{x}{t}\right)u(\partial_tu-\partial_tv_{\lin}).
\end{equation}
Using Hardy's inequality, the first term in \eqref{(i)} can be bounded from above by 
$$\int_{|x|\leq \eps_0t} |x|\frac{1}{|x|}|u|\,|\partial_tu-\partial_tv_{\lin}|\leq C_0\eps_0t,$$
where $C_0$ depends only on $\sup_{t\geq 0} \left[\|(u,\partial_t u)\|_{\hdot\times L^2}+\|(v_{\lin},\partial_t v_{\lin})\|_{\hdot\times L^2}\right]$. Similarly, the second term of \eqref{(i)} can be bounded as follows
$$t\int_{|x|\geq \eps_0t} \frac{|x|}{t}\varphi\left(\frac{x}{t}\right)\frac{1}{|x|}\Big|u(\partial_tu-\partial_tv_{\lin})\Big|\leq C_0t\sqrt{\int_{|x|\geq \eps_0t} |\partial_tu-\partial_tv_{\lin}|^2}.$$
By Corollary \ref{C:linear_component}, we get
$$ \lim_{t\to+\infty} \frac{1}{t}\left|\int_{|x|\geq \eps_0t}\varphi\left(\frac{x}{t}\right)u(\partial_t u-\partial_tv_{\lin})\,dx\right|=0.$$
Combining the preceding estimates, we get that the term $(i)$ in \eqref{(i)and(ii)} satisfies $\limsup_{t\to+\infty}\frac{1}{t}|(i)|\leq\eps_0$, and thus, as $\eps_0$ is arbitrary small,
$$\lim_{t\to+\infty}\frac{1}{t}(i)=0.$$
By the same argument on $(ii)$, we get \eqref{Ztt0}, which concludes the proof of the Lemma.
\end{proof}
\begin{lemma}
\label{L:special_sequence}
Let $u$ be as in Corollary \ref{C:global_devt}.
 Then there exists a sequence $s_n\to+\infty$ such that
\begin{gather}
\label{special_sequence1}
 \lim \int_{\RR^3} (\partial_tu(s_n,x)-\partial_t v_{\lin}(s_n,x))^2\,dx=0\\
\label{special_sequence2}
\lim_{n\to+\infty} \sup_{\lambda>0} \frac{1}{\lambda}\int_{s_n}^{s_n+\lambda} \int_{\RR^3} (\partial_tu(t,x)-\partial_t v_{\lin}(t,x))^2\,dx\,dt=0.
\end{gather}
\end{lemma}
Lemma \ref{L:special_sequence} follows from Lemma \ref{L:average} from a straightforward argument. We omit the proof, which is almost identical to the one of \cite[Corollary 5.3]{DuKeMe11a}. 

We next show Corollary \ref{C:global_devt}. Let $\{s_n\}_n$ be the sequence of times given by Lemma \ref{L:special_sequence}. 

 Consider a profile decomposition for the sequence $\big\{(u(s_n)-v_{\lin}(s_n),\partial_tu(s_n)-\partial_tv_{\lin}(s_n))\big\}_n$, with profiles $U_{\lin}^j$ and parameters $\{t_{j,n},\lambda_{j,n}\}$. As usual we can consider this as a profile decomposition for the sequence $(u(s_n),\partial_t u(s_n))$ with the extra profile $(v_{\lin}(s_n),\partial_t v_{\lin}(s_n))$.

By \eqref{special_sequence1} in Lemma \ref{L:special_sequence} and expansion \eqref{pythagore1b}, 
$$\lim_{n\to\infty} \left\|\partial_tU^j_{\lin}\left(\frac{-t_{j,n}}{\lambda_{j,n}}\right)\right\|_{L^2}=0.$$
This shows that the sequence $\left\{-\frac{t_{j,n}}{\lambda_{j,n}}\right\}_n$ is bounded if $U^j_{\lin}\neq 0$. Thus we can assume $t_{j,n}=0$ for all $j,n$.
By \eqref{special_sequence1} again, $\partial_t U^j(0)=0$ for all $j$. 

We next fix a small $\delta_0>0$ and reorder the profiles, so that there exists a $J_0\geq 1$ such that 
$$\|\nabla U^j_{0,\lin}\|_{L^2}^2\geq \delta_0\text{ for }j=1,\ldots,J_0\qquad \text{and}\qquad\|\nabla U^j_{0,\lin}\|_{L^2}^2<\delta_0\text{ for }j>J_0.$$
Note that we can always assume $J_0\geq 1$: if not, Proposition \ref{P:lin_NL} would imply that the solution $u$ scatters, giving the conclusion in this case.

We will show that $U^j=W$ (up to transformations) for $j\leq J_0$ and $U^j=0$ for $j>J_0$.

Fix $k_0\geq 1$, and consider a sequence $\{\tlambda_n\}_n$ of positive numbers such that:
\begin{gather}
\label{tlambda_n_1}
 \tlambda_n\ll \lambda_{k_0,n}\\
\label{tlambda_n_2}
\forall j\in\{1,\ldots,J_0\},\quad \lambda_{j,n}\ll \lambda_{k_0,n}\Longrightarrow \lambda_{j,n}\ll \tlambda_n.
\end{gather}
Let $\varphi \in C^{\infty}(\RR^3)$ such that $\varphi(x)=1$ if $|x|\geq \frac{3}{4}$ and $\varphi(x)=0$ if $|x|\leq \frac 12$. 
Consider 
\begin{equation*}
 \tu_{0,n}(x)=v_{\lin}(s_n,x)+\varphi\left(\frac{x}{\tlambda_n}\right)\sum_{j=1}^{J_0} \frac{1}{\lambda_{j,n}}U_{0,\lin}^j\left(\frac{x}{\lambda_{j,n}}\right)+w_{0,n}^{J_0}(x),\quad 
\tu_{1,n}(x)=\partial_tu(s_n,x).
\end{equation*}
Let $\{\tu_n\}$ be the sequence of solutions of \eqref{CP} such that 
$$\big(\tu_n(s_n),\partial_t\tu_n(s_n)\big)=(\tu_{0,n},\tu_{1,n}).$$
 Let $T>0$ be in the domain of definition of $U^{k_0}$. Noting that 
$$ \tu_{0,n}=v_{\lin}(s_n)+\sum_{\substack{j=1\ldots J_0\\\lambda_{j,n}\gg \lambda_{k_0,n}}} \frac{1}{\lambda_{j,n}^{1/2}}U^j_{0,\lin}\left(\frac{\cdot}{\lambda_{j,n}}\right)+\sum_{j=J_0+1\ldots J} \frac{1}{\lambda_{j,n}^{1/2}}U^j_{0,\lin}\left(\frac{\cdot}{\lambda_{j,n}}\right)+w_{0,n}^{J}+o_n(1)\text{ in }\hdot,$$
we get by Proposition \ref{P:lin_NL} that for $t\in[0,\lambda_{k_0,n}T)$,
\begin{multline*}
\tu_{n}(s_n+t,x)=v_{\lin}(s_n+t,x)+\sum_{\substack{j=1\ldots J_0\\\lambda_{j,n}\gg \lambda_{k_0,n}}} \frac{1}{\lambda_{j,n}^{1/2}}U^j\left(\frac{t}{\lambda_{j,n}},\frac{x}{\lambda_{j,n}}\right)\\
+\sum_{j=J_0+1\ldots J} \frac{1}{\lambda_{j,n}^{1/2}}U^j\left(\frac{t}{\lambda_{j,n}},\frac{x}{\lambda_{j,n}}\right)+w_n^{J}(t,x)+r_n^J(t,x),
\end{multline*}
where 
$$\lim_{J\to\infty}\limsup_{n\to\infty}\left[\sup_{t\in[0,\lambda_{k_0,n}T)} \left\|\big(r_n^J(t),\partial_tr_n^J(t)\big)\right\|_{\hdot\times L^2}\right]=0.$$
This implies that for all $\theta\in[0,T]$, 
$$\lambda_{k_0,n}^{3/2}(\partial_t \tu_n-\partial_t v_{\lin})\left(s_n+\lambda_{k_0,n}\theta,\lambda_{k_0,n}\cdot\right)\xrightharpoonup[n\to\infty]{}\partial_t U^{k_0}(\theta),$$
weakly in $L^2$. In particular, fixing $\eps_0>0$,
\begin{multline*}
\forall\theta\in[0,T],\quad \lim_{n\to\infty}\int_{|y|\geq \eps_0+|\theta|}\lambda_{k_0,n}^{3/2}(\partial_t \tu_n-\partial_t v_{\lin})\left(s_n+\lambda_{k_0,n}\theta,\lambda_{k_0,n}y\right)\partial_tU^{k_0}(\theta,y)\,dy\\
=\int_{|y|\geq \eps_0+|\theta|}\left(\partial_t U^{k_0}(\theta,y)\right)^2dy, 
\end{multline*}
and by the dominated convergence theorem,
\begin{multline}
\label{dominated} 
\lim_{n\to\infty}\int_0^T\int_{|y|\geq \eps_0+|\theta|}\lambda_{k_0,n}^{3/2}(\partial_t \tu_n-\partial_t v_{\lin})\left(s_n+\lambda_{k_0,n}\theta,\lambda_{k_0,n}y\right)\partial_tU^{k_0}(\theta,y)\,dy\,d\theta\\
=\int_0^T\int_{|y|\geq \eps_0+|\theta|}\left(\partial_t U^{k_0}(\theta,y)\right)^2dy\,d\theta. 
\end{multline} 
The left-hand side in \eqref{dominated} is equal to
\begin{multline*}
 \frac{1}{\lambda_{k_0,n}}\int_0^{\lambda_{k_0,n}T} \int_{|x|\geq \lambda_{k_0,n}\eps+t} (\partial_t\tu_n-\partial_tv_{\lin})(s_n+t,x)\frac{1}{\lambda_{k_0,n}^{3/2}}\partial_tU^k\left(\frac{t}{\lambda_{k_0,n}},\frac{x}{\lambda_{k_0,n}}\right)\,dx\,dt\\
=
\frac{1}{\lambda_{k_0,n}}\int_0^{\lambda_{k_0,n}T} \int_{|x|\geq \lambda_{k_0,n}\eps+t} (\partial_tu-\partial_tv_{\lin})(s_n+t,x)\,\frac{1}{\lambda_{k_0,n}^{3/2}}\partial_tU^k\left(\frac{t}{\lambda_{k_0,n}},\frac{x}{\lambda_{k_0,n}}\right)\,dx\,dt\underset{n\to+\infty}{\longrightarrow} 0
\end{multline*}
for large $n$. We used that for large $n$, $\lambda_{k_0,n}\eps\geq \tlambda_n$, and that by finite speed of propagation, $\tu(s_n+t,x)=u(s_n+t,x)$ for $|x|\geq |t| +\frac 34 \tlambda_n$. As a consequence, for all $\eps_0>0$, 
$$\int_0^T \int_{|y|\geq \eps_0+|\theta|} (\partial_t U^{k_0}(\theta,y))^2\,dy\,d\theta=0,$$
which shows
\begin{equation}
\label{zerodt}
\int_0^T \int_{|y|\geq |\theta|}(\partial_t U^{k_0}(\theta,y))^2\,dy\,d\theta=0. 
\end{equation} 
Recall that $W$ is (up to scaling and change of sign) the only nonzero $\hdot$, radial solutions of $\Delta f+f^5=0$ on $\RR^3$. Using \eqref{zerodt}, we deduce as in \cite{DuKeMe11a} (see the end of the proof of Proposition 5.1) that $U^{k_0}=0$ or $U^{k_0}=\pm \frac{1}{\mu_{k_0}^{3/2}}W\left(\frac{\cdot}{\mu_{k_0}}\right)$ for some $\mu_{k_0}>0$.
\subsection{End of the proof}
\label{SS:end_of_proof}
In this Subsection we conclude the proofs of Theorems \ref{T:bup_sum} and \ref{T:global_sum}.  The proofs are very similar, and we will focus on the proof of Theorem \ref{T:bup_sum} (finite time blow-up case), leaving the proof of the other case to the reader. 

Assume that $u$ satisfies the assumptions of Theorem \ref{T:bup_sum} (with $T_+(u)=1$) and let $w_{0,n}$ be as in Corollary \ref{C:devt}. Theorem \ref{T:bup_sum} will follow  from Corollary \ref{C:devt} and:
\begin{equation}
 \label{wn0}
\lim_{n\to\infty}\|w_{0,n}\|_{\hdot}=0.
\end{equation} 

We will deduce \eqref{wn0} from the following lemma:
\begin{lemma}
\label{L:induction}
Let $u$ be as in Corollary \ref{C:devt}.
 For all $k=1,\ldots,J_0+1$, there exists (after extraction of subsequences in $n$), a sequence $\left\{t_n^k\right\}_n$ such that
$t_n\leq t_n^k<1$ and, for large $n$,
\begin{equation}
\label{utn-un}
 \forall x\in \RR^3, \quad |x|\geq t_n^k-t_n\Longrightarrow \left(u\left(t_n^k,x\right),\partial_t u\left(t_n^k,x\right)\right)=\left(u_{0,n}^k(x),u_{1,n}^k(x)\right),
\end{equation} 
where $\left(u_{0,n}^k,u_{1,n}^k\right)\in \hdot \times L^2$ and
\begin{multline}
\label{devt_lemma}
\lim_{n\to \infty}\left\|
u_{0,n}^k-v(t_n^k)-\sum_{j=k}^{J_0} \frac{\iota_j}{\lambda_{j,n}^{1/2}}W\left(\frac{\cdot}{\lambda_{j,n}}\right)-w_{n}(t_n^k-t_n)\right\|_{\hdot}\\
=
\lim_{n\to \infty}\left\|u_{1,n}^k-\partial_t v(t_n^k)-\partial_tw_n\left(t_n^k-t_n\right)\right\|_{L^2}=0.
\end{multline} 
\end{lemma}
By convention, the sum in \eqref{devt_lemma} is $0$ when $k=J_0+1$.

Indeed, let us first assume Lemma \ref{L:induction} and prove Theorem \ref{T:bup_sum}.
\begin{proof}[Proof of Theorem \ref{T:bup_sum}]
 By \eqref{devt_lemma} in Lemma  \ref{L:induction} with $k=J_0+1$, we have
$$\lim_{n\to \infty}
\left\|\left(u^{J_0+1}_{0,n},u^{J_0+1}_{1,n}\right) -\left(v\left(t_n^{J_0+1}\right),\partial_t v\left(t_n^{J_0+1}\right)\right)-\left(w_n\left(t_n^{J_0+1}-t_n\right),\partial_t w_n\left(t_n^{J_0+1}-t_n\right)\right)\right\|_{\hdot\times L^2}=0,$$
where $t_n<t_n^{J_0+1}\leq 1$. Denote by $u_n^{J_0+1}$ the solution with initial condition $(u_{0,n}^{J_0+1},u_{1,n}^{J_0+1})$ at $t=0$. By Proposition \ref{P:lin_NL}, for $t\in[0,s_0]$ ($s_0>0$ is chosen small so that $1+s_0$ is in the domain of $v$),
\begin{equation}
\label{approx_unJ}
 u_{n}^{J_0+1}(t)=v\left(t_n^{J_0+1}+t\right)+w_n\left(t_n^{J_0+1}+t-t_n\right)+\tilde{r}_n(t),
\end{equation} 
where:
$$ \lim_{n\to \infty} \sup_{t\in [0,s_0]}\left(\|\tilde{r}_n(t)\|_{\hdot}+\|\partial_t \tilde{r}_n(t)\|_{L^2}\right)=0.$$
By finite speed of propagation and \eqref{utn-un},
\begin{equation}
\label{FSP_unJ}
\left(u_n^{J_0+1}(t,x),\partial_t u_n^{J_0+1}(t,x)\right)=\left(u(t_n^{J_0+1}+t,x),\partial_t u(t_n^{J_0+1}+t,x)\right)
\end{equation} 
for $0<t<1-t_n^{J_0+1}$, $|x|\geq t+t_n^{J_0+1}-t_n$.

By Proposition \ref{P:linear_behavior}, for all $t\in \RR$,
\begin{equation}
 \label{asymp_w}
\int_{|x|\geq |t|}\left[|\nabla w_n(t,x)|^2+(\partial_t w_n(t,x))^2\right]\,dx\geq \frac{1}{2}\int |\nabla w_{0,n}(x)|^2\,dx.
\end{equation} 
Indeed, this holds, according to Proposition \ref{P:linear_behavior}, for all $t\geq 0$ or for all $t\leq 0$. From the fact that $\partial_t w_{n\restriction t=0}=0$, we deduce that $w_n(t)=w_n(-t)$ and thus that $w_n$ satisfies \eqref{asymp_w} for all $t$.

Combining \eqref{approx_unJ}, \eqref{FSP_unJ} and \eqref{asymp_w} with $t=\frac{1-t_n^{J_0+1}}{2}$, we get
\begin{equation}
\label{bound_by_zero}
\int_{|x|\geq \frac{1+t_n^{J_0+1}}{2}-t_n} \left(\left|\nabla a\left(\frac{t_n^{J_0+1}+1}{2}\right)\right|^2+\left|\partial_t a\left(\frac{t_n^{J_0+1}+1}{2}\right)\right|^2\right)\,dx\geq \frac{1}{2}\int |\nabla w_{0,n}(x)|^2\,dx+o_n(1).
\end{equation} 
As $t_n<t_n^{J_0+1}$, we have
$$ \frac{1+t_n^{J_0+1}}{2}-t_n>\frac{1-t_n^{J_0+1}}{2}=1-\frac{1+t_n^{J_0+1}}{2},$$
thus the left-hand term in \eqref{bound_by_zero} is $0$, which shows that 
$$\lim_{n\to\infty}\int |\nabla w_{0,n}(x)|^2\,dx=0,$$
which conclude the proof of \eqref{bup_devt}. 

It remains to show \eqref{quant_energy}. By \eqref{bup_devt}, \eqref{bup_lambda_jn},
$$\lim_{n\to\infty} E(a(t_n),\partial_{t}a(t_n))=J_0E[W].$$
Using that $E(a(t),\partial_ta(t))$ has a limit as $t\to 1$ (see e.g. section $3$ of \cite{DuKeMe11a}), we get
$$\lim_{t\to 1^-} E(a(t),\partial_{t}a(t))=J_0E[W],$$
hence \eqref{quant_energy}, which concludes the proof of Theorem \ref{T:bup_sum}.
\end{proof}

\begin{proof}[Proof of Lemma \ref{L:induction}]
We show Lemma \ref{L:induction} by induction.

For $k=1$, \eqref{devt_lemma} follows from Corollary \ref{C:devt} with $t_n^1=t_n$, $(u_{0,n}^1,u_{1,n}^1)=(u(t_n),\partial_tu(t_n))$. 

We fix $k\in \{1,\ldots J_0\}$ and assume that the statement of the lemma holds for this $k$. Let $u_n^k$ be the solution of \eqref{CP} with initial condition $(u_{0,n}^k,u_{1,n}^k)$ at $t=0$. By finite speed of propagation and \eqref{utn-un},
$$ \left(u\left(t_n^k+t,x\right),\partial_t u\left(t_n^k+t,x\right)\right)=\left(u_n^k\left(t,x\right),\partial_t u_n^k\left(t,x\right)\right),\text{ for }|x|\geq |t|+t_n^k-t_n.$$
Furthermore by \eqref{devt_lemma} and Proposition \ref{P:lin_NL}, we get that for all $A>0$, for all $t\in [-A\lambda_{k,n},A\lambda_{k,n}]$,
\begin{equation}
\label{dev_unk}
 u_n^k(t)=v(t+t_n^k)+\sum_{j=k}^{J_0}\frac{1}{\lambda_{j,n}^{1/2}}W\left(\frac{x}{\lambda_{j,n}}\right)+w_n\left(t_n^k-t_n+t\right)+r_n^k(t)
\end{equation}
where
$$ \lim_{n\to \infty} \sup_{-A\lambda_{k,n}\leq t\leq A\lambda_{k,n}} \left(\left\|r_n^k(t)\right\|_{\hdot}+\left\|\partial_t r_n^k(t)\right\|_{L^2}\right)=0.$$
As a consequence, for all $p\geq 1$, there exists $N_p$ such that
$$\forall n\geq N_p,\quad \left\|r_n^{k}\left(p\lambda_{k,n}\right)\right\|_{\hdot}+\left\|\partial_t r_n^k\left(p\lambda_{k,n}\right)\right\|_{L^2}\leq \frac{1}{2^p}.$$
This yields an increasing sequence of integer $\{n_p\}_p$ such that $n_p\to +\infty$ and
\begin{equation}
\label{small_r}
\lim_{p\to \infty} \left\|r_{n_p}^k\left(p\lambda_{k,n_p}\right)\right\|_{\hdot}+\left\|\partial_t r_{n_p}^k\left(p\lambda_{k,{n_p}}\right)\right\|_{L^2}=0. 
\end{equation} 
In view of \eqref{ordre_lambda}, we can also chose $n_p$ large enough so that
\begin{equation*}
 \lim_{p\to\infty} \frac{p\lambda_{k,n_p}}{\lambda_{k+1,n_p}}=0.
\end{equation*} 
Coming back to \eqref{dev_unk}, we obtain
\begin{align*}
 u_{n_p}^k\left(p\lambda_{k,n_p}\right)&=v\left(t_{n_p}^k+p\lambda_{k,n_p}\right)+\sum_{j=k}^{J_0}\frac{1}{\lambda_{j,n_p}^{1/2}}W\left(\frac{\cdot}{\lambda_{j,n_p}}\right)+w_{n_p}\left(t_{n_p}^k-t_{n_p}+p\lambda_{k,n_p}\right)+r_{n_p}^k\left(p\lambda_{k,n_p}\right)\\
\partial_t u_{n_p}^k\left(p\lambda_{k,n_p}\right)&=\partial_t v\left(t_{n_p}^k+p\lambda_{k,n_p}\right)+\partial_tw_{n_p}\left(t_{n_p}^k-t_{n_p}+p\lambda_{k,n_p}\right)+\partial_tr_{n_p}^k\left(p\lambda_{k,n_p}\right).
\end{align*} 
Let $\varphi\in C^{\infty}(\RR^3)$ be a radial function such that $\varphi(x)=1$ for $|x|\geq 1$ and $\varphi(x)=0$ for $|x|\leq 1/2$. Let $t_{n_p}^{k+1}=t_{n_p}^k+p\lambda_{k,n_p}$. Define:
\begin{align*}
 u_{0,n_p}^{k+1}&=v\left(t_{n_p}^{k+1}\right)+\varphi\left(\frac{x}{p\lambda_{k,n_p}}\right)\sum_{j=k}^{J_0}\frac{1}{\lambda_{j,n_p}^{1/2}}W\left(\frac{\cdot}{\lambda_{j,n_p}}\right)+w_{n_p}\left(t_{n_p}^{k+1}-t_{n_p}\right)+r_{n_p}^k\left(p\lambda_{k,n_p}\right)\\
u_{1,n_p}^{k+1}&=\partial_t v\left(t_{n_p}^{k+1}\right)+\partial_tw_{n_p}\left(t_{n_p}^{k+1}-t_{n_p}\right)+\partial_tr_{n_p}^k\left(p\lambda_{k,n_p}\right).
\end{align*} 
Using that, as $p\to \infty$, 
$$\lambda_{k,n_p}\ll p\lambda_{k,n_p}\ll \lambda_{k+1,n_p}\ll\ldots\ll \lambda_{J_0,n_p},$$
we get, in view of \eqref{small_r},
\begin{multline*}
\lim_{p\to\infty}
\left\|
 u_{0,n_p}^{k+1}-v\left(t_{n_p}^{k+1}\right)-\sum_{j={k+1}}^{J_0}\frac{1}{\lambda_{j,n_p}^{1/2}}W\left(\frac{\cdot}{\lambda_{j,n_p}}\right)-w_{n_p}\left(t_{n_p}^{k+1}-t_{n_p}\right)\right\|_{\hdot}\\
+
\left\|u_{1,n_p}^{k+1}-\partial_t v\left(t_{n_p}^{k+1}\right)-\partial_tw_{n_p}\left(t_{n_p}^{k+1}-t_{n_p}\right)\right\|_{L^2}=0.
\end{multline*} 
This yields \eqref{devt_lemma} at the level $k+1$. By the definition of $\left(u_{0,n_p}^{k+1},u_{1,n_p}^{k+2}\right)$,
$$ u_{0,n_p}^{k+1}(x)= u_{n_p}^k\left(p\lambda_{k,n_p},x\right),\quad u_{1,n_p}^{k+1}(x)=\partial_t u_{n_p}^k\left(p\lambda_{k,n_p},x\right),\quad |x|\geq p\lambda_{k,n_p}.$$
By finite speed of propagation, \eqref{utn-un} and the definition of $t_n^{k+1}$,
$$ u_{n_p}^k\left(p\lambda_{k,n_p},x\right)=u\left(t_{n_p}^{k+1},x\right) ,\quad \partial_t u_{n_p}^k\left(p\lambda_{k,n_p},x\right)=\partial_tu\left(t_{n_p}^{k+1},x\right),\quad |x|\geq t_{n_p}^{k+1}-t_{n_p}.$$
Noting that $t_{n_p}^{k+1}-t_{n_p}=t_{n_p}^k-t_{n_p}+p\lambda_{n_p}\geq p\lambda_{n_p}$ we obtain \eqref{utn-un} at the level $k+1$.
\end{proof}
\section{The case of only one profile}
\label{S:one_profile}
In this section we show Theorems \ref{T:bup_one} and \ref{T:global_one}. Again, the proofs are very similar and we will focus on the proof of Theorem \ref{T:bup_one}.

Let $u$ be as in Theorem \ref{T:bup_one}. By Theorem \ref{T:bup_sum} and the assumption that $J_0=1$, changing $u$ into $-u$ if needed, there exist sequences $t_n\to 1$, $t_n<1$ and $\{\lambda_n\}_n$, such that
$$\lim_{n\to \infty} \left\|\left(a(t_n)-\frac{1}{\lambda_n^{1/2}}W\left(\frac{\cdot}{\lambda_n}\right),\partial_t a(t_n)\right)\right\|_{\hdot\times L^2}=0,$$
where $\lim_{n\to\infty}\frac{\lambda_n}{1-t_n}=0$.
By \eqref{quant_energy},
\begin{equation}
\label{Eat}
\lim_{t\to 1} E(a(t),\partial_t a(t))=E(W,0).
\end{equation}
Let $\delta_0>0$ be a small parameter. 
If the conclusion of Theorem \ref{T:bup_one} does not hold we get, in view of the variational characterization of $W$ (\cite{Aubin76}, \cite{Talenti76}), that there exists a sequence $\{t_n'\}_n$ such that $t_n'>t_n$,
\begin{equation}
\label{between_tn_t'n}
\left|\int_{\RR^3} |\nabla a(t)|^2+(\partial_t a(t))^2-\int_{\RR^3} |\nabla W|^2\right|<\delta_0\text{ for }t_n<t<t_n', 
\end{equation} 
and
\begin{equation}
 \label{at_t'n}
\left|\int_{\RR^3} |\nabla a(t_n')|^2+(\partial_t a(t_n'))^2-\int_{\RR^3} |\nabla W|^2\right|=\delta_0. 
\end{equation} 
Consider, after extraction, a profile decomposition $\left\{U_{\lin}^j,\lambda_{j,n},t_{j,n}\right\}$ for the sequence $(a(t_n'),\partial_t a(t_n'))$. 
Note that by \eqref{at_t'n}, the orthogonal expansion of the $\hdot\times L^2$ norm and Claim \ref{C:variational}, there exists a constant $c_0>0$ such that for all $j$,
\begin{equation}
\label{positive_energy}
E(U^j)\geq c_0\left(\|\nabla U^j_{\lin}(0)\|^2_{L^2}+\|\partial_t U^j_{\lin}(0)\|^2_{L^2}\right)\quad \text{and}\quad E(w_n^j)\geq c_0\left(\|\nabla w_{0,n}^{j}\|_{L^2}^2+\|w_{1,n}^j\|_{L^2}^2\right).
\end{equation} 

\EMPH{Step 1}
We first show that there is only one nonzero profile in the profile decomposition, and that
$$ \forall J\geq 1,\quad \lim_{n\to \infty} \left(\left\|\nabla w_{0,n}^J\right\|_{L^2}+\left\|w_{1,n}^J\right\|_{L^2}\right)=0$$

Indeed, if this does not hold we get by \eqref{Eat} and \eqref{positive_energy} that 
\begin{equation}
\label{subc_energy}
E\left(U_{0}^j,U_1^j\right)<E(W,0) 
\end{equation} 
for all $j$. By the result of \cite{KeMe08}, we obtain that for every $j$, the nonlinear profile $U^j$ scatters in both time directions or blows up in both time directions. Scattering for all $j$ would imply that the solution $u$ is defined after $t=1$, a contradiction. Thus one profile at least must blow up in both time directions, say $U^1$. We can assume $t_n^1=0$ for all $n$. By \eqref{between_tn_t'n},
$$ \sup_{T^-(U^1)<t<0}\left(\left\|\nabla U^1(t)\right\|_{L^2}^2+\left\|\partial_t U^1\right\|^2_{L^2}\right)\leq \|\nabla W\|_{L^2}^2+\delta_0.$$
This shows that the blow-up of $U^1$ for negative time is a small type II blow-up in the sense of Theorem 1 of \cite{DuKeMe11a}. The conclusion of this theorem implies 
$$ E(U^1_0,U^1_1)\geq E(W,0),$$
contradicting \eqref{subc_energy}. 

\EMPH{Step 2: contradiction}
By Step 1, we have only one nonzero profile, say $U^1_{\lin}$, and the dispersive part $(w_{0,n}^1,w_{1,n}^1)$ tends to $0$ in the energy space. Furthermore, scattering of $U^1$ in one time direction is excluded.  Indeed, if it scatters for positive time, we get that $u$ is defined beyond $t=1$, a contradiction. Furthermore, \eqref{Eat} and \eqref{between_tn_t'n} imply that $U^1$ is close to the manifold $\{\pm \mu^{3/2}W(\mu \cdot), \mu>0\}$ for negative time, which shows that $U^1$ cannot scatter backward in time. This shows that $-t_{1,n}/\lambda_{1,n}$ is bounded, and that we can assume that $t_{1,n}=0$ for all $n$. Finally, we obtain
\begin{equation*}
 \lim_{n\to\infty}\left\|\left(u(t'_n)-v(t_n')-\frac{1}{\lambda_{1,n}^{1/2}}U^1_0\left(\frac{x}{\lambda_{1,n}}\right),\partial_t u(t'_n)-\partial_t v(t_n')-\frac{1}{\lambda_{1,n}^{3/2}}U^1_1\left(\frac{x}{\lambda_{1,n}}\right)\right)
\right\|_{\hdot\times L^2}=0.
\end{equation*} 
By \eqref{Eat},
\begin{equation}
 \label{energyU1}
E(U^1_0,U^1_1)=E(W,0).
\end{equation} 
By \eqref{at_t'n},
\begin{equation}
\label{U1pm}
\|\nabla U^1_0\|^2_{L^2}+\|U^1_1\|_{L^2}^2=\|\nabla W\|_{L^2}^2\pm\delta_0.
\end{equation} 
We distinguish two cases.

If the sign in \eqref{U1pm} is a minus sign, then by the classification of solutions at the critical energy from \cite{DuMe08}, we obtain that $U^1$ scatters in at least one of the time directions, a contradiction.

Next assume that the sign in \eqref{U1pm} is a plus sign. In view of \eqref{between_tn_t'n},
$$ \forall t\in (T_-(U^1),0),\quad \|\nabla U^1(t)\|_{L^2}^2+\|\partial_tU^1(t)\|_{L^2}^2\leq \|\nabla W\|_{L^2}^2+\delta_0$$
This show by the rigidity Theorem \ref{T:rigidityW+} in Section \ref{S:W+} that $U^1=W^+$ (up to a time translation). By Theorem \ref{T:blowupW+}, $W^+$ blows-up for positive time with type I blow-up, which contradicts the fact (consequence of \eqref{bound_intro}) that $U^1$ is bounded. This concludes the proof of Theorem \ref{T:bup_one}.

\section{Properties of $W^+$}
\label{S:W+}
Recall from \cite{DuMe08} that there exists a solution $W^+$ of \eqref{CP} such that
\begin{gather}
\label{energieW+}
 E[W^+]=E[W],\quad \int_{\RR^3} |\nabla W^+(0)|^2>\int_{\RR^3} |\nabla W|^2\\
\label{T+W+}
T_+(W^+)=1\\
\label{T-W+}
T_-(W^+)=-\infty\quad\text{and}\quad  \exists C>0,\; \forall t\leq 0,\;\|W^+(t)-W\|_{\hdot}+\|\partial_tW^+(t)\|_{L^2}\leq Ce^{t/C}.
\end{gather}
Note that \eqref{energieW+} implies, by Claim \ref{C:variational} that $\int |\nabla W^+(t)|^2>\int|\nabla W|^2$ for all $t\in (-\infty,1)$. The properties \eqref{energieW+} and \eqref{T-W+} characterize $W^+$ up to time translation (see \cite[Lemma 6.4]{DuMe08}):
\begin{prop}
 \label{P:exp_uniq_W+}
Let $u$ be a solution of \eqref{CP} (possibly not radial) such that
\begin{gather}
\label{energieW+_car}
 E[u]=E[W],\quad \int_{\RR^3} |\nabla u(0)|^2>\int_{\RR^3} |\nabla W|^2\\
\label{T-W+_car}
T_-(u)=-\infty\quad\text{and}\quad  \exists C>0,\; \forall t\leq 0,\; \|u(t)-W\|_{\hdot}+\|\partial_tu(t)\|_{L^2}\leq Ce^{t/C}.
\end{gather}
Then $T_+(u)<+\infty$ and, denoting $t_0=1-T^+(u)$,
$$ u(t)=W^+(t+t_0).$$
\end{prop}
In Subsection \ref{SS:rigidity}, we deduce from Proposition \ref{P:exp_uniq_W+} a stronger uniqueness statement. In Subsection \ref{SS:typeI}, we show that the blow-up of $W^+$ at $t=1$ is of type $I$, i.e.
$$ \limsup_{t\to 1} \|(W^+(t),\partial_tW^+(t))\|_{\hdot\times L^2}=+\infty.$$
This section uses some of the results of our previous works \cite{KeMe08}, \cite{DuMe08} and \cite{DuKeMe10P}, but is completely independent of the preceding sections.
\subsection{A rigidity theorem}
\label{SS:rigidity}
In this subsection we show the following rigidity result:
\begin{theoint}
 \label{T:rigidityW+}
There exists a small $\delta_0>0$ such that the following holds. Let $u$ be a solution of \eqref{CP} (not necessarily radial) such that
\begin{gather}
\label{energyW}
E(u_0,u_1)=E(W,0)\\
\label{nablaW}
 \forall t\in (T_-(u),0),\quad \|\nabla W\|_{L^2}^2<\|\nabla u(t)\|_{L^2}^2+\frac 12\|\partial_t u(t)\|_{L^2}^2\leq \|\nabla W\|_{L^2}^2+\delta_0.
\end{gather}
Then, up to change of sign, scaling and time translation, $u=W^+$.
\end{theoint}
Let us mention that Theorem \ref{T:rigidityW+} also holds in space dimension $N=4,5$ for the corresponding energy critical wave equation (replacing the coefficient $1/2$ in front of $\|\partial_t u(t)\|^2$ by $(N-2)/2)$). To simplify the exposition, we only write the proof in space dimension $3$.
\begin{proof}[Sketch of proof]
The proof of Theorem \ref{T:rigidityW+} is very close to the proof of Theorem 2 (a) in \cite{DuMe08}, and we only sketch it.  It is divided into 5 steps. The main difference with respect to the proof in \cite{DuMe08} is the use of the description of small type II blow-up solutions (from \cite{DuKeMe10P}) in Step 1.

\EMPH{Step 1. Compactness}
We first show that there exist $\lambda(t),x(t)$ defined for $t\in (T_-(u),0]$ such that 
$$K_-=\left\{ \frac{1}{\lambda(t)^{1/2}}u\left(t,\frac{\cdot}{\lambda(t)}+x(t)\right), \frac{1}{\lambda(t)^{3/2}}\partial_t u\left(t,\frac{\cdot}{\lambda(t)}+x(t)\right),\; t\in (T_-(u),0]\right\}$$
has compact closure in $\hdot\times L^2$. This follows from the compactness argument introduced in \cite{KeMe06}, \cite{KeMe08}. It is sufficient to prove that for any sequence $t_n\to T_-(u)$, there exists a subsequence (still denoted by $\{t_n\}_n$) and sequences $\{\lambda_n\}_n$, $\{x_n\}_n$ of parameters such that 
$$ \left(\frac{1}{\lambda_n^{1/2}}u\left(t_n,\frac{\cdot}{\lambda_n}+x_n\right),\frac{1}{\lambda_n^{3/2}}\partial_t u\left(t_n,\frac{\cdot}{\lambda_n}+x_n\right)\right) $$
converges in $\hdot\times L^2$.  Note that by \eqref{energyW}, \eqref{nablaW} $u$ does not scatter in any time direction. Indeed, assume for example that $u$ scatters forward in time. Then 
$$\lim_{t\to+\infty}\|u(t)\|_{L^6}=0.$$
Hence by \eqref{nablaW} and conservation of energy, $E[u]\geq \frac{1}{2}\|\nabla W\|^2_{L^2}=\frac{3}{2}E[W]$ a contradiction with \eqref{energyW}.

For a given sequence $t_n\to T_-(u)$ consider, after extraction, a profile decomposition:
\begin{multline*}
 (u(t_n),\partial_tu(t_n))=\sum_{j=1}^J \left(\frac{1}{\lambda_{j,n}^{1/2}}U^{j}_{\lin}\left(\frac{-t_{j,n}}{\lambda_{j,n}},\frac{\cdot-x_{j,n}}{\lambda_{j,n}}\right),\frac{1}{\lambda_{j,n}^{3/2}}\partial_t U^{j}_{\lin}\left(\frac{-t_{j,n}}{\lambda_{j,n}},\frac{\cdot-x_{j,n}}{\lambda_{j,n}}\right)\right)+(w_{0,n}^J,w_{1,n}^J).
\end{multline*}
where $(w_{0,n}^J,w_{1,n}^J)$, $\lambda_{j,n}$ and $x_{j,n}$ satisfy the pseduo-orthogonality property:
\begin{equation*}
j\neq k\Longrightarrow \lim_{n\to \infty} \frac{\lambda_{j,n}}{\lambda_{k,n}}+\frac{\lambda_{k,n}}{\lambda_{j,n}}+\frac{|t_{j,n}-t_{k,n}|}{\lambda_{j,n}}+\frac{\left|x_{j,n}-x_{k,n}\right|}{\lambda_{j,n}}=+\infty.
\end{equation*} 
and \eqref{small_w}. As usual, we denote by $U^j$ the corresponding nonlinear profiles. As $u$ does not scatter, there is at least one nonzero profile, say $U^1_{\lin}$. We must show
\begin{equation}
 \label{conclusion_step1}
j\geq 2\Longrightarrow U^j_{\lin}=0\;\text{and}\;\lim_{n\to\infty}\left\|(w_{0,n}^j,w_{1,n}^j)\right\|_{\hdot\times L^2}=0.
\end{equation} 
By Claim \ref{C:variational}, there is a constant $c>0$ such that for all $j$,
$$ E\left( U^j_{\lin}\left(\frac{-t_{j,n}}{\lambda_{j,n}}\right),\partial_t U^j_{\lin}\left(\frac{-t_{j,n}}{\lambda_{j,n}}\right)\right)\geq c\left\|\left(U^j_{\lin}\left(\frac{-t_{j,n}}{\lambda_{j,n}}\right),\partial_t U^j_{\lin}\left(\frac{-t_{j,n}}{\lambda_{j,n}}\right)\right)\right\|_{\hdot\times L^2}^2$$
and
$$ E\left(w_{0,n}^j,w_{1,n}^j\right)\geq c\|(w_{0,n}^j,w_{1,n}^j)\|_{\hdot\times L^2}^2.$$
Thus, in view of the expansion of the energy \eqref{pythagore2}, if \eqref{conclusion_step1} does not hold, there exists $\eps>0$ (independent of $j$) such that
$$ \forall j\geq 2,\quad E[U^j]\leq E[W]-\eps,\quad E(w_{0,n}^j,w_{1,n}^j)\leq E[W]-\eps.$$
In view of \eqref{nablaW}, there is at most one index $j_0$ such that  $\|\nabla U^{j_0}(-t_{j_0,n}/\lambda_{j_0,n})\|_{L^2}^2>\|\nabla W\|^2_{L^2}$. Assume that such an index exists. By \cite{KeMe08}, $U^{j_0}$ blows-up in finite time. By Proposition \ref{P:lin_NL}, \eqref{energyW}, \eqref{nablaW} and expansions as in \eqref{pythagore1a},\eqref{pythagore1b} and \eqref{pythagore2}, $U^{j_0}$ is a type II blow-up solution that satisfies 
$$ \sup_{t\in I_{\max}(U^{j_0})} \|\nabla U^{j_0}(t)\|_{L^2}^2+\frac{1}{2}\|\nabla U^{j_0}(t)\|_{L^2}^2\leq \|\nabla W\|^2+\delta_0,\quad E[U^{j_0}]< E[W].$$
This contradicts the fact that the main theorem of \cite{DuKeMe10P} implies that $E[U^{j_0}]\geq E[W]$ showing that for all $j$,
$$ \|\nabla U^{j}(-t_{j,n}/\lambda_{j,n})\|_{L^2}<\|\nabla W\|_{L^2}.$$
Let us mention that the main theorem of \cite{DuKeMe10P} only holds in space dimension $3$ and $5$. To prove Theorem \ref{T:rigidityW+} in dimension $4$, one should use \cite[Corollary 3.2]{DuKeMe10P}, which also implies that $E[U^{j_0}]\geq E[W]$ and holds in space dimensions $3$, $4$ and $5$.

By \cite{KeMe08}, all profiles $U^j$ scatter. This shows that $u$ scatters, a contradiction. We have shown \eqref{conclusion_step1}, i.e.
$$ (u(t_n),\partial_t u(t_n))=\left( \frac{1}{\lambda_{1,n}^{1/2}}U^1_{\lin}\left(\frac{-t_{1,n}}{\lambda_{1,n}},\frac{\cdot-x_{1,n}}{\lambda_{1,n}}\right), \frac{1}{\lambda_{1,n}^{3/2}}\partial_tU^1_{\lin}\left(\frac{-t_{1,n}}{\lambda_{1,n}},\frac{\cdot-x_{1,n}}{\lambda_{1,n}}\right)\right)+o_n(1) \text{ in }\hdot\times L^2.$$
It remains to notice that the fact that $u$ does not scatter in any time direction implies that $-t_{1,n}/\lambda_{1,n}$ is bounded, which shows the desired compactness property and concludes Step 1.

\EMPH{Step 2. Global existence for negative time} It follows from Step 1 and \cite[Proposition 4.5]{DuKeMe10P} that $T_-(u)=-\infty$. We omit the proof and refer the reader to \cite{DuKeMe10P}. Let us mention that the main ingredients of this proof are the impossibility of self-similar blow-up, shown in \cite{KeMe08}, and monotonicity formulas as \eqref{virial} below.

\EMPH{Step 3. Convergence for a subsequence}
We introduce
$$ \dd(t)=\int |\nabla u(t)|^2-\int |\nabla W|^2+\frac{1}{2}\int (\partial_t u(t))^2.$$
Then one can show:
$$ \lim_{T\to -\infty} \frac{1}{T}\int_{T}^0\dd(t)\,dt=0.$$
This follows from the following properties, shown in \cite[Proposition 2.8]{DuMe08}:
$$ \lim_{t\to -\infty}|t|\lambda(t)=+\infty,\quad \lim_{t\to -\infty} \frac{x(t)}{t}=0,$$
and from the identity:
\begin{multline}
 \label{virial}
\frac{d}{dt} \int_{\RR^3} u\partial_t u\varphi\left(\frac{x}{R}\right)=\int_{\RR^3} (\partial_t u)^2-\int |\nabla u|^2+\int u^6+O\left(\int_{|x|\geq R} (\partial_tu)^2+|\nabla u|^2+|u|^6\right)\\
\geq 2d(t)+O\left(\int_{|x|\geq R} (\partial_tu)^2+|\nabla u|^2+|u|^6\right),
\end{multline} 
where $\varphi\in C_0^{\infty}(\RR^3)$, $\varphi(x)=1$ for $|x|\leq 1$. We omit again the details, refering to Subsection 3.1 of \cite{DuMe08}.

\EMPH{Step 4. Modulation}
Let $t\leq 0$. By Claim 3.6 and Lemma 3.7 in \cite{DuMe08}, the property $\dd(t)<\delta_0$ implies (if $\delta_0>0$ is small enough) that there exists $\mu(t)>0$, $X(t)\in  \RR^3$ such that
$$\frac{1}{\mu(t)^{1/2}}u\left(t,\frac{\cdot}{\mu(t)}+X(t)\right)\in \left\{\partial_{x_1}W,\partial_{x_2}W,\partial_{x_3}W,\frac{1}{2}W+x\cdot \nabla W\right\}^{\bot},$$
where the orthogonality is to be understood in the sense of the $\hdot$ scalar product. Furthermore, chosing $\alpha(t)$ such that
$$ f(t)=\frac{1}{\mu(t)^{1/2}}u\left(t,\frac{\cdot}{\mu(t)}+X(t)\right)-(1+\alpha(t))W \in \left\{W,\partial_{x_1}W,\partial_{x_2}W,\partial_{x_3}W,\frac{1}{2}W+x\cdot \nabla W\right\}^{\bot},$$
we have the following estimates, uniformly for $t\leq 0$:
\begin{gather}
 \label{est_modul1}
|\alpha(t)|\approx \|\nabla(\alpha(t) W+f(t))\|_{L^2}\approx \|\nabla f(t)\|_{L^2}+\|\partial_t u(t)\|_{L^2}\approx \dd(t)\\
\label{est_modul2}
\left|\frac{\alpha'}{\mu}\right|+\left|\frac{\mu'}{\mu^2}\right|+\left|X'(t)\right|\lesssim \dd(t).
\end{gather}
It is easy to see that the compactness of $\overline{K}_-$ implies that the following set has compact closure in $\hdot\times L^2$:
$$\widetilde{K}_-=\left\{ \frac{1}{\mu(t)^{1/2}}u\left(t,\frac{\cdot}{\mu(t)}+X(t)\right), \frac{1}{\mu(t)^{3/2}}\partial_t u\left(t,\frac{\cdot}{\mu(t)}+X(t)\right),\; t\in (T_-(u),0]\right\}.$$

\EMPH{Step 5. Exponential convergence and conclusion}
In this step we show that (after a change of sign if necessary) there exist $\mu_{\infty}>0$, $X_{\infty}\in \RR^3$ and $C>0$ such that
\begin{equation}
\label{exp_CV} 
t\leq 0\Longrightarrow \left\|\left((u(t),\partial_t u(t))-\left(\frac{1}{\mu_{\infty}^{1/2}}W\left(\frac{\cdot-X_{\infty}}{\mu_{\infty}} \right),0\right) \right)\right\|_{\hdot\times L^2}\leq Ce^{t/C}.
\end{equation} 
This will conclude the proof according to Proposition \ref{P:exp_uniq_W+}. The arguments are again very close to the ones of \cite{DuMe08} and we will only sketch them, refering the reader to Subsection 3.3 of that paper for the details. 

By a careful analysis of the remainder term in the identity \eqref{virial}, using the modulation theory described in Step 4, one can show (see \cite[Lemma 3.8]{DuMe08}):
\begin{equation}
\label{virial_type}
 \sigma<\tau\leq 0\Longrightarrow \int_{\sigma}^{\tau} \dd(t)\,dt\leq C\left(\max_{\sigma\leq t\leq \tau} |X(t)|+\frac{1}{\mu(t)}\right)(\dd(\sigma)+\dd(\tau)).
\end{equation} 
Furthermore, integrating \eqref{est_modul2} between $\sigma$ and $\tau$, we get
\begin{equation}
 \label{control_param}
 \sigma<\tau\leq 0\Longrightarrow |X(\sigma)-X(\tau)|+\left|\frac{1}{\mu(\sigma)}-\frac{1}{\mu(\tau)}\right|\leq \int_{\sigma}^{\tau} \dd(t)\,dt.
\end{equation} 
According to Step 3, there exists a decreasing sequence $\tau_n\to -\infty$ such that 
$$ \lim_{n\to \infty} \dd(\tau_n)=0.$$
We first note that \eqref{virial_type} and \eqref{control_param} imply that $|X(t)|+\frac{1}{\mu(t)}$ is bounded on $(-\infty,0]$. Indeed, let $m\gg 1$ and, for $n\geq m$, $M_n=\max\{|X(t)|+1/\mu(t),\; t\in [\tau_n,\tau_m]\}$. Combining \eqref{virial_type} and \eqref{control_param}, we get for $t\in [\tau_n,\tau_m]$,
$$ |X(t)|+\frac{1}{\mu(t)}\leq C\int_{\tau_n}^{\tau_m} \dd(t)\,dt+|X(\tau_m)|+\frac{1}{\mu(\tau_m)}\leq C\,M_n(\dd(\tau_m)+\dd(\tau_n))+|X(\tau_m)|+\frac{1}{\mu(\tau_m)}.$$
 Hence
$$ M_n \leq CM_n(d(\tau_m)+d(\tau_n))+|X(\tau_m)|+\frac{1}{\mu(\tau_m)},$$
which implies the boundedness of $|X(t)|+1/\mu(t)$, fixing $m$ large (so that $C\dd(\tau_m)\leq 1/2$) and letting $n\to \infty$.

Coming back to \eqref{virial_type}, we obtain the diffential inequality $\int_{-\infty}^{\tau}\dd(t)\,dt\leq C\dd(\tau)$ and thus
\begin{equation}
\label{int_exp}
\int_{-\infty}^{\tau}\dd(t)\,dt\leq Ce^{c\tau}. 
\end{equation} 
Together with \eqref{control_param}, we get that there exist $X_{\infty}\in \RR^3$, $\ell_{\infty}\geq 0$ such that
$$ |X_{\infty}-X(\tau)|+\left|\frac{1}{\mu(\tau)}-\ell_{\infty}\right|\leq Ce^{\tau/C},\quad \tau\leq 0.$$
After a space translation, one may assume $X_{\infty}=0$. The inequality
$$ |X(\tau)|+\left|\frac{1}{\mu(\tau)}-\ell_{\infty}\right|\leq C\max_{-\infty\leq t\leq \tau_n}\left(|X(t)|+\frac{1}{\mu(t)}\right)d(\tau_n),\quad \tau\leq\tau_n$$
which follows from \eqref{virial_type} and \eqref{control_param} implies that we cannot have $\ell_{\infty}=0$. Hence, letting $\mu_{\infty}=1/\ell_{\infty}$
\begin{equation}
\label{exp_modul}
|X(\tau)|+|\mu(\tau)-\mu_{\infty}|\leq Ce^{\tau/C},\quad \tau\leq 0. 
\end{equation} 
It remains to show that $\dd(t)\leq Ce^{tC}$ for $t\leq 0$. This follows from \eqref{int_exp} and the estimates $|\alpha(t)|\approx \dd(t)$ and $|\alpha'(t)|\lesssim \dd(t)$ in \eqref{est_modul1}, \eqref{est_modul2}. In view of \eqref{est_modul1} and \eqref{exp_modul}, this shows \eqref{exp_CV} and concludes the proof of the Theorem.
\end{proof}
We conclude this section by giving a corollary to Theorem \ref{T:rigidityW+} and the main result of \cite{KrNaSc10P}, which completes Theorem 2 of \cite{DuMe08}.
\begin{corol}
 \label{C:classification}
Let $u$ be a radial solution of \eqref{CP} such that $E[u]=E[W]$. Assume that $\int |\nabla u_0|^2>\int |\nabla W|^2$. Then $u$ blows up in finite time in both time directions or $u=W^+$ up to the symmetries of the equation.
\end{corol}
\begin{remark}
 In Theorem 2, c) of \cite{DuMe08}, the same result is stated for (possibly nonradial) solutions such that $u_0\in L^2$. 
\end{remark}
\begin{remark}
 The radial restriction comes from \cite{KrNaSc10P}. If the main result of \cite{KrNaSc10P} holds without radiality assumption, then Corollary \ref{C:classification} holds for general nonradial solutions of \eqref{CP}.
\end{remark}
\begin{remark}
For simplicity, we stated Corollary \ref{C:classification} in space dimension $3$. Since the proof of Theorem \ref{T:rigidityW+} can be adapted to space dimensions $4$ and $5$, and the main Theorem of \cite{KrNaSc10P} also holds in dimension $5$, we can conclude that Corollary \ref{C:classification} holds for the energy-critical equation in dimension $5$ also. Similarly, the generalization of \cite{KrNaSc10P} to dimension $4$ would imply Corollary \ref{C:classification} in dimension $4$.
\end{remark}
\begin{proof}
 Let 
$$\SSS=\left\{\left(\frac{\iota}{\lambda^{1/2}}W\left(\frac{\cdot}{\lambda}\right),0\right),\quad \lambda>0,\; \iota=\pm 1\right\}.$$
Let $\eta>0$. By the main theorem of \cite{KrNaSc10P}, there exists a neighborhood $\BBB$ of $\SSS$ in $\hdot\times L^2$, within $O(\sqrt{\eta})$ distance in $\hdot\times L^2$ such that for any radial solution $u$ of \eqref{CP} satisfying $E[u]<E[W]+\eta$, one of the following holds:
\begin{gather}
\label{scattering}
 T_+(u)=+\infty \text{ and }u \text{ scatters forward in time.}\\
\label{close_to_B}
T_+(u)=+\infty\text{ and }(u,\partial_tu)(t)\in \BBB \text{ for large positive }t\\
\label{blowup}
\text{ or }T_+(u)<\infty.
\end{gather}
Let $u$ be a solution of \eqref{CP} such that $E[u]=E[W]$ and $\int |\nabla u_0|^2>\int |\nabla W|^2$. Then by Claim \ref{C:variational},
$\int|\nabla u(t)|^2>\int |\nabla W|^2$ for all $t$ in the domain of existence of $u$. In particular, by the same argument as in Step 1 of the the proof of Theorem \ref{T:rigidityW+}, \eqref{scattering} is excluded. 

Assume that \eqref{close_to_B} holds. Then
$$ \forall t\geq 0,\quad \int |\nabla W|^2<\int |\nabla u(t)|^2<\int |\nabla W|^2+C\sqrt{\eta}.$$
Taking $\eta$ small enough, we get by Theorem \ref{T:rigidityW+} that $u=W^+$ up to the symmetries of the equation. The same argument for negative times implies that $u$ blows up in both time directions or $u=W^+$ modulo symmetries, concluding the proof.
\end{proof}

\subsection{Type I blow-up}
\label{SS:typeI}
The purpose of this subsection is to show that the blow-up of $W^+$ is of type I, i.e:
\begin{theoint}
\label{T:blowupW+}
$$ \limsup_{t\to 1^-} \|\nabla W^+(t)\|_{L^2}^2+\|\partial_t W^+(t)\|_{L^2}^2=+\infty.$$
\end{theoint}
We divide the proof into a few propositions:
\begin{prop}
\label{P:typeI}
 Let $u$ be a radial solution of \eqref{CP} such that $T_+=T_+(u)<\infty$ and 
$$\forall t\in [T_+-\eps,T_+), \quad \forall r>0, \quad \partial_t u(t,r)\geq 0.$$
 Then the blow-up of $u$ is of type I.
\end{prop}
\begin{proof}
 Assume that the blow-up of $u$ is a type II blow-up. We will reach a contradiction in two steps, using the general description of blow-up type II solutions given in \cite[Section 3]{DuKeMe11a}. 

\EMPH{Step 1: convergence of $u$ in $L^6$}
As in the beginning of subsections \ref{SS:self_simF}, we 
denote by $(v_0,v_1)$ the weak limit, in $\hdot\times L^2$ of $(u(t),\partial_t u(t))$ as $t\to T_+$, and by $v$ the solution of \eqref{CP} with initial condition $(v_0,v_1)$ at $t=T_+$. We have (taking a smaller $\eps>0$ if necessary),
\begin{equation*}
 u\in C^0\Big((T_+-\eps,T_+),\hdot\Big),\quad
 v\in C^0\Big((T_+-\eps,T_+],\hdot\Big).
\end{equation*} 
and thus, by one-dimensional Sobolev embeddings, and using that $u$ and $v$ are radial functions,
\begin{equation*}
 u\in C^0\Big((T_+-\eps,T_+)\times \left(\RR^3\setminus\{0\}\right)\Big),\quad 
 v\in C^0\Big((T_+-\eps,T_+]\times \left(\RR^3\setminus\{0\}\right)\Big).
\end{equation*} 
By finite speed of propagation, $u-v$ is supported in the set $\left\{|x|\leq T_+-t\right\}$, which implies
$$ x\neq 0\Longrightarrow \lim_{t\to T_+} u(t,x)=v_0(x).$$
As $\partial_t u\geq 0$, we obtain 
$$ \forall x\neq 0,\;\forall t\in (T_+-\eps,T_+),\quad u(T_+-\eps,x)\leq u(t,x)\leq v_0(x).$$
Using that $v_0\in L^6$, we get by dominated convergence:
\begin{equation}
 \label{u_to_v}
\lim_{t\to T_+} \left\|u(t)-v_0\right\|_{L^6}=0.
\end{equation} 

\EMPH{Step 2: profile decomposition and contradiction}
Let $t_n\to T_+$. Consider, after extraction of a subsequence, a profile decomposition for the sequence $\left(u(t_n)-v_0,\partial_t u(t_n)-v_1\right)$:
\begin{equation}
\label{decompo_profil2}
\left\{\begin{aligned}
 u(t_n)-v_0&=\sum_{j=1}^J \frac{1}{\lambda_{j,n}^{\frac{1}{2}}}U_{\lin}^j\left(\frac{-t_{j,n}}{\lambda_{j,n}},\frac{x}{\lambda_{j,n}}\right)+w_{0,n}^J(x),\\
\partial_t u(t_n)-v_1&=\sum_{j=1}^J \frac{1}{\lambda_{j,n}^{\frac{3}{2}}}\partial_t U_{\lin}^j\left(\frac{-t_{j,n}}{\lambda_{j,n}},\frac{x}{\lambda_{j,n}}\right)+w_{1,n}^J(x),
\end{aligned}\right.
\end{equation}
where, the solution $w_n^J$ the linear wave equation \eqref{lin_wave} with initial conditions $\left(w_{0,n}^J,w_{1,n}^J\right)$ satisfies \eqref{small_w}.

Consider, for every $j$, the nonlinear profile $U^j$ associated to $\left(U^j_{\lin},\left\{-t_{j,n}/\lambda_{j,n}\right\}_n\right)$. 
Let $\JJJ$ be the set of indices $j$ such that $U^j$ does not scatter forward in time. The set $\JJJ$ if finite and not empty (if $\JJJ$ is empty, $u$ scatters forward in time, a contradiction). For $j\in \JJJ$, one may assume
 \begin{gather}
\label{Uj0}
t_{j,n}=0\text{ or}\\
\label{m_infty}
\lim_{n\to+\infty}\frac{-t_{j,n}}{\lambda_{j,n}}=-\infty.
\end{gather}
For $j\in \JJJ$, chose $T_j$ in the domain of definition of $U^j$ and such that 
\begin{equation}
\label{choice_Tj}
 U^j(T_j)\neq 0.
\end{equation} 
If \eqref{Uj0} holds, we can also assume $T_j>0$. Note that this implies in both cases \eqref{Uj0} or \eqref{m_infty} that for any $j$ in $\JJJ$, for large $n$,
\begin{equation}
\label{positive_time}
 t_{j,n}+\lambda_{j,n}T_j>0.
\end{equation} 
Reordering the profiles and extracting subsequences, we may assume:
\begin{equation*}
1\in \JJJ\text{ and}\quad \forall n,\; t_{1,n}+\lambda_{1,n}T_1=\min_{j\in \JJJ} \left(t_{j,n}+\lambda_{j,n}T_{j}\right).
\end{equation*}  
We will show that $U^1(T_1)=0$, which will contradicts \eqref{choice_Tj} for $j=1$.

By Proposition \ref{P:lin_NL}, $t_n+t_{1,n}+\lambda_{1,n}T_1<T_+(u)$ for large $n$ and 
\begin{equation}
\label{approximation}
\left\{\begin{aligned}
 u(t_n+t_{1,n}+\lambda_{1,n}T_1)&=v_0+\frac{1}{\lambda_{1,n}^{\frac{1}{2}}}U^1\left(T_1,\frac{\cdot}{\lambda_{1,n}}\right) \\&
\qquad+\sum_{j=2}^J \frac{1}{\lambda_{j,n}^{\frac{1}{2}}}U^j\left(\frac{t_{1,n}+\lambda_{1,n}T_1-t_{j,n}}{\lambda_{j,n}},\frac{\cdot}{\lambda_{j,n}}\right)
+\tilde{w}_{0,n}^J,\\
\partial_t u(t_n+t_{1,n}+\lambda_{1,n}T_1)&=v_1+\frac{1}{\lambda_{1,n}^{\frac{3}{2}}}\partial_t U^1\left(T_1,\frac{\cdot}{\lambda_{1,n}}\right)\\
&\qquad+\sum_{j=2}^J \frac{1}{\lambda_{j,n}^{\frac{3}{2}}}\partial_t U^j\left(\frac{t_{1,n}+\lambda_{1,n}T_1-t_{j,n}}{\lambda_{j,n}},\frac{\cdot}{\lambda_{j,n}}\right)+ \tilde{w}_{1,n}^J,
\end{aligned}\right.
\end{equation}
where the solution $\tilde{w}_n^J$ of the linear wave equation with initial conditions $\left(\tilde{w}_{0,n}^J,\tilde{w}_{1,n}^J\right)$ satisfies \eqref{small_w}.
As a consequence of \eqref{positive_time} for $j=1$, we get (using that for large $n$, $t_n+t_{1,n}+\lambda_{1,n}T_1<T_+$),
\begin{equation}
\label{temps_OK}
 \lim_{n\to \infty}t_n+t_{1,n}+\lambda_{1,n}T_1=T_+.
\end{equation} 
In view of \eqref{approximation}  and using the pseudo-orthogonality of the sequences of parameters, we get, by the expansion of the $L^6$ norm:
$$ \liminf_{n\to\infty}\left\| u(t_n')-v_0\right\|_{L^6}\geq \left\|U^1(T_1)\right\|_{L^6},$$
where 
$ t_n'=t_n+t_{1,n}+\lambda_{1,n}T_1\underset{n\to \infty}{\longrightarrow} T_+.$
This contradicts \eqref{u_to_v}, unless $U^1(T_1)=0$, concluding the proof. 
\end{proof}
We next prove a result on positive solutions of \eqref{CP}. In this lemma, it is not necessary to assume that the solution $u$ is in the energy space.
\begin{lemma}
 \label{L:positivity1}
Let $u$ be a $C^3$, radial solution of \eqref{CP} such that
\begin{gather}
\label{u0_positiveP}
 \forall r>0,\quad u_0(r)\geq 0\\
\label{u0_increasesP}
r\mapsto ru_0(r)\text{ is nondecreasing on }(0,+\infty)\\
\label{u1_positiveP}
\forall r>0,\quad u_1(r)\geq 0\\
\label{u1_increasesP}
r\mapsto ru_1(r)\text{ is nondecreasing on }(0,+\infty)\\
\label{derivee_2}
\forall r>0,\quad \Delta u_0(r)+u_0^5(r)>0.
\end{gather}
Then
\begin{equation}
\label{u_positiveP}
 \forall t\in (0,T_+(u)),\;\forall r>0,\quad u(t,r)>0\text{ and }\partial_tu(t,r)>0.
\end{equation}
\end{lemma} 

We first prove the following positivity lemma for the linear wave equation:
\begin{lemma}
 \label{L:positivity}
Let $L>0$, $f$  a continuous radial function on $(0,L)\times \RR^3$, and 
$u$ a $C^2$, radial solution of 
\begin{equation*}
\begin{gathered}
 \partial_t^2u-\Delta u=f\\
u_{\restriction t=0}=u_0,\quad \partial_t u_{\restriction t=0}=u_1,
\end{gathered}
\end{equation*}
defined for $t\in [0,L)$ such that
\begin{gather}
\label{u0_positive}
 \forall r\in (0,L),\quad u_0(r)\geq 0\\
\label{u0_increases}
r\mapsto ru_0(r)\text{ is nondecreasing on }(0,L)\\
\label{u1_positive}
\forall r\in (0,L),\quad u_1(r)\geq 0\\
\label{f_positive}
\forall t\in (0,L),\; \forall r,\; 0<r<L-t\Longrightarrow f(t,r)\geq 0.
\end{gather}
Then
\begin{equation}
\label{u_positive}
 \forall t\in (0,L),\; \forall r,\; 0<r<L-t\Longrightarrow u(t,r)\geq 0.
\end{equation}
Furthermore if one of the inequalities in \eqref{u0_positive} or \eqref{u1_positive} is strict, or if the function $t\mapsto r u_0(r)$ is strictly increasing on $(0,L)$, then \eqref{u_positive} holds with a strict inequality.
\end{lemma}
\begin{proof}
Consider the functions defined for $(t,s)\in (0,L)\times \RR$ by
\begin{equation}
\label{def_odd_f}
v(t,s)=su(t,|s|),\; v_0(s)= su_0(|s|),\; v_1(s)=su_1(t,|s|),\; F(t,s)=sf(t,|s|).
\end{equation} 
Note that all these functions are odd in $s$.
Then
\begin{equation*}
\partial_t^2 v-\partial_s^2 v=F,\quad v_{\restriction t=0}=v_0,\; \partial_t v_{\restriction t=0}=v_1.
\end{equation*}
Thus, for $t\in (0,L)$, $s\in \RR$,
\begin{equation}
\label{formula_v}
 2v(t,s)=\underbrace{v_0(t+s)-v_0(t-s)}_{(a)}+\underbrace{\int_{t-s}^{t+s} v_1(\rho)\,d\rho}_{(b)}+\underbrace{\int_0^t\int_{t-\tau-s}^{t-\tau+s}F(\tau,\rho)\,d\rho\,d\tau}_{(c)}.
\end{equation}  
We check that for $0<t<L$, $0<s<L-t$, each term $(a)$, $(b)$ or $(c)$ is nonnegative.

If $t-s\geq 0$, then $(a)=(t+s)u_0(t+s)-(t-s)u_0(t-s)\geq 0$ by \eqref{u0_increases}. If $t-s<0$ then $(a)=(t+s)u_0(t+s)+(s-t)u_0(s-t)\geq 0$ by \eqref{u0_positive}.

Distinguishing between the cases $t-s>0$ and $t-s<0$, we get the following formula for $(b)$:
\begin{equation*}
 (b)=\int_{|t-s|}^{t+s} \rho u_1(\rho)\,d\rho,
\end{equation*}
which is positive by \eqref{u1_positive}.

Similarly,
\begin{equation*}
 (c)=\int_0^t\int_{t-\tau-s}^{t-\tau+s} F(t,\rho)\,d\rho=\int_0^t\int_{|t-\tau-s|}^{t-\tau+s}\rho f(\tau,\rho)\,d\rho\,d\tau,
\end{equation*}
and $(c)$ is positive by \eqref{f_positive}. Hence \eqref{u_positive}. 

If one of the inequalities in \eqref{u0_positive} or \eqref{u1_positive} is strict, or if the fonction $r\mapsto r u_0(r)$ is strictly increasing on $(0,L)$, the same proof shows that the inequality in \eqref{u_positive} is strict. 
\end{proof}
\begin{proof}[Proof of Lemma \ref{L:positivity1}]
Let $L>0$. We will show that \eqref{u_positiveP} holds for $t\in (0,T_+)$, $r\in (0,L-t)$. As $L$ is arbitrary, the proposition will follow.

As $u_{\restriction t=0}(r)\geq 0$, $\partial_t u_{\restriction t=0}(r)\geq 0$ and 
$\partial_t^2 u_{\restriction t=0}(r)>0$ for all $r\in (0,L]$, we get that there exists a small $t_0>0$ such that
\begin{equation}
\label{smallt}
 \forall t\in (0,t_0),\; \forall r\in (0,L],\quad u(t,r)>0\text{ and }\partial_tu(t,r)>0.
\end{equation} 
We argue by contradiction. Assume that
$$ \NNN=\Big\{t\in (0,T_+(u_0)),\; \exists r\in (0,L-t)\text{ s.t. } u(t,r)\leq 0\text{ or }\partial_tu(t,r)\leq 0\Big\}$$
is not empty. Let $t_1=\inf \NNN$. Then by \eqref{smallt}, $t_1>0$. Let $w=\partial_t u$. Then
\begin{align*}
 \partial_t^2 u-\Delta u=u^5,\quad u_{\restriction t=0}=u_0,\; \partial_t u_{\restriction t=0}=u_1\\
 \partial_t^2 w-\Delta w=5u^4w,\quad w_{\restriction t=0}=u_1,\; \partial_t w_{\restriction t=0}=\Delta u_0+u_0^5.
\end{align*}
By Lemma \eqref{L:positivity}, the definition of $\NNN$ and assumptions \eqref{u0_positiveP}, \eqref{u0_increasesP}, \eqref{u1_positiveP}, \eqref{u1_increasesP} and \eqref{derivee_2}
$$ \forall r\in (0,L-t_1),\quad u(t_1,r)>0,\quad \partial_t u(t_1,r)=w(t_1,r)>0.$$
By continuity of $u$ and $\partial_t u$, there exists $t_2>t_1$ such that
$$ \forall r\in (0,L-t_2),\quad u(t_2,r)>0,\quad \partial_t u(t_2,r)>0,$$
contradicting the fact that $t_1=\inf \NNN$. The proof is complete.
\end{proof}
In view of Proposition \ref{P:typeI},
to conclude the proof of Theorem \ref{T:blowupW+}, it is sufficient to show the following result:
\begin{prop}
\label{P:conclu}
The solution $W^+$ satisfies
$$ W^+(t,x)\geq 0,\quad \partial_t W^+(t,x)\geq 0$$
on its domain of definition.
\end{prop}
\begin{proof}
Consider the sequence $\left\{u_n\right\}_n$ of solutions of \eqref{CP} with initial data $(u_n^0,u^1_n)$, where
\begin{equation*}
 u_n^0=\left(1+\frac{1}{n}\right)W,\quad u_n^1=0.
\end{equation*} 
We will show that $u_n$ converges, up to a time translation, to $W^+$. We start by analyzing $u_n$.
\EMPH{Step 1. Properties of $u_n$}
We have 
$$ E[u_n]=\frac{1}{2}\left(1+\frac 1n\right)^2\int |\nabla W|^2-\frac 16 \left(1+\frac 1n\right)^6\int W^6.$$
Using the equality $\int |\nabla W|^2=\int W^6$, we get
\begin{equation}
\label{lower_energy}  
E[u_n]-E[W]=\left[\frac{1}{2}\left(1+\frac 1n\right)^2-\frac 16 \left(1+\frac 1n\right)^6-\frac{1}{3} \right] \int |\nabla W|^2<0.
\end{equation} 
Furthermore
$$ \int |\nabla u^0_n|^2>\int |\nabla W|^2.$$
As a consequence, by \cite{KeMe08}, $u_n$ blows-up in finite time in both time directions.

We next show
\begin{equation}
\label{un_positive}
 \forall t\in (0,T_+(u_n)),\; \forall r>0,\quad \partial_t u_n(t,r)>0,\quad u_n(t,r)\geq W.
\end{equation} 
Indeed, we can check that $u_n$ satisfies the assumptions of Lemma \ref{L:positivity1}. Obviously, $u^0_n\geq 0$ and $u^1_n\geq 0$. Furthermore,
$$ ru^0_n(r)=\left(1+\frac 1n\right) \frac{1}{\left(r^{-2}+\frac{1}{3}\right)^{\frac 12}}$$
is an increasing function of $r$, and $r u^1_n(r)=0$. Finally,
$$ \Delta u^0_n+\left(u^0_n\right)^5=-\left(1+\frac 1n\right)W^5+\left(1+\frac 1n\right)^5W^5>0,$$
and \eqref{un_positive} follows by Lemma \ref{L:positivity1}.

Fix a small parameter $\delta_0>0$. 
By Proposition \ref{P:typeI}, the blow-up of $u_n$ for positive time is of type I.  As a consequence, 
\begin{equation*}
t_n=\inf\left\{t>0, \int |\nabla(W-u_n(t))|^2+\int (\partial_t u_n(t))^2>\delta_0\right\}
\end{equation*} 
is well-defined (and, for large $n$, strictly positive). We have
\begin{gather}
\label{before_tn}
\forall t\in [0,t_n),\quad \left\|W-u_n(t)\right\|^2_{\hdot}+\left\|\partial_t u_n(t)\right\|^2_{L^2}<\delta_0\\
\label{at_tn}
\left\|W-u_n(t_n)\right\|^2_{\hdot}+\left\|\partial_t u_n(t_n)\right\|^2_{L^2}=\delta_0.
\end{gather}

\EMPH{Step 2. Convergence of $u_n(t_n)$}

After extraction of a subsequence, we can assume that the sequence $(u_n(t_n),\partial_t u_n(t_n))$ admits a profile decomposition $\left\{U^j_{\lin};t_{j,n},\lambda_{j,n}\right\}$. Reordering the profiles, we assume
\begin{equation}
\label{U1big}
 \left\|(U_0^1,U_1^1)\right\|_{\hdot\times L^2}=\max_j \left\|(U_0^j,U_1^j)\right\|_{\hdot\times L^2}.
\end{equation}
Then
\begin{equation}
\label{bound_below_U1}
\int \left|\nabla U_0^1\right|^2+\int\left(U_1^1\right)^2\geq \frac{2}{3}\int |\nabla W|^2.
\end{equation}
If not, by \cite{KeMe08}, all the profiles would scatter in both time directions contradicting the fact that $u_n$ blows up. By \eqref{at_tn} and \eqref{bound_below_U1} and 
the expansion of the $\hdot$ norm,
\begin{equation}
\label{bound_above_Uj}
\forall j\geq 2,\quad \int \left|\nabla U_0^j\right|^2+\int\left(U_1^j\right)^2\leq \frac{1}{3}\int |\nabla W|^2+\delta_0.
\end{equation} 
As a consequence, by \cite{KeMe08} again, all the nonlinear profiles $(U^j)_{j\geq 2}$ scatter in both time directions.
By \eqref{at_tn}, there exists a constant $C>0$ such that $C^{-1}\leq \lambda_{1,n}\leq C$ for all $n$. Rescaling $U^1$, we will assume
\begin{equation*}
 \lambda_{1,n}=1.
\end{equation*}
 We will prove that $U^1$ satisfies the following properties:
\begin{gather}
\label{energy_U1}
 E[U^1]=E[W]\\
\label{limite_un}
\lim_{n\to\infty}\left\|(u_n(t_n),\partial_t u_n(t_n))-(U^1_0,U^1_1)\right\|_{\hdot\times L^2}=0\\
\label{T-U1}
T^-(U^1)=-\infty,\\
\label{U1_close_W}
\forall t\leq 0,\quad \int \left|\nabla\left(U^1(t)-W\right)\right|^2+\int (\partial_t U^1(t))\leq \delta_0\\
\label{U1_not_W}
\int \left|\nabla\left(U^1(0)-W\right)\right|^2+\int (\partial_t U^1(0))=\delta_0\\
\label{U1_increases}
\forall t\in\left(-\infty, T_+(U^1)\right),\quad \partial_tU^1(t)\geq 0 \text{ and } U^1(t)\geq W\\
\label{U1_supercritical}
\forall t\in\left(-\infty,T_+(U^1)\right),\quad \int \left|\nabla U^1(t)\right|^2>\int \left|\nabla W\right|^2.
\end{gather}
We first show
\begin{equation}
\label{U1_close_W_bis}
 \forall t\in (T^-(U^1),0],\quad \int \left|\nabla\left(U^1(t)-W\right)\right|^2+\int (\partial_t U^1(t))^2\leq \delta_0.
\end{equation}
Indeed $t_n\to +\infty$. If not, by continuity of the flow, a subsequence of $\{(u_n(t_n),\partial_tu_n(t_n)\}_n$ would converge to $(W,0)$ in the energy space, a contradiction. Furthermore if $\eps>0$, then for $t\in \left[T^-(U^1)+\eps,0\right]$ we have
\begin{equation*}
\left\{\begin{aligned}
 u_n(t_n+t,x)-W&=U^1(t,x)-W+\sum_{j=2}^J \frac{1}{\lambda_{j,n}^{\frac{1}{2}}}U^j\left(\frac{t-t_{j,n}}{\lambda_{j,n}},\frac{x}{\lambda_{j,n}}\right)
+\tilde{w}_{n}^J(t,x),\\
\partial_t u_n(t_n+t)&=\sum_{j=1}^J \frac{1}{\lambda_{j,n}^{\frac{3}{2}}}\partial_t U^j\left(\frac{t-t_{j,n}}{\lambda_{j,n}},\frac{x}{\lambda_{j,n}}\right)+ \partial_t\tilde{w}_{n}^J(t,x),
\end{aligned}\right.
\end{equation*}
with
$$\lim_{J\to +\infty}\limsup_{n\to \infty} \lf\|\tilde{w}_n^J\rg\|_{L^{8}\lf(T^-(U^1)+\eps,0\rg)}=0.$$
Expanding the $\hdot$ norm, we get that if $t\in \lf(T^-(U^1)+\eps,0\rg)$, 
$$ \liminf_{n\to\infty} \left(\int |\nabla (u_n(t_n+t)-W)|^2+\int (\partial_t u_n(t_n+t))^2\right)\geq \int \left|\nabla (U^1(t)-W)\right|^2+\int \left(\partial_t U^1(t)\right)^2,$$
and \eqref{U1_close_W_bis} follows from \eqref{before_tn}. 

We get as a consequence that $T^-(U^1)=-\infty$. Indeed if $T^-(U^1)$ is finite \eqref{U1_close_W_bis} implies
$$ \forall t\in (T^-(U^1),0],\quad \int \left|\nabla U^1(t)\right|^2+\int (\partial_t U^1(t))^2\leq \int |\nabla W|^2+C\sqrt{\delta_0},
$$
and $U^1$ is (if $\delta_0$ is small enough)  a small blow-up type $II$ solution in the sense of \cite{DuKeMe11a}. But \eqref{U1_close_W_bis} contradicts the description of this type of solutions given in \cite{DuKeMe11a}. Thus $T^-(U^1)=-\infty$ (i.e. \eqref{T-U1} holds) and \eqref{U1_close_W_bis} implies \eqref{U1_close_W}.

By \eqref{bound_above_Uj}, the energies of all profiles $U^j$ and of the remainder $w_n^J$ are positive. 
By the Pythagorean expansion of the energy, $E[U^1]\leq E[W]$. By \cite{KeMe08}, $E[U^1]<E[W]$ is excluded: in this case the solution blows up in finite time in both time directions (excluded because $T^-(U^1)=-\infty$) or scatters in both time directions (excluded by \eqref{U1_close_W}). Hence \eqref{energy_U1} holds.

Again by the Pythagorean expansion of the energy and  using that 
$$ \lim_{n\to \infty}E[u_n]=E[W],$$
we get
$$\forall j\geq 2,\quad U^j=0\text{ and }\limsup_{n\to \infty}\|\nabla w^J_{0,n}\|^2_{L^2}+\|w^J_{1,n}\|^2_{L^2}=0.$$
Similarly,
$$
\lim_{J\to \infty} \limsup_{n\to \infty}\|\nabla w^J_{0,n}\|^2_{L^2}+\|w^J_{1,n}\|^2_{L^2}=0.$$  
In other words \eqref{limite_un} holds.

The equality \eqref{U1_not_W} is then a direct consequence of \eqref{at_tn}. The inequalities \eqref{U1_increases} follow from \eqref{un_positive} and \eqref{limite_un}.

Finally as $\int |\nabla u_n(t_n)|^2>\int |\nabla W|^2$ for all $n$, we have $\int |\nabla U^1(0)|^2\geq \int |\nabla W|^2$. By the variational characterisation of $W$ given by Aubin \cite{Aubin76} and Talenti \cite{Talenti76}, the equalities $E[U^1]=E[W]$ and $\int |\nabla U^1(0)|^2=\int |\nabla W|^2$ would imply $U^1=W$, which is excluded by \eqref{U1_not_W}. Hence $\int |\nabla U^1(0)|^2>\int |\nabla W|^2$. As $E[U^1]=E[W]$, we know from \cite{DuMe08} that the sign of $\int |\nabla U^1(t)|^2-\int |\nabla W|^2$ does not change. Hence \eqref{U1_supercritical}.

\EMPH{Step 3: conclusion}
By Theorem \ref{T:rigidityW+}, modulo symmetries, $U^1=W^+$. In view of Proposition \ref{P:typeI} the blow-up of $W^+$ is of type I, which concludes the proof of Theorem \ref{T:blowupW+}.
\end{proof}

\begin{remark}
As a consequence of the proof above and of \cite{KeMe08} we obtain a complete description of the family of solutions $U_{c}$, $c>0$, with initial data $(cW,0)$. If $c\neq 1$, an  explicit computation shows that $E[U_c]<E[W]$. Hence, by \cite{KeMe08}, if $0<c<1$, $U_c$ scatters in both time directions. If $c=1$, $U_c$ is the stationary solution $W$. If $c>1$, $U_c$ blows-up in both time directions and, by Proposition \ref{P:typeI} and Lemma \ref{L:positivity1}, the blow-up is of type I. 
\end{remark}

\appendix
\section{Pseudo-orthogonality of profiles}
\label{A:quasi_ortho}
In this appendix we show that
\begin{equation}
\label{quasi_ortho1}
\lim_{n\to+\infty}\int_{S_n} \nabla U_{\lin}^j(\sigma_n,y)\cdot \nabla w_{n}^J(\sigma_n,y)\,dy=0, 
\end{equation} 
where (using the notations of Subsection \ref{SS:self_simF}), $\sigma_n\in I_n^+$, $2\leq j\leq J$ and $S_n=\{\sigma_n+r_0-\eps_0\leq |y|\leq \sigma_n+r_0\}$.

The proofs of the analog of \eqref{quasi_ortho1} for the time derivatives and of the fact that the other integrals in \eqref{quasi_ortho} tend to $0$ are similar, but easier, and we leave them to the reader. We will use the following result, which follows immediately from the dominated convergence theorem:
\begin{lemma}
\label{L:weak_CV}
 Let $f_n\rightharpoonup 0$ in $L^2(\RR)$, and consider $I_n$ a sequence of interval such that the sequence of characteristic functions $\{\chi_{I_n}\}_n$ tends pointwise to $\chi_I$ (where $I$ is an interval). Then
$$ \chi_{I_n} f_n \rightharpoonup 0\text{ in }L^2(\RR)\text{ as }n\to\infty.$$
\end{lemma}
Note that in the Lemma, $I$ might be empty or reduced to a single point.

Using that $(\partial_t^2-\partial_r^2)(r\,w_n^J)=0$, we get that
$$ w_n^J(\tau,r)=\frac{1}{r}\left[G_n^J(\tau+r)-G_n^J(\tau-r)\right],$$
where $G_n^J(s)=\frac{s}{2}w_{0,n}^J(|s|)+\frac{1}{2}\int_0^{|s|} \sigma w_{1,n}^J(\sigma)\,d\sigma.$
Similarly $U^j_{\lin}(\tau,r)=\frac{1}{2}\left[F^j(\tau+r)-F^j(\tau-r)\right]$, where $F^j$ is defined as $G_n^J$ from the initial data of $U^j_{\lin}$.

\EMPH{Step 1}
We first check:
\begin{multline}
\label{reduced_quasi_ortho}
 \int_{\sigma_n+r_0-\eps_0}^{\sigma_n+r_0}\partial_r U^j_{\lin,n}(\sigma_n,r)\partial_r w_n^J(\sigma_n,r)r^2dr
\\=\frac{1}{4}
 \int_{\sigma_n+r_0-\eps_0}^{\sigma_n+r_0}\frac{1}{\lambda_{j,n}^{1/2}}\partial_s F^j\left( \frac{\sigma_n-\tau_{j,n}-r}{\lambda_{j,n}} \right)\partial_s G_n^J(\sigma_n-r)\,dr+o_n(1).
\end{multline} 
Indeed, by integration by parts,
\begin{multline*}
  \int_{\sigma_n+r_0-\eps_0}^{\sigma_n+r_0}\partial_r U^j_{\lin,n}(\sigma_n,r)\partial_r w_n^J(\sigma_n,r)r^2dr
\\= \int_{\sigma_n+r_0-\eps_0}^{\sigma_n+r_0}\partial_r \left(rU^j_{\lin,n}(\sigma_n,r)\right)\partial_r (rw_n^J(\sigma_n,r))dr+\left[rU^j_{\lin,n}(\sigma_n,r)w_n^J(\sigma_n,r)\right]^{\sigma_n+r_0}_{\sigma_n+r_0-\eps_0},
\end{multline*} 
and the boundary terms tend to $0$ by standard arguments, using that $w_n^J(\sigma_n)$ tends weakly to $0$ in $\hdot$.
Thus \eqref{reduced_quasi_ortho} will follow from:
\begin{equation}
 \label{2_zero_limits}
\lim_{n\to\infty} \int_{\sigma_n+r_0-\eps_0}^{\sigma_n+r_0}\left(\partial_s F^j\left(  \frac{\sigma_n+r-\tau_{j,n}}{\lambda_{j,n}}\right)\right)^2\frac{dr}{\lambda_{j,n}}=\lim_{n\to\infty} 
\int_{\sigma_n+r_0-\eps_0}^{\sigma_n+r_0} \left( \partial_s G_n^J(\sigma_n+r) \right)^2\,dr=0.
\end{equation} 
It is easy to show that the first limit is $0$, using the fact that $\partial_sF^j\in L^2(\RR)$, that $(2\sigma_n+r_0-\eps_0)/\lambda_{j,n})\to \infty$ as $n\to \infty$, and the change of variable $s=\frac{\sigma_n+r-\tau_{j,n}}{\lambda_{j,n}}$. To get that the second limit is also $0$, we write
\begin{multline*}
 \int_{\sigma_n+r_0-\eps_0}^{\sigma_n+r_0} \left( \partial_s G_n^J(\sigma_n+r) \right)^2\,dr\\
\leq \int_{\sigma_n+r_0-\eps_0}^{\sigma_n+r_0} \left(\frac{\partial}{\partial r}\left((\sigma_n+r)w_{0,n}^J(\sigma_n+r)\right)\right)^2\,dr+\int_{\sigma_n+r_0-\eps_0}^{\sigma_n+r_0} \left((\sigma_n+r)w_{1,n}^J(\sigma_n+r)\right)^2\,dr\\
=\int_{2\sigma_n+r_0-\eps_0}^{2\sigma_n+r_0} \left(\frac{\partial}{\partial s}\left(s\,w_{0,n}^J(s)\right)\right)^2\,ds+\int_{2\sigma_n+r_0-\eps_0}^{2\sigma_n+r_0} \left(s\,w_{1,n}^J(s)\right)^2\,ds.
\end{multline*}
This last line goes to $0$ as $n$ goes to infinity,
because, as $\sigma_n\in I_n^+$, $2\sigma_n+r_0-\eps_0\geq 1$, and 
$$ \lim_{n\to+\infty} \int_{r\geq 1} (\partial_rw_{0,n}^J(r))^2+(w_{1,n}^J(r))^2r^2\,dr=0.$$
This completes Step 1.

\EMPH{Step 2: end of the proof}
We have
\begin{multline}
\label{theend}
\int_{\sigma_n+r_0-\eps_0}^{\sigma_n+r_0} \frac{1}{\lambda_{j,n}^{1/2}} \partial_sF^j\left( \frac{\sigma_n-\tau_{j,n}-r}{\lambda_{j,n}} \right)\,\partial_s G_n^J(\sigma_n-r)\,dr\\
=\int_{\frac{r_0-\eps_0+\tau_{j,n}}{\lambda_{j,n}}}^\frac{r_0+\tau_{j,n}}{\lambda_{j,n}} \partial_s F^j(-\theta)\,\lambda_{j,n}^{1/2} \partial_sG_n^J(\tau_{j,n}-\lambda_{j,n}\theta)\,d\theta.
\end{multline}
Using that $\partial_sG_n^J(\tau-r)=\partial_r(rw_n^J(\tau,r))-\partial_{\tau}(rw_n^J(\tau,r))$, we get by \eqref{weak_CV_wJ} that $\lambda_{j,n}^{1/2} \partial_sG_n^J(\tau_{j,n}-\lambda_{j,n}\theta)$ tends weakly to $0$ in $L^2(\RR,d\theta)$, which implies by Lemma \ref{L:weak_CV} that the last integral in \eqref{theend} goes to $0$, concluding the proof.

\bibliographystyle{acm}
\bibliography{toto}
\end{document}


\maketitle

Let $\big\{(u_{0,n},u_{1,n})\big\}_n$ be a sequence of radial functions, bounded in $(\hdot\times L^2)(\RR^3)$, and assume that it has a profile decomposition with profiles  $\left(U^j_{\lin},\left\{\lambda_{j,n},t_{j,n}\right\}_n\right)_{j\geq 1}$, as defined in \cite{BaGe99}. In the arXiv preprint \cite{DuKeMe12bX} and its published version \cite{DuKeMe12b}, it is claimed that the following Pythagorean expansions hold:
\begin{gather}
\label{pythagore1a} 
\left\|u_{0,n}\rg\|_{\hdot}^2=\sum_{j=1}^J \left\|U^j_{\lin}\left(\frac{-t_{j,n}} {\lambda_{j,n}}\right)\rg\|_{\hdot}^2+\left\|w_{0,n}^J\rg\|_{\hdot}^2+o_n(1)\\
\label{pythagore1b} 
\lf\|u_{1,n}\right\|^2_{L^2}=\sum_{j=1}^J \lf\|\partial_t U^j_{\lin}\left(\frac{-t_{j,n}} {\lambda_{j,n}}\right)\right\|^2_{L^2}+\lf\|w_{1,n}^J\right\|^2_{L^2}+o_n(1).
\end{gather}
(see (2.6) and (2.7) of \cite{DuKeMe12bX} or \cite{DuKeMe12b}), where $o_n(1)$ denotes any numerical sequence going to $0$ as $n$ goes to infinity. Both properties are false. Let us give a counter-example to the Pythagorean expansion \eqref{pythagore1b} (a very close counter-example to \eqref{pythagore1a} can be given). Let $U_{\lin}$ be any nonzero solution for the linear wave equation 
\begin{equation}
\label{lin_wave}
 (\partial_t^2-\Delta)U_{\lin}=0,\quad (U_{\lin},\partial_tU_{\lin})_{\restriction t=0}\in \hdot\times L^2.
\end{equation} 
Let $\{t_n\}_n$ be a sequence of positive numbers such that $t_n\to+\infty$ as $n\to \infty$.
Then the sequence $\big\{(u_{0,n},u_{1,n})\big\}_n$ defined by $u_{0,n}=2 U_{\lin}(t_{n},x)$, $u_{1,n}(x)=0$, admits a profile decomposition 
with two nonzero profiles $U^1_{\lin}(t)=U_{\lin}(t)$, $U^2_{\lin}(t)=U_{\lin}(-t)$, parameters $t_{1,n}=t_n$, $t_{2,n}=-t_n$ and $\lambda_{1,n}=\lambda_{2,n}=1$. This profile decomposition obviously does not satisfy \eqref{pythagore1b}.

We note however that if $t_{j,n}=0$ for all $j,n$, then \eqref{pythagore1a} and \eqref{pythagore1b} hold. Furthermore, the Pythagorean expansion of the square of the $\hdot\times L^2$ norm (that is the sum of \eqref{pythagore1a} and \eqref{pythagore1b}) is true for any profile decomposition, as proved in \cite{BaGe99}.

The expansion \eqref{pythagore1b} is used in \cite{DuKeMe12bX} and \cite{DuKeMe12b} to study the asymptotics of radial solutions $u$ of 
\begin{equation}
 \label{CP}
 \left\{\begin{aligned}
         \partial_t^2u -\Delta u&=u^5,\quad t\in I,\; x\in \RR^3\\
         (u,\partial_tu)_{\restriction t=0}&=\vu_0\in \dot H^1\times L^2,
        \end{aligned}\right.
\end{equation} 
such that $T_+(u)=+\infty$ and 
$$\sup_{t\geq 0} \|\vec{u}(t)\|_{\hdot\times L^2}<\infty.$$
More precisely, it is used in the proof of Corollary 4.2 of \cite{DuKeMe12bX} and \cite{DuKeMe12b}. Consider, as in Corollary 4.2, a radial solution $u$ of \eqref{CP} satisfying the above assumption and $v_{\lin}$ its linear asymptotic component as $t\to+\infty$ (see \cite{DuKeMe12bX} for the precise definition). The proof of the corollary uses the fact that there exists a sequence $s_n\to+\infty$ such that
\begin{gather}
\label{special_sequence1}
 \lim \int_{\RR^3} (\partial_tu(s_n,x)-\partial_t v_{\lin}(s_n,x))^2\,dx=0\\
\label{special_sequence2}
\lim_{n\to+\infty} \sup_{\lambda>0} \frac{1}{\lambda}\int_{s_n}^{s_n+\lambda} \int_{\RR^3} (\partial_tu(t,x)-\partial_t v_{\lin}(t,x))^2\,dx\,dt=0.
\end{gather}
(see Lemma 4.4).
Assuming\footnote{if $u$ is a function of space and time, we denote $\vec{u}=(u,\partial_tu)$} that $\vec{u}(s_n)-\vec{v}_{\lin}(s_n)$ has a profile decomposition with profiles $(U^j_{\lin})_{j \geq 1}$, and parameters $\left(\left\{t_{j,n},\lambda_{j,n}\right\}_n\right)_{j\geq 1}$, we deduce in \cite{DuKeMe12bX}, from \eqref{special_sequence1}, \eqref{pythagore1b} and the asymptotic equipartition of the energy for solutions of the linear wave equation that for all $j$,
\begin{equation}
 \label{limsupfinite}
 \limsup_{n\to\infty}\left|\frac{t_{j,n}}{\lambda_{j,n}}\right|<\infty.
\end{equation}

Since \eqref{pythagore1b} is false in general, there is a gap in the proof, that is filled by the following lemma, proved in this erratum:
\begin{lemma}
 \label{L:1}
Let $u$ and $v_{\lin}$ be as in Corollary 4.2 of \cite{DuKeMe12bX} or \cite{DuKeMe12b}, and $s_n\to+\infty$ such that \eqref{special_sequence1}, \eqref{special_sequence2} hold. Let  $\left(U^j_{\lin},\left\{t_{j,n},\lambda_{j,n}\right\}_n\right)_{j \geq 1}$ be a profile decomposition of $\vec{u}(s_n)-\vec{v}_{\lin}(s_n)$. Then \eqref{limsupfinite} holds for all $j$.
\end{lemma}

If $(f,g)\in (\dot{H}^1\times L^2)(\RR^3)$, we denote 
$$S_{\lin}(t)(f,g)=\cos(t\sqrt{-\Delta})f+\frac{\sin(t\sqrt{-\Delta})}{\sqrt{-\Delta}}g$$
the solution of the linear wave equation with data $(f,g)$, and 
$$\vec{S}_{\lin}(t)(f,g)=\left(S_{\lin}(t)(f,g),\frac{\partial}{\partial t} S_{\lin}(t)(f,g)\right).$$ We note that by conservation of the energy
\begin{multline*}
\forall t\in \RR, \quad\forall \Big((f,g),(\varphi,\psi)\Big)\in \left(\hdot\times L^2\right)^2,\\
\left(\vS_{\lin}(t)(f,g),(\varphi,\psi)\right)_{\hdot\times L^2}=\left((f,g),\vS_{\lin}(-t)(\varphi,\psi)\right)_{\hdot\times L^2}. 
\end{multline*}
The proof of Lemma \ref{L:1} relies on the following claim:
\begin{claim}
\label{C:2}
Let $\big\{(f_n,g_n)\big\}_n$ be a sequence of radial functions, bounded in $(\hdot\times L^2)(\RR^3)$, $\{\alpha_n\}_n$ a sequence of positive numbers such that 
 \begin{equation}
  \label{E3}
  \lim_{n\to\infty} \|g_n\|_{L^2(\{|x|>\alpha_n\})}=0.
 \end{equation} 
 Let $\{t_n\}_n\in \RR^{\NN}$ such that $\lim_n|t_n|/\alpha_n=+\infty$. Assume that in $\hdot\times L^2$:
 $$w-\lim_{n\to\infty} \vec{S}_{\lin}(-t_n)(f_n,g_n)=(U_0,U_1),$$
 then
 $$w-\lim_{n\to\infty} \vec{S}_{\lin}(t_n)(f_n,g_n)=(U_0,-U_1).$$
 \end{claim}
\begin{proof}
 If $(\varphi,\psi)\in C_0^{\infty}(\RR^3)^2$, then
 \begin{align*}
 \left(\vec{S}_{\lin}(-t_n)(f_n,g_n),(\varphi,\psi)\right)_{\hdot\times L^2}&=\left((f_n,g_n),\vec{S}_{\lin}(t_n)(\varphi,\psi)\right)_{\hdot\times L^2}\\
 &=\left((f_{n},-g_{n}),\vec{S}_{\lin}(t_n)(\varphi,\psi)\right)_{\hdot\times L^2}+o_n(1)\text{ as }n\to\infty. 
 \end{align*}
To obtain the last equality, we used 
 $$\lim_{n\to\infty}\left((0,g_{n}),\vec{S}_{\lin}(t_n)(\varphi,\psi)\right)_{\hdot\times L^2}=0,$$
which follows from 
\eqref{E3} and Lemma 4.1 in \cite{DuKeMe11a} (or the fact that, by the strong Huygens principle, the support of $\vS_{\lin}(t_n)(\varphi,\psi)$ is included in $\{|x|\geq t_n-C\}$), and the assumption that $t_n/\alpha_n$ goes to infinity.
As a consequence, using that for any $(u_0,u_1)$, $\vec{S}_{\lin}(-t)(u_0,-u_1)=\left(S_{\lin}(t)(u_0,u_1),-\partial_t S_{\lin}(t)(u_0,u_1)\right)$,
\begin{align*}
 \left(\vec{S}_{\lin}(-t_n)(f_n,g_n),(\varphi,\psi)\right)_{\hdot\times L^2}&=
\left(\vec{S}_{\lin}(-t_n)(f_{n},-g_{n}),(\varphi,\psi)\right)_{\hdot\times L^2}+o_n(1)\\
&=\left(\vec{S}_{\lin}(t_n)(f_n,g_n),(\varphi,-\psi)\right)_{\hdot\times L^2}+o_n(1).
\end{align*}
as $n\to\infty$. This concludes the proof.
\end{proof}
As a corollary of Claim \ref{C:2} we obtain:
\begin{claim}
 \label{C:3}
 Let $\left\{(u_{0,n},u_{1,n})\right\}_n$ be a sequence of radial functions, bounded in $\hdot\times L^2$, such that 
 \begin{equation}
  \label{E5}
  \lim_{n\to\infty} \|u_{1,n}\|_{L^2}=0.
 \end{equation} 
 Then, extracting subsequences, $\big\{(u_{0,n},u_{1_n})\big\}_n$ has a profile decomposition $\left(U^j_{\lin},\left\{\lambda_{j,n},t_{j,n}\right\}_n\right)_{j\geq 1}$ such that for all $j\geq 1$, one of the following holds:
 \begin{equation}
  \label{E6}
  \forall n,\quad t_{j,n}=0\quad \text{and}\quad\partial_tU^j_{\lin}(0)=0
 \end{equation}
or 
\begin{gather}
 \label{E7}
 \lim_{n\to\infty}\frac{-t_{j,n}}{\lambda_{j_n}}\in \{\pm\infty\} \text{ and }\\
\label{E8}
 \exists k\neq j \text{ s.t. }U^k_{\lin}(t)=U^j_{\lin}(-t)\quad \text{and}\quad \forall n,\; t_{j,n}=-t_{k,n},\; \lambda_{j,n}=\lambda_{k,n}.
\end{gather}
\end{claim}
\begin{proof}
We fix $j\geq 1$. We have:
\begin{equation}
 \label{E10}
 \vec{U}_{\lin}^j(0)=w-\lim_{n\to\infty} \vec{S}_{\lin}\left( \frac{t_{j,n}}{\lambda_{j,n}} \right)\left(\lambda_{j,n}^{1/2}u_{0,n}\left( \lambda_{j,n}\cdot \right),\lambda_{j,n}^{3/2}u_{1,n}\left( \lambda_{j,n}\cdot \right)\right).
\end{equation} 
As usual, we can assume that $t_{j,n}=0$ for all $n$ or that \eqref{E7} holds. In the first case, \eqref{E10} proves that $\partial_tU_{\lin}^j(0)=0$. 
In the second case, by Claim \ref{C:2}, we have 
\begin{equation}
\label{E10'}
  \left(U_{\lin}^j(0),-\partial_tU_{\lin}^j(0)  \right)=w-\lim_{n\to\infty} \vec{S}_{\lin}\left( -\frac{t_{j,n}}{\lambda_{j,n}} \right)\left(\lambda_{j,n}^{1/2}u_{0,n}\left( \lambda_{j,n}\cdot \right),\lambda_{j,n}^{3/2}u_{1,n}\left( \lambda_{j,n}\cdot \right)\right).
\end{equation} 
Hence there exists $k\neq j$ such that (rescaling and time-translating $U_{\lin}^k$ and adjusting the parameters $\lambda_{k,n}$ and $t_{k,n}$ if necessary):
$$ U_{\lin}^k(0)=U_{\lin}^j(0), \quad \partial_tU^k_{\lin}(0)=-\partial_tU^j_{\lin}(0)\quad \text{and}\quad\forall n,\; t_{k,n}=-t_{j,n},\; \lambda_{k,n}=\lambda_{j,n}.$$
The proof is complete.
\end{proof}
We are now in position to prove Lemma \ref{L:1}. Let $u$ and $\{s_n\}$ be as in the lemma. We argue by contradiction, assuming that there exists $j\geq 1$ such that $U_{\lin}^j\neq 0$ and $\lim_{n\to \infty}-\frac{t_{j,n}}{\lambda_{j,n}}=+\infty$ (by Claim \ref{C:3} there is no loss of generality). By Claim \ref{C:3} we can assume, reordering the profiles,
$$ U_{\lin}^{j+1}(t)=U_{\lin}^j(-t),\quad \text{and}\quad \forall n,\; t_{j+1,n}=-t_{j,n},\;\lambda_{j+1,n}=\lambda_{j,n}.$$

\emph{Step 1. Convergence to $0$ of the time derivative for a well-chosen sequence of times.} We prove that (after extraction of a subsequence) there exist $\tau_0>0$ such that 
$$\lim_{n\to\infty} \left\|\partial_tu(s_n+\lambda_{j,n}\tau_0)-\partial_tv_{\lin}(s_n+\lambda_{j,n}\tau_0)\right\|_{L^2}=0.$$
Indeed, let 
$$\eps_n=\sup_{\lambda>0} \frac{1}{\lambda}\int_{s_n}^{s_n+\lambda} \left\|\partial_tu(t)-\partial_t v_{\lin}(t)\right\|_{L^2}^2.$$
By our assumption, $\eps_n$ goes to $0$ as $n$ goes to infinity. Let
$$E_n=\Big\{\tau\in [0,1]\text{ s.t. }\left\|\partial_tu(s_n+\lambda_{j,n}\tau)-\partial_tv_{\lin}(s_n+\lambda_{j,n}\tau)\right\|_{L^2}\geq \eps_n^{1/4}\Big\}.$$
Then
\begin{align*}
 \eps_n&\geq \frac{1}{\lambda_{j,n}}\int_{s_n}^{s_n+\lambda_{j,n}}\left\|\partial_tu(t)-\partial_tv_{\lin}(t)\right\|_{L^2}^2\,dt\\
&=\int_0^1 \left\|\partial_t u(s_n+\lambda_{j,n}s)-\partial_tv_{\lin}(s_n+\lambda_{j,n}s)\right\|^2_{L^2}\,ds\\
&\geq |E_n|\eps_n^{1/2},
 \end{align*}
where $|.|$ is the Lebesgue measure on $\RR$. Hence 
$$\lim_{n\to\infty}|E_n|=0.$$
Extracting subsequences and re-numbering, we can assume that for all $n$, $|E_n|\leq \frac{1}{2^{n+2}}$. Thus 
$$\left|\bigcup_{n\geq 0} E_n\right|\leq \frac{1}{2}<1.$$
Choosing any $\tau_0$ in $[0,1]\setminus \cup_{n}E_n$ such that $0<\tau_0$ we obtain the conclusion of this step.

\emph{Step 2. Approximation and finite speed of propagation.}
We fix a small parameter $\delta_0>0$ such that any solution of the nonlinear wave equation \eqref{CP} with $\|(u_0,u_1)\|_{\hdot\times L^2}<\delta_0$ is globally defined and scatters. Let $\JJJ$ be the set of indexes $k$ such that 
$$\left\|\vec{U}^k_{\lin}(0)\right\|_{\hdot\times L^2}\geq \delta_0$$
and
$$\exists C_k>0,\;\forall n,\quad \max\left(\lambda_{k,n},|t_{k,n}|\right)\leq C_k\lambda_{j,n}.$$
Because of the first condition and the Pythagorean expansion of the square of the $\hdot\times L^2$ norm, $\JJJ$ is finite. Let $\varphi\in C_0^{\infty}(\RR^3)$, radial, such that $\varphi(r)=1$ if $r\geq 2$, and $\varphi(r)=0$ if $r\leq 1$. Let $\tlambda_n=\sqrt{\lambda_{j,n}|t_{j,n}|}$, so that 
\begin{equation}
\label{cond_tlambdan}
\lambda_{j,n}\ll \tlambda_n\ll |t_{j,n}|.
\end{equation} 
Let $J_0>\max\left(\JJJ\cup \{j+1\}\right)$ and 
$$(\tu_{0,n},\tu_{1,n})=\vec{v}_{\lin}(s_n)+\varphi\left( \frac{x}{\tlambda_n} \right)\sum_{k\in \JJJ} \vU^k_{\lin,n}(0)+\sum_{\substack{k=1,\ldots,J_0\\k\notin \JJJ}} \vU_{\lin,n}^k(0)+\vec{w}_n^{J_0}(0).$$
By \eqref{cond_tlambdan} and the definition of $\JJJ$, we have, using again Lemma 4.1 in \cite{DuKeMe11a},  
$$\lim_{n\to\infty}\left\|\varphi\left( \frac{x}{\tlambda_n} \right)\sum_{k\in \JJJ} \vU^k_{\lin,n}(0)\right\|_{\hdot\times L^2}=0.$$
Furthermore, 
$$|x|>2\tlambda_n\Longrightarrow (\tu_{0,n},\tu_{1,n})=\vec{u}(s_n).$$
Letting $\tu_n$ be the solution of the nonlinear wave equation \eqref{CP} with initial data $(\tu_{0,n},\tu_{1,n})$ at $t=0$, we obtain by approximation (Proposition 2.1 of \cite{DuKeMe12bX}) that for large $J$, $0<t<\lambda_{j,n}$,
\begin{equation}
 \label{E11}
 \tu_n(t)=v_{\lin}(s_n+t)+\sum_{\substack{k=1\ldots J\\ k\notin \JJJ}} U_n^k(t)+w_n^J(t)+r_n^J(t),
\end{equation} 
where as usual, the $U_n^k$ are the modulated nonlinear profiles (see \cite[(2.11)]{DuKeMe12bX}), $w_n^J(t)=S_L(t)(\vec w_n^J(0))$ and $r_n^J$ is a small remainder, that satisfies 
$$\lim_{J\to\infty}  \limsup_{n\to\infty}\left(\sup_{t\in [0,\lambda_{j,n}]} \|\vec{r}_n^J(t)\|_{\hdot\times L^2}\right)=0.$$
Using \eqref{E11} and its time derivative at $t=\lambda_{j,n}\tau_0$, we see that it yields a profile decomposition $\left(\tU^k_{\lin},\{\tlambda_{k,n},\tilde{t}_{k,n}\}_n\right)_{k\geq 1}$ of $\vec{\tu}_n(\lambda_{j,n}\tau_0)-\vec{v}_{\lin}(s_n+\lambda_{j,n}\tau_0)$. The parameters $\tlambda_{k,n}$, $\tilde{t}_{k,n}$ are given by:
$$ \tlambda_{k,n}=\lambda_{k,n},\quad \tilde{t}_{k,n}=t_{k,n}-\lambda_{j,n}\tau_0.$$
The profiles corresponding to the indices $j$ and $j+1$ are still $\tU_{\lin}^j=U_{\lin}^j$, $\tU^{j+1}_{\lin}=U_{\lin}^{j+1}=U_{\lin}^j(-t)$.
Furthermore, by finite speed of propagation,
$$\vec{\tu}_n(\tau_0\lambda_{j,n},x)=\vec{u}(s_n+\tau_0\lambda_{j,n},x),\quad |x|\geq \tau_0\lambda_{j,n}+2\tilde{\lambda}_n.$$
We next use Claim \ref{C:2}, with $\alpha_n=2\tlambda_n/\lambda_{j,n}+\tau_0$, $t_n=t_{j,n}/\lambda_{j,n}$ and 
$$(f_n,g_n)=\left( \lambda_{j,n}^{1/2}\tu(\tau_0 \lambda_{j,n},\lambda_{j,n}x\right),\lambda_{j,n}^{3/2}\partial_t \tu_n\left(\tau_0\lambda_{j,n},\lambda_{j,n}x)\right).$$
Note that since by our choice of $\tlambda_n$, $|t_{j,n}|/\tlambda_n$ goes to $\infty$ as $n$ goes to infinity, the assumption $\lim_n|t_n|/\alpha_n=\infty$ is satisfied.
We have 
$$w-\lim_{n\to\infty} \vec{S}_{\lin}\left( \frac{t_{j,n}}{\lambda_{j,n}} \right)(f_n,g_n)=w-\lim_{n\to\infty} \vec{S}_{\lin}(\tau_0)\vec{S}_{\lin}\left( \frac{\tilde{t}_{j,n}}{\tlambda_{j,n}} \right)(f_n,g_n)=\left(U^{j}_{\lin}(\tau_0),\partial_tU^j_{\lin}(\tau_0)\right).$$
and
\begin{multline*}
w-\lim_{n\to\infty} \vec{S}_{\lin}^j\left( \frac{-t_{j,n}}{\lambda_{j,n}} \right)(f_n,g_n)=w-\lim_{n\to\infty} \vec{S}_{\lin}(\tau_0)\vec{S}_{\lin}\left( \frac{\tilde{t}_{j+1,n}}{\tlambda_{j+1,n}} \right)(f_n,g_n)\\
=\left(\tU^{j+1}_{\lin}(\tau_0),\partial_t\tU^{j+1}_{\lin}(\tau_0)\right)=\left( U_{\lin}^j(-\tau_0),-\partial_tU_{\lin}^j(-\tau_0) \right)
\end{multline*}
(recall that we have re-ordered the profiles so that \eqref{E8} is satisfied with $k=j+1$),
and thus, by Claim \ref{C:2}, $U_{\lin}^j(\tau_0)=U_{\lin}^j(-\tau_0)$ and $\partial_tU_{\lin}^{j}(\tau_0)=\partial_tU_{\lin}^j(-\tau_0)$. 
This proves
$$U_{\lin}^j(t)=U_{\lin}^j(t+2\tau_0)$$
and,
since $U_{\lin}^j$ is a finite-energy solution to the linear wave equation, that $U_{\lin}^j$ is identically $0$, contradicting our assumption. The proof is complete.
\qed

We note also that the false Pythagorean expansion \eqref{pythagore1b} is used in a similar context in the finite time blow-up case (see \cite[Subsection 5.2]{DuKeMe11a} and  \cite{DuKeMe12bX,DuKeMe12b}). A very small modification of the argument above is sufficient to treat this case also.
\bibliographystyle{plain}
\bibliography{toto}